\documentclass[UTF-8,reqno]{amsart}
\usepackage{enumerate}
\usepackage{mhequ}
\usepackage[margin=1.2in]{geometry}
\usepackage{amssymb,url,color, booktabs,nccmath}
\usepackage{mathrsfs}
\usepackage{enumitem}
\usepackage{graphicx}
\usepackage{float}
\usepackage{tikz}
\usetikzlibrary{shapes,snakes}
\usetikzlibrary{calc}
\usetikzlibrary{decorations.shapes}
\usepackage{color}
\usepackage[colorlinks=true]{hyperref}
\hypersetup{
    linkcolor=blue,          % color of internal links
    linktoc=page,            % in the table of contents, link only page numbers
    citecolor=red,        % color of links to bibliography
    filecolor=blue,      % color of file links
    urlcolor=cyan
}

\definecolor{darkergreen}{rgb}{0.0, 0.5, 0.0}

\numberwithin{equation}{section}
\def\theequation{\arabic{section}.\arabic{equation}}
\newcommand{\be}{\begin{eqnarray}}
\newcommand{\ee}{\end{eqnarray}}
\newcommand{\ce}{\begin{eqnarray*}}
\newcommand{\de}{\end{eqnarray*}}
\newtheorem{theorem}{Theorem}[section]
\newtheorem{lemma}[theorem]{Lemma}
\newtheorem{remark}[theorem]{Remark}
\newtheorem{definition}[theorem]{Definition}
\newtheorem{proposition}[theorem]{Proposition}
\newtheorem{Examples}[theorem]{Example}
\newtheorem{corollary}[theorem]{Corollary}

\newenvironment{nouppercase}{%
  \renewcommand{\uppercasenonmath}[1]{}}{}

\def\a{\alpha}

\def\d{\mathop{}\!\mathrm{d}}

\def\[{{\Big[}}
\def\]{{\Big]}}
\def\<{{\langle}}
\def\>{{\rangle}}
\def\({{\Big(}}
\def\){{\Big)}}

\def\bx{{\mathbf{x}}}

\def\min{{\mathord{{\rm min}}}}

\def\={&\!\!=\!\!&}

\def\1{{\mathbf{1}}}

\def\geq{\geqslant}
\def\leq{\leqslant}

\def\a{\alpha}

\def\d{\mathop{}\!\mathrm{d}}

\def\[{{\Big[}}
\def\]{{\Big]}}
\def\<{{\langle}}
\def\>{{\rangle}}
\def\({{\Big(}}
\def\){{\Big)}}

\def\bx{{\mathbf{x}}}

\def\min{{\mathord{{\rm min}}}}

\def\={&\!\!=\!\!&}
\def\bt{\begin{theorem}}
\def\et{\end{theorem}}
\def\bl{\begin{lemma}}
\def\el{\end{lemma}}
\def\br{\begin{remark}}
\def\er{\end{remark}}
\def\bx{\begin{Examples}}
\def\ex{\end{Examples}}
\def\bd{\begin{definition}}
\def\ed{\end{definition}}
\def\bp{\begin{proposition}}
\def\ep{\end{proposition}}
\def\bc{\begin{corollary}}
\def\ec{\end{corollary}}

\def\geq{\geqslant}
\def\leq{\leqslant}

\def\<{\langle} \def\>{\rangle}
\def\x{{\bf x}}
\def\y{{\bf y}}

\allowdisplaybreaks

\tikzset{
        dot/.style={circle,fill=black,inner sep=0pt, outer sep=0.7pt, minimum size=1mm},
        Phi/.style={white!40!red,thick,snake=coil,segment amplitude=0.6pt, segment length=2pt},
         Z/.style={black!40!green,thick,snake=coil,segment amplitude=0.6pt, segment length=2pt},
        C/.style={thick,black!20!blue},
          Cr/.style={thick,black!20!red},
            Cg/.style={thick,black!20!green},
       }

\begin{document}

\title[The Fluctuating Boltzmann equation]{\LARGE Fluctuating Kinetic Theory: A Poissonian Stochastic Boltzmann Equation}

\author[Zhengyan Wu]{\large Zhengyan Wu}
\address[Z. Wu]{Department of Mathematics, Technische Universit\"at M\"unchen, Boltzmannstr. 3, 85748 Garching, Germany}
\email{wuzh@cit.tum.de}

\begin{abstract}
We introduce a nonlinear Poissonian fluctuating Boltzmann equation whose
noise encodes both the fluctuations and the path large-deviation rate
function of the underlying hard-sphere gas.  Unlike the Gaussian-noise
equations commonly studied in fluctuating
hydrodynamics, the present equation raises a new difficulty: a Poisson
random measure produces jumps that may destroy the nonnegativity of the
solution.  To address this problem, we construct a finite-dimensional
coarse-grained jump process and prove its well-posedness and
nonnegativity.  For each fixed mesh, this process satisfies a good path
large-deviation principle.  We then fix a regularized velocity cutoff and
study the asymptotic behavior of the discrete rate functions as the mesh
is refined.  On a class of biased regular paths, their limit is the
corresponding cutoff Boltzmann large-deviation rate function associated
with the hard-sphere gas.  This establishes the consistency of the
coarse-grained fluctuating Boltzmann model with the underlying particle
system at the level of path large deviations.
\end{abstract}


\date{\today}

\begin{nouppercase}
\maketitle
\end{nouppercase}

\setcounter{tocdepth}{1}
\tableofcontents

\section{Introduction}
Classical kinetic theory connects microscopic particle dynamics,
mesoscopic equations for one-particle distributions, and macroscopic fluid
equations.  Under a kinetic scaling, the particle dynamics leads to kinetic
equations, while a subsequent hydrodynamic scaling may lead from these
kinetic equations to fluid equations, as illustrated by the
Boltzmann-to-Navier--Stokes limit
\cite{GolseSaintRaymondNavierStokesLimit}.  Fluctuating hydrodynamics gives
a related mesoscopic description by adding stochastic fluxes to
deterministic balance laws in order to retain finite-system thermal
fluctuations, as in the Landau--Lifshitz theory
\cite{LandauLifshitzFluidMechanics}, while the deterministic fluid equations
are recovered when the fluctuation strength tends to zero.  Although
kinetic equations and fluctuating-hydrodynamic SPDEs are both mesoscopic
models, they retain different information about the underlying particle
dynamics and arise under different scaling regimes, so a comparison of
their scales requires a joint scaling of the particle number and the
kinetic-to-fluid parameters.  This observation motivates the program that
we call \emph{fluctuating kinetic theory}, which seeks to construct
stochastic kinetic equations that retain the fluctuations lost in the
kinetic law of large numbers and, in the longer term, to identify joint
fluctuation and hydrodynamic scalings leading to fluctuating fluid models
such as the Landau--Lifshitz--Navier--Stokes equations.

In this program, the form of the stochastic correction should, whenever
possible, be determined by the underlying particle dynamics, so that its
nonlinear coefficient, conservation laws, covariance, and large-deviation
Hamiltonian agree with the corresponding microscopic fluctuation and
large-deviation statistics.  Recent examples include
Vlasov--Fokker--Planck--Dean--Kawasaki equations
\cite{MullerVonRenesseZimmerVFPDK,HaoWuZimmerVFPDK} and a fluctuating
homogeneous Landau equation \cite{DuongHeWuFluctuatingLandau}, both of
which are associated with continuous diffusion-type mean-field particle
systems and are therefore driven by conservative Gaussian noise of
Dean--Kawasaki type.  Boltzmann dynamics, by contrast, is generated by
isolated binary collisions that produce discontinuous changes in particle
velocities, and the microscopic fluctuation and large-deviation theory of
a dilute hard-sphere gas records the statistics of these collision events
\cite{BodineauGallagherSaintRaymondSimonella}.  This collision structure
naturally leads to a compensated Poisson noise whose jumps carry the signed
pre/post-collisional increments and occur at the corresponding
state-dependent Boltzmann rates, thereby retaining the information needed
to capture both Gaussian fluctuations and nonlinear path large
deviations.

The purpose of this paper is to carry out this program for the Boltzmann
equation through the following three steps.
\begin{enumerate}[label=\textbf{Step \arabic*.},leftmargin=*]
 \item \emph{Identification of the noise.}  We use the known hard-sphere
 fluctuations and large deviations to identify a closed nonlinear
 Poissonian collision noise, matching not only the Gaussian covariance but
 also the microscopic path Hamiltonian.

 \item \emph{Regularized models and solution theory.}  Because the resulting
 continuum jump equation is singular and does not preserve nonnegativity,
 we construct nonnegativity-preserving coarse-grained regularizations and establish their
 well-posedness.

 \item \emph{Large deviations and mesh consistency.}  We prove a path
 large-deviation principle for the coarse process at each fixed mesh.  At a
 fixed velocity cutoff, we then compare the discrete rate functions with the
 continuum Boltzmann action as the mesh is refined.
\end{enumerate}

We first recall the deterministic Boltzmann equation and fix the notation
needed to formulate its fluctuating counterpart.  Let $d\geq2$.  The
one-particle phase space is
$\mathbb T^d\times\mathbb R^d$, and we write its points as $z=(x,v)$.
The associated collision space is
\begin{equation*}
 \mathcal C=\mathbb T^d\times\mathbb R^d\times\mathbb R^d
 \times\mathbb S^{d-1}.
\end{equation*}
We write $c=(x,v,v_*,\sigma)\in\mathcal C$ and
$dc=dx\,dv\,dv_*\,d\sigma$, where $d\sigma$ is surface measure on
$\mathbb S^{d-1}$.  For any phase-space function $r$, we use the standard
abbreviations $r_*=r(x,v_*)$, $r'=r(x,v')$, and
$r_*'=r(x,v_*')$.  The corresponding post-collisional velocities
$(v',v_*')$ are
\begin{equation}\label{collision-map}
 v'=\frac{v+v_*}{2}+\frac{|v-v_*|}{2}\sigma,
 \qquad
 v_*'=\frac{v+v_*}{2}-\frac{|v-v_*|}{2}\sigma.
\end{equation}
A collision with mark $c$ replaces the two incoming velocities $v,v_*$ by
$v',v_*'$.  Hence, for a test function $\varphi$, the change of the
corresponding two-particle observable is
\begin{equation}\label{collision-increment}
 \Delta\varphi(c)
 =\varphi(x,v')+\varphi(x,v_*')
  -\varphi(x,v)-\varphi(x,v_*).
\end{equation}
The inhomogeneous Boltzmann equation is
\begin{equation}\label{intro-boltzmann-equation}
 \partial_tf+v\cdot\nabla_xf=Q(f,f),
\end{equation}
where
\begin{equation}\label{weak-boltzmann-operator}
 \langle Q(f,f),\varphi\rangle
 =\frac12\int_{\mathcal C} B(v-v_*,\sigma)f(x,v)f(x,v_*)
 \Delta\varphi(c)\,dx\,dv\,dv_*\,d\sigma.
\end{equation}

At this stage we impose only the assumptions needed to define the collision
integrals.  The kernel $B$ is nonnegative and measurable and is locally
integrable in the following sense: for every $R<\infty$,
\begin{equation}\label{collision-kernel-assumption}
 \int_{B_R\times B_R\times\mathbb S^{d-1}}
 B(v-v_*,\sigma)\,dv\,dv_*\,d\sigma<\infty.
\end{equation}
We also impose the standard collision symmetries.  In the parametrization
\eqref{collision-map}, exchanging the two incoming particles sends
$(v,v_*,\sigma)$ to $(v_*,v,-\sigma)$, while exchanging the pre- and
post-collisional states sends
\begin{equation*}
 (v,v_*,\sigma)
 \quad\text{to}\quad
 \left(v',v_*',\frac{v-v_*}{|v-v_*|}\right),
 \qquad v\ne v_*.
\end{equation*}
Accordingly, we assume that, for almost every
$(v,v_*,\sigma)$ with $v\ne v_*$,
\begin{align}
 B(v-v_*,\sigma)
 &=B(v_*-v,-\sigma),\notag\\
 B(v-v_*,\sigma)
 &=B\left(v'-v_*',\frac{v-v_*}{|v-v_*|}\right).
 \label{collision-kernel-symmetries}
\end{align}
The first identity is the particle-exchange symmetry, and the second is the
pre/post-collisional symmetry.  Since both transformations preserve
$dv\,dv_*\,d\sigma$, these identities are equivalent to invariance of the
collision measure $B(v-v_*,\sigma)\,dv\,dv_*\,d\sigma$ under the two
transformations.  For the initial profile, we require that
$\int_{\mathbb T^d\times\mathbb R^d}f_0\,dz=1$.

\subsection{A Poissonian fluctuating Boltzmann equation}

To retain the randomness of the collision flux, we represent every binary
collision as a jump.  To distinguish the free variable of the resulting
phase-space measure from the collision mark $c=(x,v,v_*,\sigma)$, write
$\bar z=(\bar x,\bar v)$ for the former.  A collision with mark $c$ removes
one unit of mass from each incoming velocity and adds one unit at each
outgoing velocity.  Its signed phase-space increment is therefore
\begin{equation}\label{collision-jump-measure}
 J_c(d\bar z)=\delta_{(x,v')}(d\bar z)+\delta_{(x,v_*')}(d\bar z)
     -\delta_{(x,v)}(d\bar z)-\delta_{(x,v_*)}(d\bar z).
\end{equation}
If $\varepsilon$ is the mass carried by one effective particle, a single
collision changes the empirical density by $\varepsilon J_c$.  We construct
the collision events from a Poisson random measure.  Let
$\mathfrak N(dt,dc,du)$ be a Poisson random measure on
\begin{equation*}
 [0,\infty)\times\mathcal C\times[0,\infty)
\end{equation*}
with intensity $dt\,dc\,du$, where
$dc=dx\,dv\,dv_*\,d\sigma$, and set
\begin{equation*}
 \widetilde{\mathfrak N}(dt,dc,du)
 =\mathfrak N(dt,dc,du)-dt\,dc\,du.
\end{equation*}
Here $\widetilde{\mathfrak N}$ is the compensated Poisson random measure.
Under the usual integrability conditions, stochastic integrals of
predictable integrands with respect to
$\widetilde{\mathfrak N}$ are martingales.  In addition to the time $t$
and the collision mark $c$, the reference Poisson measure carries the
auxiliary variable $u\in[0,\infty)$.  A point $(t,c,u)$ is accepted as a
collision event when $u$ lies below the state-dependent collision rate.  In
this way, the fixed reference measure generates the required collision
process.
Guided by the jump increment $\varepsilon J_c$, we propose the following
formal Poissonian fluctuating Boltzmann equation, with $\bar z$ as its free
phase-space variable:
\begin{align}
 df_t^\varepsilon+\bar v\cdot\nabla_{\bar x}f_t^\varepsilon\,dt
 &=Q(f_t^\varepsilon,f_t^\varepsilon)\,dt\notag\\
 &\quad+\varepsilon\int_{\mathcal C\times[0,\infty)} J_c
 \mathbf1_{\left\{u\leq
 \frac{1}{2\varepsilon}B(v-v_*,\sigma)
 f_{t-}^\varepsilon(x,v)f_{t-}^\varepsilon(x,v_*)\right\}}
 \widetilde{\mathfrak N}(dt,dc,du).
 \label{poisson-fluctuating-boltzmann-strong}
\end{align}
We conjecture that the stochastic integral in
\eqref{poisson-fluctuating-boltzmann-strong} is the collision noise
associated with the Boltzmann equation.  Its threshold assigns to a
collision with mark $c$ the predictable rate
\begin{equation}\label{intro-collision-event-rate}
 \frac{1}{2\varepsilon}B(v-v_*,\sigma)
 f_{t-}^\varepsilon(x,v)f_{t-}^\varepsilon(x,v_*).
\end{equation}
This normalization is determined by the deterministic Boltzmann equation.
For every test function $\varphi$, one accepted collision changes
$\langle f^\varepsilon,\varphi\rangle$ by
$\langle\varepsilon J_c,\varphi\rangle
=\varepsilon\Delta\varphi(c)$.  Multiplying this jump size by the rate in
\eqref{intro-collision-event-rate} gives
\begin{equation*}
 \varepsilon\Delta\varphi(c)\,
 \frac{1}{2\varepsilon}B(v-v_*,\sigma)
 f^\varepsilon(x,v)f^\varepsilon(x,v_*)
 =\frac12Bf^\varepsilon f_*^\varepsilon\Delta\varphi(c).
\end{equation*}

Consequently, the weak form of 
\eqref{poisson-fluctuating-boltzmann-strong} reads 
\begin{align}
 d\langle f_t^\varepsilon,\varphi\rangle
 &=\langle f_t^\varepsilon,v\cdot\nabla_x\varphi\rangle\,dt
 +\frac12\int_{\mathcal C} Bf_t^\varepsilon f_{t,*}^\varepsilon
 \Delta\varphi\,dx\,dv\,dv_*\,d\sigma\,dt\notag\\
 &\quad+\varepsilon\int_{\mathcal C\times[0,\infty)} \Delta\varphi(c)
 \mathbf1_{\left\{u\leq
 \frac{1}{2\varepsilon}Bf_{t-}^\varepsilon
 f_{t-,*}^\varepsilon\right\}}
 \widetilde{\mathfrak N}(dt,dc,du).
 \label{poisson-fluctuating-boltzmann-weak}
\end{align}
Denote the martingale in the last line by
$M_t^\varepsilon(\varphi)$.  The predictable cross-variation formula for
compensated Poisson integrals then gives
\begin{align}
 d\langle M^\varepsilon(\varphi),M^\varepsilon(\psi)\rangle_t
 =\frac{\varepsilon}{2}\int_{\mathcal C}
 B(v-v_*,\sigma)f_{t-}^\varepsilon(x,v)
 f_{t-}^\varepsilon(x,v_*)
 \Delta\varphi(c)\Delta\psi(c)\,dc\,dt.
 \label{intro-FB-predictable-bracket}
\end{align}
The bracket in \eqref{intro-FB-predictable-bracket} is of order
$\varepsilon$, so the martingale is of order $\sqrt\varepsilon$.  This
suggests the following formal central-limit picture.  As
$\varepsilon\downarrow0$, the unscaled equation should converge to the
deterministic Boltzmann equation, while the centered field at scale
$\varepsilon^{-1/2}$ retains the rescaled collision martingale.  After
testing against a smooth bounded function $\varphi$, each jump of this
rescaled martingale has size $\sqrt\varepsilon\,\Delta\varphi(c)$ and
therefore vanishes in the limit.  If, in addition, the collision
activities $Bf^\varepsilon f_*^\varepsilon\,dt\,dc$ converge to
$Bff_*\,dt\,dc$, the bracket of the rescaled martingale converges to the
Gaussian collision covariance obtained from
\eqref{intro-FB-predictable-bracket} by replacing $f^\varepsilon$ with the
deterministic Boltzmann solution $f$.  A martingale central-limit argument
would then yield a continuous Gaussian noise.  Linearizing the Boltzmann
drift around $f$ gives the same limiting drift and covariance as the
hard-sphere fluctuation theory recalled below, providing a formal
consistency check for the proposed noise.

Beyond central-limit fluctuations, path large deviations provide the more
decisive reason for the Poissonian noise proposed in
\eqref{poisson-fluctuating-boltzmann-strong}.  The Gaussian covariance only
determines the quadratic part of the fluctuation structure and does not
identify the full nonlinear collision noise.  By contrast, a path
large-deviation principle retains the exponential Hamiltonian generated by
the collision counts.  Our guiding requirement is that the Hamiltonian and
rate function induced by the fluctuating equation should agree, at least
formally, with those of the microscopic hard-sphere gas.  The
weak-convergence framework for Poisson-driven equations developed by
Budhiraja, Chen and Dupuis \cite{BudhirajaChenDupuis} suggests the following
path rate function for \eqref{poisson-fluctuating-boltzmann-weak}.  For a
nonnegative density $g$, set
\begin{equation}\label{Boltzmann-accepted-collision-activity}
 a_g(c):=\frac12B(v-v_*,\sigma)g(x,v)g(x,v_*),
 \qquad c=(x,v,v_*,\sigma)\in\mathcal C.
\end{equation}
Given a nonnegative measurable control $h=h(t,c,u)$, its controlled
Boltzmann skeleton is the weak equation
\begin{align}
 \langle g_t,\varphi\rangle-\langle g_0,\varphi\rangle
 ={}&\int_0^t\langle g_s,v\cdot\nabla_x\varphi\rangle\,ds\notag\\
 &+\int_0^t\int_{\mathcal C}
 \left[a_{g_s}(c)+\int_0^{a_{g_s}(c)}
 \bigl(h(s,c,u)-1\bigr)\,du\right]
 \Delta\varphi(c)\,dc\,ds
 \label{BCD-Boltzmann-skeleton}
\end{align}
for every
$\varphi\in C_c^\infty(\mathbb T^d\times\mathbb R^d)$.
For a prescribed initial state $X_0$, define the fluctuating Boltzmann
candidate rate by
\begin{equation}\label{BCD-Boltzmann-rate-before-reduction}
 I_{X_0}^{\rm FB}(g)
 :=\inf_h\left\{
 \int_0^T\int_{\mathcal C}\int_0^\infty
 \ell(h(t,c,u))\,du\,dc\,dt:
 \begin{array}{l}
 h\geq0\text{ is measurable, and }g\text{ satisfies}\\[-1mm]
 \eqref{BCD-Boltzmann-skeleton}\text{ with }g_0=X_0
 \end{array}\right\},
\end{equation}
where $\ell(z)=z\log z-z+1$, $0\log0=0$, and the infimum of the empty
set is $+\infty$.  The
derivation of \eqref{BCD-Boltzmann-rate-before-reduction} from the abstract
Poisson framework is given in
Appendix~\ref{app:poisson-spde-ldp-framework}.
Since $\ell$ is nonnegative and
$\ell(1)=0$, the choice $h\equiv1$ has zero cost; in that case the square
bracket in \eqref{BCD-Boltzmann-skeleton} reduces to $a_{g_s}(c)$, and the
skeleton is the deterministic Boltzmann equation.  A control $h$ changes
the local intensity of accepted collision events, and
\eqref{BCD-Boltzmann-rate-before-reduction} is the relative-entropy cost of
that change of intensity.  The associated collision Hamiltonian contains
the Poisson factor
\begin{equation*}
 a_g(c)\bigl(e^{\Delta p(c)}-1\bigr).
\end{equation*}
Its linear term in $\Delta p$ is the Boltzmann collision drift, its
quadratic term gives the covariance in
\eqref{intro-FB-predictable-bracket}, and the full exponential retains the
non-Gaussian path large deviations.  In the next subsection, we recall the
central-limit fluctuations and the path large-deviation rate function of
the underlying hard-sphere gas.  Appendix~\ref{app:fluctuating-boltzmann-rate}
then compares this microscopic rate function with
\eqref{BCD-Boltzmann-rate-before-reduction} and explains why the two agree
on the regular path class covered by the hard-sphere theory.

\subsection{The underlying hard-sphere system and its asymptotics}
\label{subsec:hard-sphere-comparison}

We next describe the microscopic dynamics guiding the choice of the
Poissonian collision noise.  The benchmark is the dilute hard-sphere
system defined in
\cite[Section~1.1, (1.1.1)--(1.1.6)]{BodineauGallagherSaintRaymondSimonella},
where the deterministic hard-sphere flow and its grand-canonical initial
law are introduced.  We describe the particle dynamics and its random
initial state before recalling the law of
large numbers, central-limit fluctuations, and path large deviations.  Let
$\delta>0$ be the diameter of each sphere.  A configuration of $n$
spheres belongs to
\begin{equation*}
 \mathcal D_n^\delta
 =\left\{(x_i,v_i)_{i=1}^n\in
 (\mathbb T^d\times\mathbb R^d)^n:
 |x_i-x_j|>\delta\text{ for }i\ne j\right\}.
\end{equation*}
Between collisions the particles follow
\begin{equation}\label{hard-sphere-free-flow}
 \dot X_i^\delta=V_i^\delta,
 \qquad \dot V_i^\delta=0.
\end{equation}
When $|X_i^\delta-X_j^\delta|=\delta$, put
$\omega=(X_i^\delta-X_j^\delta)/\delta$. The velocities are reflected by
\begin{equation}\label{hard-sphere-collision-rule}
 (V_i^\delta)'=V_i^\delta-\bigl((V_i^\delta-V_j^\delta)\cdot\omega\bigr)\omega,
 \qquad
 (V_j^\delta)'=V_j^\delta+\bigl((V_i^\delta-V_j^\delta)\cdot\omega\bigr)\omega.
\end{equation}

Equations~\eqref{hard-sphere-free-flow} and
\eqref{hard-sphere-collision-rule} define the deterministic evolution once
the particle number and the initial configuration have been chosen.  We
equip this dynamics with the grand-canonical Poissonian initial ensemble
defined in
\cite[Section~1.1, equation~(1.1.6) and the following display]{BodineauGallagherSaintRaymondSimonella}.
More precisely, $N$ denotes the random number of spheres,
$\mathbb P_\delta$ denotes their initial law, and $\mathbb E_\delta$ is the
corresponding expectation.  This law is obtained from a Poisson point
field with activity $\delta^{-(d-1)}$ and one-particle profile $f_0$ by
imposing the hard-core nonoverlap constraint and renormalizing.
The cited formulas give its precise density and partition function.

All limits in this subsection
are taken in the Boltzmann--Grad regime
\begin{equation}\label{hard-sphere-Boltzmann-Grad-regime}
 \delta\downarrow0,\qquad \delta^{-(d-1)}\uparrow\infty.
\end{equation}
The typical particle number is of order
$\delta^{-(d-1)}$ and therefore diverges as a consequence of
$\delta\downarrow0$.

\medskip
\noindent\emph{Law of large numbers and fluctuations.}
The empirical measure, normalized by the particle-number scale
$\delta^{-(d-1)}$, is
\begin{equation}\label{hard-sphere-empirical-measure}
 \pi_t^\delta
 =\delta^{d-1}\sum_{i=1}^{N}
 \delta_{(X_i^\delta(t),V_i^\delta(t))}.
\end{equation}
On a short kinetic time interval $[0,T_*]$, this empirical measure
converges in probability, when tested against smooth functions, to
$f_t(z)\,dz$, where $f$ solves \eqref{intro-boltzmann-equation}.  The
centered fluctuation field is
\begin{equation}\label{hard-sphere-fluctuation-field}
 \zeta_t^\delta
 =\delta^{-(d-1)/2}
 \left(\pi_t^\delta-\mathbb E_\delta\pi_t^\delta\right).
\end{equation}
In the same Boltzmann--Grad limit, the main fluctuation result of
\cite{BodineauGallagherSaintRaymondSimonella} states that
$\zeta^\delta$ converges in law, in the weak distributional topology used
there, to a centered Gaussian field $Y$.  Its dynamical evolution is
described by the linearized fluctuating Boltzmann equation
\begin{align}
 d\langle Y_t,\varphi\rangle
 &=\langle Y_t,v\cdot\nabla_x\varphi\rangle\,dt
 +\langle Q(f_t,Y_t)+Q(Y_t,f_t),\varphi\rangle\,dt
 +dM_t^\varphi,
 \label{linearized-fluctuating-boltzmann}
\end{align}
where
\begin{align}
 \langle M^\varphi,M^\psi\rangle_t
 =\frac12\int_0^t\int_{\mathcal C}
 B(v-v_*,\sigma)f_s(x,v)f_s(x,v_*)
 \Delta\varphi\Delta\psi
 \,dx\,dv\,dv_*\,d\sigma\,ds.
 \label{boltzmann-fluctuation-covariance}
\end{align}
Here $M$ is the continuous Gaussian collision martingale, and
\eqref{boltzmann-fluctuation-covariance} is the covariance generated by
the microscopic binary collisions.

\medskip
\noindent\emph{Path large deviations.}
For a fixed $r>0$ and a time $T\leq T_*$ allowed by the result cited
below, the probabilities of atypical paths decay at speed
$\delta^{-(d-1)}$ in the Boltzmann--Grad limit.  More precisely,
\cite[Theorem~9, equations~(7.0.10)--(7.0.11)]{BodineauGallagherSaintRaymondSimonella}
gives the following bounds for closed sets $\mathcal F$ and open sets
$\mathcal O$ in the path topology used there.
\begin{align*}
 \limsup_{\delta\downarrow0}\delta^{d-1}
 \log\mathbb P_\delta(\pi^\delta\in\mathcal F)
 &\leq-\inf_{g\in\mathcal F}I_{\rm HS}(g),\\
 \liminf_{\delta\downarrow0}\delta^{d-1}
 \log\mathbb P_\delta(\pi^\delta\in\mathcal O)
 &\geq-\inf_{g\in\mathcal O\cap\mathcal R}I_{\rm HS}(g).
\end{align*}
The upper bound applies without a regularity restriction, whereas the
lower bound is restricted to the class
$\mathcal R:=\mathcal R_{r,T}$ defined in
\cite[Theorem~3 and equation~(7.0.7)]{BodineauGallagherSaintRaymondSimonella}.
This class consists of strong solutions to a biased Boltzmann equation
generated by admissible Lipschitz bias fields.  The result
therefore gives a restricted large-deviation lower bound rather than a
full lower bound over all paths.  The same restriction plays an important
role in Section~\ref{sec:collision-limit-rigorous}, where the continuum and
discrete rate functions are compared on the corresponding class of biased
regular paths.

The hard-sphere rate contains both an initial cost and a dynamical cost.
For the grand-canonical Poissonian initial ensemble, the initial cost is
\begin{equation}\label{hard-sphere-initial-rate}
 I_0(g)=\int_{\mathbb T^d\times\mathbb R^d}
 \left[g\log\frac{g}{f_0}-g+f_0\right]dz,
\end{equation}
when $g(z)\,dz$ is absolutely continuous with respect to $f_0(z)\,dz$;
otherwise $I_0(g)=+\infty$. We use $0\log0=0$ throughout.
The complete hard-sphere rate decomposes as
$I_{\rm HS}(g)=I_0(g_0)+\mathcal I_{\rm Ham}(g)$, where
$\mathcal I_{\rm Ham}$ describes the cost of the dynamics.  We next
introduce its Hamiltonian and control representations.

\medskip
\noindent\emph{The dynamical rate functional.}
We first introduce the Hamiltonian associated with the collision dynamics.
At a fixed time, let $g$ denote the current particle density on
$\mathbb T^d\times\mathbb R^d$.  We assume that $g$ is nonnegative and has
finite mass, finite second velocity moment, and finite collision activity,
which means that
\begin{align*}
 \int_{\mathbb T^d\times\mathbb R^d}(1+|v|^2)g(x,v)\,dx\,dv&<\infty,\\
 \int_{\mathcal C}B(v-v_*,\sigma)g(x,v)g(x,v_*)
 \,dx\,dv\,dv_*\,d\sigma&<\infty.
\end{align*}
For a test function
$p\in C_c^1(\mathbb T^d\times\mathbb R^d;\mathbb R)$, its collision
increment $\Delta p$ is defined by \eqref{collision-increment}.  We define the dynamical
Hamiltonian by
\begin{equation}\label{boltzmann-collision-hamiltonian}
 \mathcal H_{\rm coll}(g,p)
 =\frac12\int_{\mathcal C}B(v-v_*,\sigma)g(x,v)g(x,v_*)
 \bigl(e^{\Delta p}-1\bigr)
 \,dx\,dv\,dv_*\,d\sigma.
\end{equation}
For $g$ in the common effective domain $\mathscr D_T$ of
Definition~\ref{def:common-effective-domain}, define the dynamical
Hamiltonian action by
\begin{align}
 \mathcal I_{\rm Ham}(g)
 :=\sup_{p\in C_c^1([0,T]\times\mathbb T^d\times\mathbb R^d)}\Bigg\{
 &\langle g_T,p_T\rangle-\langle g_0,p_0\rangle
 -\int_0^T\langle g_t,\partial_tp_t+v\cdot\nabla_xp_t\rangle dt
 \notag\\
 &-\int_0^T\mathcal H_{\rm coll}(g_t,p_t)dt\Bigg\}.
 \label{hard-sphere-hamiltonian-action}
\end{align}
The hard-sphere path rate is
\begin{equation}\label{hard-sphere-path-rate}
 I_{\rm HS}(g)=I_0(g_0)+\mathcal I_{\rm Ham}(g).
\end{equation}
The hard-sphere rate is stated in the Hamiltonian form
\eqref{hard-sphere-hamiltonian-action} in
\cite[Section~1.4, equations~(1.4.1)--(1.4.2)]{BodineauGallagherSaintRaymondSimonella}, where the collision Hamiltonian is written in terms of the
hard-sphere collision measure $d\mu(z_1,z_2,\omega)$ defined in
\cite[equation~(1.3.7)]{BodineauGallagherSaintRaymondSimonella}.  After
integrating out the spatial Dirac mass and changing from the collision
normal $\omega$ to the post-collisional direction $\sigma$, their formula
takes the form \eqref{boltzmann-collision-hamiltonian} with the corresponding
hard-sphere kernel.  We use the same Hamiltonian structure for the more
general collision kernel $B$ considered here.  Closely related Hamiltonian
variational formulas also appear in
\cite[Section~4.1, equations~(24)--(27)]{Bouchet} for dilute-gas collision
dynamics and in the rigorous large-deviation result \cite{Rezakhanlou} for
a one-dimensional stochastic kinetic model with discrete velocities.

\cite[equations~(1.4.6)--(1.4.7)]{BodineauGallagherSaintRaymondSimonella} describes regular atypical paths by a biased Boltzmann
equation whose collision rates contain an exponential multiplier. For the Poissonian model considered here, it is useful to express this bias
directly as a change of collision intensity.  The weak-convergence
representation controls each Poisson intensity by a nonnegative multiplier
$q$ and assigns to it the relative-entropy cost $\ell(q)$.  The resulting
control form provides the link between the Poissonian fluctuating Boltzmann
equation and the hard-sphere Hamiltonian rate, and it is also the form that
arises naturally for the coarse-grained model.

To define this control form, set $\mathcal Q(g)=\varnothing$
unless $g_t(dz)=g_t(z)\,dz$ for almost every $t\in[0,T]$. For such a
density path, let $\mathcal Q(g)$ be the set of measurable functions
$q:[0,T]\times\mathcal C\to[0,\infty)$ such that
\begin{equation*}
 \int_0^T\int_{\mathcal C}B(v-v_*,\sigma)g_t(x,v)g_t(x,v_*)
 \ell(q_t(c))\,dc\,dt<\infty
\end{equation*}
and, for every $\varphi\in C_c^\infty(\mathbb T^d\times\mathbb R^d)$ and every $t\in[0,T]$,
\begin{align}
 \langle g_t,\varphi\rangle-\langle g_0,\varphi\rangle
 &=\int_0^t\langle g_s,v\cdot\nabla_x\varphi\rangle\,ds\notag\\
 &\quad+\frac12\int_0^t\int_{\mathcal C}B(v-v_*,\sigma)
 g_s(x,v)g_s(x,v_*)q_s(c)\Delta\varphi(c)\,dc\,ds.
 \label{controlled-boltzmann-equation}
\end{align}
Here $\ell(q)=q\log q-q+1$, with $0\log0=0$. The corresponding
collision-control action is
\begin{equation}\label{boltzmann-path-rate-control}
 \mathcal I_B(g)=\inf_{q\in\mathcal Q(g)}\frac12
 \int_0^T\int_{\mathcal C}B(v-v_*,\sigma)g_t(x,v)g_t(x,v_*)
 \ell(q_t(c))\,dc\,dt,
\end{equation}
with the convention $\inf\varnothing=+\infty$.  Thus the infimum in
\eqref{boltzmann-path-rate-control} is taken precisely over the collision-rate
multipliers that generate $g$ through
\eqref{controlled-boltzmann-equation}.

The three dynamical rate functions introduced above are related by
\begin{align*}
 I_{X_0}^{\rm FB}(g)
 &=\begin{cases}
 \mathcal I_B(g),&g_0=X_0,\\
 +\infty,&g_0\neq X_0,
 \end{cases}\\
 \mathcal I_{\rm Ham}(g)&\leq\mathcal I_B(g).
\end{align*}
On the exponentially biased class $\mathcal R$ of
\cite{BodineauGallagherSaintRaymondSimonella}, the second relation is an
equality.  Consequently,
\begin{equation*}
 I_{\rm HS}(g)=I_0(g_0)+\mathcal I_B(g),
 \qquad g\in\mathcal R.
\end{equation*}
The first identity and the Hamiltonian--control comparison are proved in
Lemmas~\ref{lem:FB-control-rate-equivalence}
and~\ref{lem:Hamiltonian-control-rate-comparison}, respectively, in
Appendix~\ref{app:fluctuating-boltzmann-rate}.

The hard-sphere results therefore provide two precise microscopic targets
for a fluctuating Boltzmann model: the Gaussian covariance
\eqref{boltzmann-fluctuation-covariance} at central-limit scale and the
nonquadratic action \eqref{hard-sphere-path-rate} at large-deviation scale.
The rate-function comparison is proved in
Appendix~\ref{app:fluctuating-boltzmann-rate}.  The formal discussion after
\eqref{intro-FB-predictable-bracket} explains why the Poissonian collision
noise has the same Gaussian covariance at central-limit scale.

\subsection{New challenges and a coarse-grained model}
\label{subsec:coarse-model-intro}

Equation \eqref{poisson-fluctuating-boltzmann-strong} introduces a type of
fluctuating kinetic equation that differs from both the usual models of
fluctuating hydrodynamics and the Gaussian-noise kinetic equations studied
for the fluctuating Vlasov--Fokker--Planck equation
\cite{HaoWuZimmerVFPDK} and the fluctuating Landau equation
\cite{DuongHeWuFluctuatingLandau}.  The Poisson noise reflects the
discontinuous nature of binary collisions, but it also creates a new
obstruction to nonnegativity.  To distinguish the free phase-space variable
from the collision parameters, write the pre-jump state explicitly as the
absolutely continuous measure
$f(\bar x,\bar v)\,d\bar x\,d\bar v$.  An accepted collision with mark
$c=(x,v,v_*,\sigma)$ subtracts the atoms
$\varepsilon\delta_{(x,v)}$ and
$\varepsilon\delta_{(x,v_*)}$, while its intensity is positive whenever
the two incoming density values $f(x,v)$ and $f(x,v_*)$ are positive.
Nevertheless, absolute continuity gives
\begin{equation*}
  \int_{\{(x,v)\}} f(\bar x,\bar v)\,d\bar x\,d\bar v
  =
  \int_{\{(x,v_*)\}} f(\bar x,\bar v)\,d\bar x\,d\bar v
  =0.
\end{equation*}
Hence, for a nondegenerate collision, the post-jump signed measure $f(\bar x,\bar v)\,d\bar x\,d\bar v+\varepsilon J_c$ 
assigns mass $-\varepsilon$ to each departure singleton and is not
nonnegative.  It therefore no longer represents a particle density, and
the pointwise factor $f(x,v)f(x,v_*)$ in the collision intensity is no
longer defined.  Thus the equation cannot be continued in its natural
state space after such a jump.

Mollifying the noise does not remove this obstruction. If
$K_\eta$ is a smooth approximation of the identity, the regularized jump
has the form
\begin{align*}
 \varepsilon\bigl[&K_\eta(\mathord\cdot-(x,v'))
 +K_\eta(\mathord\cdot-(x,v_*'))-K_\eta(\mathord\cdot-(x,v))
 -K_\eta(\mathord\cdot-(x,v_*))\bigr].
\end{align*}
It is still a signed function with negative departure lobes. Preservation of
nonnegativity would require the pre-jump density to dominate these lobes pointwise.
There is no such lower bound near vacuum, in the velocity tails, or for a
general nonnegative initial density. Smoothing the intensity as well does
not help: it changes the rate at which a departure location is selected,
but it does not cap the amount removed by a jump by the mass actually
available there. The missing ingredient is therefore a local capacity
constraint, not spatial regularity of the noise.

The key idea of the coarse-grained model is to discretize the admissible
cell masses as nonnegative integer multiples of $\varepsilon$ and to allow
a jump only when its departure cells contain the required mass quanta.
The jump rates then vanish whenever a proposed event would remove more
mass than is locally available.  This turns the formal continuum dynamics
into a pure-jump process that remains in a nonnegative discrete state
space.  To implement this idea, we truncate the velocity region and
partition the resulting phase-space domain $\mathbb D_h$ into finitely
many cells $(E_a^h)_{a\in\mathcal I_h}$.  We give only a short description
of the resulting process here.  The domain $\mathbb D_h$, the partition,
and all coefficients and counting processes used below are defined
precisely in Subsection~\ref{subsec:def-coarse-grained-model}.  The state
variable is the vector of cell masses
\begin{equation*}
 m^{\varepsilon,h}(t)=(m_a(t))_{a\in\mathcal I_h}
 \in(\varepsilon\mathbb N_0)^{\mathcal I_h},
\end{equation*}
where $\varepsilon$ is one mass quantum.  If
$\chi_a^h=|E_a^h|^{-1}\mathbf1_{E_a^h}$, the reconstructed density
\begin{equation*}
 f_t^{\varepsilon,h}
 =\sum_{a\in\mathcal I_h}m_a(t)\chi_a^h
\end{equation*}
is piecewise constant and has cell mass
$\int_{E_a^h}f_t^{\varepsilon,h}\,dz=m_a(t)$.  Thus the collection of all
cell masses is equivalent to a finite-volume approximation of the
phase-space density.  The precise reconstruction is given in
\eqref{coarse-boltzmann-reconstruction}.

The mass vector evolves through two families of pure jumps.  A transport
event moves one quantum from a cell $a$ to an adjacent spatial cell $b$,
whereas a collision event
$\gamma=(a,b;c,d)$ removes one quantum from each incoming cell $a,b$ and
adds one to each outgoing cell $c,d$.  Their state-dependent intensities
have the form
\begin{equation*}
 \lambda_{ab}^{\varepsilon,h}(m)
 =\varepsilon^{-1}m_aT_{ab}^h,
 \qquad
 \lambda_\gamma^{\varepsilon,h}(m)
 =\frac{K_\gamma^h}{2\varepsilon}
 m_a\bigl(m_b-\varepsilon\mathbf1_{\{a=b\}}\bigr),
\end{equation*}
as defined in \eqref{coarse-transport-rate}--\eqref{coarse-boltzmann-collision-rate}.  The coefficients $T_{ab}^h$ and
$K_\gamma^h$ discretize free transport and the collision kernel and are
defined in \eqref{coarse-upwind-transport-coefficients} and
\eqref{riemann-collision-quadrature}, respectively.  These rates enforce
nonnegativity pathwise.  On one hand, an empty departure cell has zero
transport rate. On the other hand, if $a\ne b$, a collision has zero rate as soon as either
incoming cell is empty; and if $a=b$, the falling-factorial term
$m_a(m_a-\varepsilon)$ permits a collision only when that cell contains
at least two mass quanta.  Consequently no jump can remove more mass from
a cell than is present there.

Let $P_{ab}^{\varepsilon,h}$ and $P_\gamma^{\varepsilon,h}$ be the
corresponding counting processes, with the preceding predictable
intensities, and let their compensated versions be
$\widetilde P_{ab}^{\varepsilon,h}$ and
$\widetilde P_\gamma^{\varepsilon,h}$.  Their random-time-change
construction is given in \eqref{coarse-transport-counts}--
\eqref{coarse-collision-counts}, and the active channel sets
$\mathcal E_h$ and $\Gamma_h$ are specified in
Subsection~\ref{subsec:def-coarse-grained-model}.  After reconstructing the density
and separating each counting process into its compensator and martingale
part, one obtains
\begin{align*}
 df_t^{\varepsilon,h}
 &=\Bigg[\mathcal T_h^*f_t^{\varepsilon,h}
 +\frac12\sum_{\gamma=(a,b;c,d)\in\Gamma_h}K_\gamma^h
 m_a(t-)\bigl(m_b(t-)-\varepsilon\mathbf1_{\{a=b\}}\bigr)
 (\chi_c^h+\chi_d^h-\chi_a^h-\chi_b^h)\Bigg]dt\\
 &\quad+\varepsilon\sum_{(a,b)\in\mathcal E_h}
 (\chi_b^h-\chi_a^h)\,d\widetilde P_{ab}^{\varepsilon,h}(t)\\
 &\quad+\varepsilon\sum_{\gamma=(a,b;c,d)\in\Gamma_h}
 (\chi_c^h+\chi_d^h-\chi_a^h-\chi_b^h)
 \,d\widetilde P_\gamma^{\varepsilon,h}(t),
\end{align*}
where, for $f=\sum_{a\in\mathcal I_h}m_a\chi_a^h$,
\begin{align*}
 \mathcal T_h^*f
 &=\sum_{(a,b)\in\mathcal E_h}
 m_aT_{ab}^h(\chi_b^h-\chi_a^h).
\end{align*}
The corresponding density equation is derived in
Subsection~\ref{subsec:def-coarse-grained-model}, in particular in
\eqref{coarse-boltzmann-density-equation}.  The transport compensator gives
the finite-volume upwind drift.  The leading collision contribution with
coefficient $m_am_b$ is the cell-integrated Boltzmann drift, while the
falling-factorial correction remains part of the model at fixed
$\varepsilon$.  The compensated collision counts describe the corresponding
Poissonian collision fluctuations.  The transport martingale is needed to
move individual mass quanta without violating nonnegativity, but its
quadratic variation vanishes with the spatial mesh, so it creates no
additional transport noise in the continuum model.

The main results of this coarse-grained model can be summarized in the following three aspects: 
\begin{enumerate}[label=\textbf{(\arabic*)},leftmargin=*]
 \item For fixed $(\varepsilon,h)$, the cell-mass equation has a pathwise
 unique nonexplosive strong solution that remains nonnegative and
 preserves total mass; see
 Theorem~\ref{thm:intro-coarse-wellposedness-rigorous}.

 \item For fixed $h$, as $\varepsilon\downarrow0$, the density path
 satisfies a good large-deviation principle with speed
 $\varepsilon^{-1}$ and an explicit Poisson control rate function; see
 Theorem~\ref{thm:intro-fixed-mesh-ldp-rigorous}.

 \item At a fixed velocity cutoff, as $h_x+h_v\downarrow0$, the discrete
 rate functions satisfy the $\Gamma$-liminf and recovery properties on the
 class of biased regular paths.  Thus the coarse model recovers the
 continuum collision action locally in path space; see
 Theorem~\ref{thm:intro-regular-path-rate-convergence}.

\end{enumerate}

\subsection{Related literature}
\label{subsec:related-literature}
We briefly review the developments that are most closely related to the
questions studied in this paper.

\medskip
\noindent\emph{Kinetic theory and the Boltzmann equation.}
Since Boltzmann's derivation of the collision equation for a dilute gas
\cite{Boltzmann1872}, kinetic theory has served as a mesoscopic description
between Newtonian particle dynamics and continuum fluid mechanics.  The
Boltzmann equation combines free transport with local binary collisions,
preserves mass, momentum, and energy, and dissipates the Boltzmann entropy;
standard accounts of its derivation, collision geometry, and analytic
theory can be found in
\cite{CercignaniIllnerPulvirenti,VillaniCollisionalReview}.  Its PDE theory
contains several complementary notions of solution.  For cutoff kernels
and sufficiently regular data, classical solutions can be constructed in
perturbative and spatially homogeneous regimes, while the renormalized
solution theory of DiPerna and Lions gives global existence and weak
stability for large data with finite mass, energy, and entropy
\cite{DiPernaLionsBoltzmann}.  For long-range interactions, the removal of
the angular cutoff reveals fractional and hypoelliptic regularization in
the collision operator, and foundational estimates and existence results
were established in
\cite{AlexandreDesvillettesVillaniWennberg,AlexandreVillaniLongRange,
AlexandreMorimotoUkaiXuYang}.  The broader PDE theory also includes global
classical solutions near Maxwellian equilibria, boundary-value problems,
and quantitative convergence to equilibrium; see, for example,
\cite{UkaiGlobalBoltzmann,GuoWholeSpace,GuoBoundedDomains,
DesvillettesVillaniEquilibrium,MouhotNeumann}.

Hydrodynamic limits connect the deterministic kinetic description to fluid
equations.  The incompressible Navier--Stokes limit in
\cite{GolseSaintRaymondNavierStokesLimit} is particularly close to the
kinetic-to-fluid perspective that motivates fluctuating kinetic models.
Related derivations use the relative-entropy and compactness methods of
Bardos, Golse and Levermore and of Lions and Masmoudi, while later results
cover hard cutoff potentials and the full incompressible
Navier--Stokes--Fourier system; see
\cite{BardosGolseLevermoreFluidLimits,LionsMasmoudiFluidMechanicsII,
GolseSaintRaymondHardCutoff,LevermoreMasmoudiNSF}.  These results provide
the deterministic kinetic background for the drift of our model. 

\medskip
\noindent\emph{Particle systems, the Boltzmann--Grad limit, and
fluctuations.}
The microscopic origin of the Boltzmann equation is most explicit for a
dilute gas of hard spheres.  Lanford proved that, for short kinetic times,
the BBGKY hierarchy converges in the Boltzmann--Grad scaling to the
Boltzmann hierarchy and hence, for chaotic initial data, to the Boltzmann
equation \cite{Lanford1975}; modern presentations and extensions to
short-range potentials are given in
\cite{GallagherSaintRaymondTexier}.  More recently, Deng, Hani and Ma
extended the hard-sphere derivation to every time interval on which the
corresponding Boltzmann equation has a sufficiently regular solution,
thereby removing Lanford's short-time restriction within this regularity
regime \cite{DengHaniMaLongTimeBoltzmann}.  A complementary spatially
homogeneous approach begins with Kac's Markov jump model \cite{Kac1956},
for which propagation of chaos, quantitative stability, and estimates
uniform in the number of particles were developed further in
\cite{MischlerMouhotKac}.  These law-of-large-numbers results explain the
Boltzmann collision drift, but they do not describe the random field
around that limit.

At the next order, fluctuations of collision processes are naturally
Gaussian after central-limit scaling, as already investigated for kinetic
jump models by McKean \cite{McKeanFluctuations}.  Large deviations retain
more information because they record the nonlinear cost of changing the
collision flux rather than only its quadratic covariance.  Kinetic large
deviations were studied, among other works, in
\cite{Rezakhanlou,Bouchet}.  For the conservative Kac process, Heydecker
proved an upper bound and a matching lower bound on a restricted class of
regular paths, but also constructed an explicit counterexample showing
that the candidate rate function does not yield the global large-deviation
lower bound \cite{HeydeckerKacLDP}. This observation is important for the
restricted regular-path comparison carried out in
Section~\ref{sec:collision-limit-rigorous}.

The rate-functional viewpoint has also been developed further for
spatially homogeneous collision models.  Basile, Benedetto, Bertini and
Heydecker studied the large-time behavior of the Kac-walk rate function and
identified a dynamical phase transition in the collision activity
\cite{BasileBenedettoBertiniHeydecker}, while Basile, Benedetto and Orrieri
used the corresponding measure--flux action to characterize the
energy-conserving homogeneous Boltzmann solution and derive it from Kac's
walk under entropic chaoticity \cite{BasileBenedettoOrrieri}.  The work
most directly related to the present paper is that of Bodineau, Gallagher,
Saint-Raymond and Simonella
\cite{BodineauGallagherSaintRaymondSimonella}, which derives, from the
spatial hard-sphere dynamics on a short time interval, both the Gaussian
fluctuating Boltzmann equation and a path large-deviation functional.  The
covariance and collision Hamiltonian obtained there are the two
microscopic consistency tests that guide our choice of the nonlinear
Poisson noise.

\medskip
\noindent\emph{Fluctuating hydrodynamics and conservative SPDEs.}
Fluctuating hydrodynamics adds random fluxes to deterministic hydrodynamic limits, with a covariance determined by the corresponding dissipation.  In
the Landau--Lifshitz theory
\cite{LandauLifshitzFluidMechanics}, the viscous stress and heat flux in the
Navier--Stokes--Fourier system acquire Gaussian fluctuations, and its
incompressible reduction is commonly called the
Landau--Lifshitz--Navier--Stokes equation.  This physical picture is
closely related to macroscopic fluctuation theory, which describes
hydrodynamic fluctuations and current large deviations in conservative
many-particle systems \cite{BertiniEtAlMFT}.  Recent mathematical work has
studied solution concepts and large deviations for the
Landau--Lifshitz--Navier--Stokes equation
\cite{GessHeydeckerWuLLNS}, as well as the incompressible
Navier--Stokes--Fourier system with thermal noise
\cite{GessSauerbreyWuThermalNSF}.

The Dean--Kawasaki equation has so far generated the broadest mathematical
theory among models of fluctuating hydrodynamics.  Dean's formal
calculation rewrites the empirical density of interacting Langevin
particles as a conservative SPDE with square-root noise \cite{Dean1996},
but the equation driven by unregularized space--time white noise is highly
rigid and is ill-posed outside the quantized empirical-measure setting
\cite{KonarovskyiLehmannRenesse}.  This difficulty has motivated the
systematic study of correlated-noise and regularized models.  Basic
analytical results include the kinetic solution theory in
\cite{FehrmanGessDeanKawasaki}, the non-equilibrium large-deviation theory
in \cite{FehrmanGessNonEquilibrium}, and the rigorous connection between
conservative SPDEs and the fluctuations and large deviations of the
symmetric simple exclusion process in \cite{DirrFehrmanGessSSEP}.  Further
developments treat
higher-order fluctuation expansions \cite{GessWuZhangHigherOrder},
singular interactions and nonlinear fluctuations near criticality
\cite{WangWuZhangSingular,WuNonlinearFluctuations}, equations on bounded
domains \cite{PopatBoundedDomains}, and singular Keller--Segel
approximations \cite{MartiniMayorcasKSDK}.  Regularized and discrete models
are compared quantitatively with their underlying particle systems in
\cite{CornalbaShardlowZimmer,CornalbaFischerStructure}, and weak-error
estimates have been obtained for independent particles and for smooth
mean-field interactions in
\cite{DjurdjevacKrempPerkowski,DjurdjevacJiPerkowski}.

Conservative multiplicative noise also appears in stochastic thin-film
equations, where thermal fluctuations act through the mobility of the
film.  Following the physical model in \cite{GrunMeckeRauscher}, positive
or nonnegative martingale solutions have been constructed for several
mobility laws, spatial dimensions, and noise structures; see
\cite{FischerGruenThinFilm,GessGnannThinFilm,DareiotisGessGnannGruen,
MetzgerGruenThinFilm,SauerbreyThinFilm}.

Fluctuating kinetic equations seek an analogous description while
retaining the phase-space variable.  For kinetic diffusions, this leads to
Dean--Kawasaki-type Gaussian noise, as in the
Vlasov--Fokker--Planck model of \cite{HaoWuZimmerVFPDK} and the regularized
homogeneous Landau model of \cite{DuongHeWuFluctuatingLandau}.  The
Boltzmann equation has a different microscopic structure because a binary
collision is a jump rather than a continuous diffusion.  The Gaussian
equation derived in \cite{BodineauGallagherSaintRaymondSimonella} captures
the central-limit regime, while the state-dependent Poisson model proposed
here retains both this Gaussian covariance and the nonquadratic
large-deviation Hamiltonian.  This is the main structural difference
between the present equation and previously studied fluctuating kinetic
diffusions.

\medskip
\noindent\emph{Large deviations for Poisson-driven SPDEs.}
The weak-convergence approach expresses a Laplace principle as a stochastic
control problem; see the monograph \cite{DupuisEllis} and the variational
representation of Budhiraja and Dupuis for infinite-dimensional Brownian
noise \cite{BudhirajaDupuisBrownian}.  For Poisson random measures, the
corresponding representation uses predictable controls to change the
intensity and assigns to an intensity multiplier $q$ the relative-entropy
density $\ell(q)=q\log q-q+1$ \cite{BudhirajaDupuisMaroulas}.  Budhiraja,
Chen and Dupuis used this representation to develop an abstract
large-deviation framework for Poisson-driven SPDEs
\cite{BudhirajaChenDupuis}, in which the main tasks are to establish
convergence of the controlled stochastic equations to a skeleton equation
and compactness of the skeleton trajectories with bounded control cost.

The weak-convergence method has since been used for a wide range of
SPDEs with jump noise.  Yang, Zhai and Zhang treated fully nonlinear
stochastic evolution equations driven jointly by Brownian motion and
Poisson random measures \cite{YangZhaiZhangJumpSPDE}, and Zhai and Zhang
applied the method to two-dimensional stochastic Navier--Stokes equations
with multiplicative L\'evy noise \cite{ZhaiZhangNavierStokesLevy}.  Later
applications include locally monotone SPDEs
\cite{XiongZhaiLocallyMonotone}, two-dimensional Navier--Stokes equations
with locally Lipschitz jump coefficients
\cite{BrzezniakPengZhaiNavierStokes}, and generalized stochastic porous
medium equations \cite{WuZhaiPorousMedia}.  Together, these works show how
the same control representation can be adapted to different SPDE
structures once the controlled dynamics and the required compactness are
understood.

This framework is especially natural for the collision process considered
here because a control acts directly on the intensity of each collision
channel, and its relative-entropy cost is dual to the exponential Poisson
Hamiltonian.  Section~\ref{sec:fixed-mesh-ldp-rigorous} uses this structure
to prove the fixed-mesh large-deviation principle, while
Section~\ref{sec:collision-limit-rigorous} compares the resulting discrete
action with the continuum hard-sphere rate functional.

\subsection{Structure of the paper}
\label{subsec:paper-structure}
Section~\ref{sec:coarse-setup-main} introduces the finite phase-space mesh,
the transport and collision channels, and the random-time-change equation
for the cell masses.  It then reconstructs the piecewise constant density
and states the three main results: fixed-mesh well-posedness and
nonnegativity, the fixed-mesh path LDP, and the regular-path limit of the
rate functions under mesh refinement.
Section~\ref{sec:coarse-wellposedness} proves existence, pathwise
uniqueness, nonexplosion, and the conservation properties of the
coarse-grained jump process.  Section~\ref{sec:fixed-mesh-ldp-rigorous}
establishes the small-mass path LDP for fixed $h$ and identifies its
Poisson control action.  Section~\ref{sec:collision-limit-rigorous} studies
the behavior of this action as the mesh is refined at a fixed velocity
cutoff and proves the $\Gamma$-liminf and recovery properties on biased
regular paths.

The two appendices collect material that clarifies the relation with the
continuum formal model without interrupting the finite-mesh proofs.
Appendix~\ref{app:poisson-spde-ldp-framework} recalls the weak-convergence
framework for Poisson-driven equations and applies it to the continuum
fluctuating Boltzmann equation.  Appendix~\ref{app:fluctuating-boltzmann-rate}
compares the resulting control rate with the collision-control and
Hamiltonian representations of the hard-sphere functional.

\section{Setup of a coarse-grained model and main results}
\label{sec:coarse-setup-main}

Throughout the paper, we write
$\mathbb R_{\geq0}:=[0,\infty)$ for the set of nonnegative real numbers.
We use the following path-space convention throughout this section.  If
$(E,d_E)$ is a Polish space, where the metric may be replaced by an
equivalent bounded metric, then $D([0,T];E)$ denotes the space of
right-continuous paths with left limits, endowed with the Skorokhod
$J_1$ topology.  Let $\Lambda_T$ be the set of strictly increasing
continuous bijections of $[0,T]$ onto itself.  A sequence
$x_n\in D([0,T];E)$ converges to $x\in D([0,T];E)$ in the $J_1$ topology
if and only if there exist $\lambda_n\in\Lambda_T$ such that
\begin{equation}\label{definition-Skorokhod-J1-convergence}
 \sup_{t\in[0,T]}|\lambda_n(t)-t|\longrightarrow0,
 \qquad
 \sup_{t\in[0,T]}d_E\bigl(x_n(\lambda_n(t)),x(t)\bigr)
 \longrightarrow0.
\end{equation}
See \cite[Chapter~3, Section~5]{EthierKurtz} for the definition and the
corresponding Polish-space theory.

Let $\mathcal M_+(\mathbb T^d\times\mathbb R^d)$ be the finite nonnegative Borel measures,
equipped with narrow convergence, and put
\begin{equation}\label{boltzmann-path-space}
 \mathscr X_T=D([0,T];\mathcal M_+(\mathbb T^d\times\mathbb R^d)),
\end{equation}
which is Polish under the preceding $J_1$ topology.  We write
$\mathcal M_{2,+}(\mathbb T^d\times\mathbb R^d)$ for the subspace with finite second velocity
moment and use the stronger moment topology only when it is explicitly
stated. Smooth tests
are taken from $C_c^\infty(\mathbb T^d\times\mathbb R^d)$. For fluctuation fields we use
\begin{equation*}
 \mathcal S(\mathbb T^d\times\mathbb R^d)
 =C^\infty(\mathbb T^d;\mathcal S(\mathbb R^d)),
 \qquad \mathcal S'(\mathbb T^d\times\mathbb R^d)=\mathcal S(\mathbb T^d\times\mathbb R^d)^*.
\end{equation*}
The notation $\Delta\varphi$ always denotes \eqref{collision-increment}.

All random measures below are defined on a filtered probability space
satisfying the usual conditions.  Let $(\mathsf E,\mathcal E)$ be a
measurable space and let $\nu$ be a
sigma-finite measure on it.  A Poisson random measure $\mathcal N$ on
$[0,T]\times\mathsf E$ with intensity $dt\,\nu(de)$ is an independently
scattered integer-valued random measure such that, whenever
$A\in\mathcal E$ and $t\nu(A)<\infty$,
\begin{equation*}
 \mathcal N((0,t]\times A)
 \quad\text{has the Poisson distribution with mean }t\nu(A).
\end{equation*}
Counts on disjoint measurable sets are independent.  Its compensated
Poisson random measure is
\begin{equation*}
 \widetilde{\mathcal N}(dt,de)
 :=\mathcal N(dt,de)-dt\,\nu(de).
\end{equation*}
Here $dt\,\nu(de)$ is the compensator of $\mathcal N$, whereas
$\widetilde{\mathcal N}$ is the compensated random measure.  The displayed
equality first defines a signed random measure on sets of finite intensity.
If $G$ is predictable with respect to the usual augmentation of the
filtration generated by $\mathcal N$ and any initial randomness, and
\begin{equation*}
 \mathbb E\int_0^T\int_{\mathsf E}|G(t,e)|^2\,\nu(de)\,dt<\infty,
\end{equation*}
then the integral
\begin{equation*}
 M_t=\int_0^t\int_{\mathsf E}G(s,e)\,
 \widetilde{\mathcal N}(ds,de)
\end{equation*}
is defined by $L^2$ approximation.  It is a square-integrable martingale and
satisfies
\begin{align*}
 \mathbb E|M_t|^2
 &=\mathbb E\int_0^t\int_{\mathsf E}|G(s,e)|^2\,\nu(de)\,ds,\\
 \langle M\rangle_t
 &=\int_0^t\int_{\mathsf E}|G(s,e)|^2\,\nu(de)\,ds,\\
 [M]_t
 &=\int_0^t\int_{\mathsf E}|G(s,e)|^2\,\mathcal N(ds,de).
\end{align*}

Let $\mathcal N$ now be a
Poisson random measure on
$[0,T]\times\mathsf E\times[0,\infty)$ with intensity
$dt\,\nu(de)\,du$, and let $r_t(e)$ be a nonnegative predictable function.
Then
\begin{equation*}
 \mathcal N^r(dt,de)
 :=\int_0^\infty\mathbf1_{\{u\leq r_t(e)\}}
 \mathcal N(dt,de,du)
\end{equation*}
has predictable compensator $r_t(e)\,dt\,\nu(de)$, and its compensated
measure is
\begin{align}\label{predictable-acceptance-compensation}
 \widetilde{\mathcal N}^r(dt,de)
 &: =\mathcal N^r(dt,de)-r_t(e)\,dt\,\nu(de)\notag\\
 &=\int_0^\infty\mathbf1_{\{u\leq r_t(e)\}}
 \widetilde{\mathcal N}(dt,de,du).
\end{align}
The entropy density for a change of Poisson intensity is
\begin{equation}\label{poisson-entropy-density}
 \ell(q)=q\log q-q+1,
 \qquad
 \ell^*(r)=e^r-1.
\end{equation}

\subsection{Definition of a coarse-grained model}
\label{subsec:def-coarse-grained-model}
We impose the capacity constraint by coarse graining phase space.  Two
parameters enter the construction and play different roles.  We write
$h=(h_x,h_v)$, where $h_x$ is the spatial mesh size and $h_v$ is the maximal
diameter of a velocity cell, whereas
$\varepsilon>0$ is the mass carried by one particle (or, equivalently, one
mass quantum).  When the total mass is normalized to one, we assume that
$\varepsilon^{-1}\in\mathbb N$.  To keep the state space finite,
we also choose a velocity cutoff $R(h)<\infty$, with $R(h)\uparrow\infty$ as
$h\downarrow0$ (that is, as $h_x,h_v\downarrow0$), and set
\begin{equation*}
 \mathbb D_h=\mathbb T^d\times B_{R(h)},\qquad
 B_h(c)=B(c)\mathbf1_{\{|v|\vee|v_*|\vee|v'|\vee|v_*'|\leq R(h)\}}.
\end{equation*}
Here $c=(x,v,v_*,\sigma)\in\mathcal C$ is the collision configuration
introduced above, $(v',v_*')$ are the corresponding post-collisional
velocities, and $B(c)=B(v-v_*,\sigma)$ is the collision kernel.  Thus
$B_h$ suppresses every collision for which at least one pre- or
post-collisional velocity lies outside $B_{R(h)}$.  This truncation
preserves mass and the collision invariants and becomes invisible on every
fixed compact velocity set as $h\downarrow0$.  The local integrability
assumption \eqref{collision-kernel-assumption} guarantees that every
cell-integrated collision coefficient defined below is finite.

For fixed $h$, assume that $h_x=N_x^{-1}$ for some
$N_x\in\mathbb N$.  Identifying $\mathbb T^d$ with
$\mathbb R^d/\mathbb Z^d$, let
$\pi:\mathbb R^d\to\mathbb T^d$ be the quotient map and, for
$i=(i_1,\ldots,i_d)\in\{0,\ldots,N_x-1\}^d$, define the half-open cell
\begin{equation*}
 X_i^h:=\pi\left(\prod_{r=1}^d
 [i_rh_x,(i_r+1)h_x)\right).
\end{equation*}
The family $(X_i^h)_i$ is a periodic spatial grid: its indices are understood
modulo $N_x$ in each coordinate, so a neighbor across one boundary of the
fundamental domain re-enters through the opposite boundary.  Choose
$N_v(h)\in\mathbb N$ and a measurable partition
$(V_j^h)_{j=1}^{N_v(h)}$ of $B_{R(h)}$, up to Lebesgue-null sets, such that
\begin{equation*}
 B_{R(h)}=\mathop{\dot\bigcup}_{j=1}^{N_v(h)}V_j^h,
 \qquad
 0<|V_j^h|<\infty,
 \qquad
 \operatorname{diam}(V_j^h)\leq h_v.
\end{equation*}
The phase-space index set and its associated product cells are
\begin{equation*}
 \mathcal I_h
 :=\{0,\ldots,N_x-1\}^d\times\{1,\ldots,N_v(h)\},
 \qquad
 E_{ij}^h:=X_i^h\times V_j^h,
 \quad (i,j)\in\mathcal I_h.
\end{equation*}
Thus an index $a\in\mathcal I_h$ can be written as $a=(i,j)$ and labels one
phase-space cell $E_a^h=E_{ij}^h$.  For $r\in\{1,\ldots,d\}$, we denote the $r$th
coordinate vector in physical space by
$\mathbf e_r^{\,x}\in\mathbb R^d$.

The state variable is the c\`adl\`ag vector
\begin{equation*}
 m^{\varepsilon,h}(t)=(m_a(t))_{a\in\mathcal I_h}
 \in(\varepsilon\mathbb N_0)^{\mathcal I_h}.
\end{equation*}
For every cell label $q\in\mathcal I_h$, define
$e_q\in\mathbb R^{\mathcal I_h}$ componentwise by
\begin{equation*}
 (e_q)_\alpha:=\mathbf1_{\{\alpha=q\}}
 =
 \begin{cases}
  1,&\alpha=q,\\
  0,&\alpha\neq q,
 \end{cases}
 \qquad \alpha\in\mathcal I_h.
\end{equation*}
Thus $(e_q)_{q\in\mathcal I_h}$ is the canonical basis of the cell-mass
state space.  In particular, the symbols $e_a,e_b,e_c,e_d$ below are the
basis vectors corresponding to the four cell labels $a,b,c,d$.  Their jump combinations are explicitly
\begin{align*}
 \bigl[\varepsilon(e_b-e_a)\bigr]_\alpha
 &=\varepsilon\bigl(\mathbf1_{\{\alpha=b\}}
                    -\mathbf1_{\{\alpha=a\}}\bigr),\\
 \bigl[\varepsilon(e_c+e_d-e_a-e_b)\bigr]_\alpha
 &=\varepsilon\bigl(\mathbf1_{\{\alpha=c\}}
 +\mathbf1_{\{\alpha=d\}}
 -\mathbf1_{\{\alpha=a\}}
 -\mathbf1_{\{\alpha=b\}}\bigr).
\end{align*}
By contrast, $\mathbf e_r^{\,x}\in\mathbb R^d$ moves the spatial grid
index from $i$ to $i\pm\mathbf e_r^{\,x}$.
Here $m_a(t)$ is the mass in $E_a^h$ at time $t$, so that
$m_a(t)/\varepsilon$ is the number of mass quanta in that cell.  Set
\begin{equation}\label{coarse-boltzmann-reconstruction}
 \chi_a^h=|E_a^h|^{-1}\mathbf1_{E_a^h},
 \qquad
 f_t^{\varepsilon,h}=\sum_{a\in\mathcal I_h}m_a(t)\chi_a^h,
\end{equation}
and extend $f_t^{\varepsilon,h}$ by zero outside $\mathbb D_h$.  Since
$\int_{\mathbb D_h}\chi_a^h\,dx\,dv=1$, the coefficients in
\eqref{coarse-boltzmann-reconstruction} are precisely the cell masses:
\begin{equation*}
 \int_{E_a^h}f_t^{\varepsilon,h}(x,v)\,dx\,dv=m_a(t),
 \qquad a\in\mathcal I_h.
\end{equation*}

We first introduce the transport jumps, using the standard conservative
upwind choice throughout the paper.  Write
$q^+:=\max\{q,0\}$ and $q^-:=\max\{-q,0\}$.
Let $v_r$ be the $r$th component of $v=(v_1,\ldots,v_d)$.  With the periodic convention,
$X_{i+\mathbf e_r^{\,x}}^h$ and $X_{i-\mathbf e_r^{\,x}}^h$ are the two
neighbours of $X_i^h$ in the $r$th spatial direction.  Set
\begin{align}
 T_{(i,j),(i+\mathbf e_r^{\,x},j)}^h
 &=\frac{1}{h_x|V_j^h|}\int_{V_j^h}(v_r)^+\,dv,
 &T_{(i,j),(i-\mathbf e_r^{\,x},j)}^h
 &=\frac{1}{h_x|V_j^h|}\int_{V_j^h}(v_r)^-\,dv,
 \qquad r=1,\ldots,d.
 \label{coarse-upwind-transport-coefficients}
\end{align}
These are the only nonzero transport coefficients.  If
$a=(i,j)$ and $b=(i',j')$, then $T_{ab}^h=0$ unless $j'=j$ and, for some
$r\in\{1,\ldots,d\}$,
\begin{equation*}
 i'=i+\mathbf e_r^{\,x}
 \qquad\text{or}\qquad
 i'=i-\mathbf e_r^{\,x}.
\end{equation*}
Here the spatial indices are understood modulo $N_x$.  Thus a transport
jump keeps the velocity cell fixed and moves the mass only to a periodic
nearest-neighbour spatial cell; $T_{ab}^h$ is its single-particle rate.

This upwind choice is natural for a pure-jump realization of free
transport.  In each spatial direction, the decomposition
$v_r=(v_r)^+-(v_r)^-$ represents the signed transport velocity by two
nonnegative jump rates.  A mass quantum with positive $r$th velocity moves
from $X_i^h$ to $X_{i+\mathbf e_r^{\,x}}^h$, while a negative $r$th
velocity moves it to $X_{i-\mathbf e_r^{\,x}}^h$.  The compensator of
these jumps therefore gives the conservative upwind discretization of
$-v\cdot\nabla_xf$, and the nonnegative coefficients in
\eqref{coarse-upwind-transport-coefficients} can serve as Poisson jump
rates.  To see the resulting upwind operator explicitly, consider one
space dimension and set
\begin{equation*}
 \bar v_{j,+}^h
 :=\frac{1}{|V_j^h|}\int_{V_j^h}v^+\,dv,
 \qquad
 \bar v_{j,-}^h
 :=\frac{1}{|V_j^h|}\int_{V_j^h}v^-\,dv.
\end{equation*}
The transport drift in cell $(i,j)$ is then
\begin{equation*}
 \frac{\bar v_{j,+}^h}{h_x}
 \bigl(m_{(i-1,j)}-m_{(i,j)}\bigr)
 +\frac{\bar v_{j,-}^h}{h_x}
 \bigl(m_{(i+1,j)}-m_{(i,j)}\bigr).
\end{equation*}
If $V_j^h\subset(0,\infty)$ and
$\bar v_j^h:=|V_j^h|^{-1}\int_{V_j^h}v\,dv>0$, this becomes
\begin{equation*}
 -\bar v_j^h
 \frac{m_{(i,j)}-m_{(i-1,j)}}{h_x}.
\end{equation*}
If $V_j^h\subset(-\infty,0)$, then $\bar v_j^h<0$, and the drift becomes
\begin{equation*}
 -\bar v_j^h
 \frac{m_{(i+1,j)}-m_{(i,j)}}{h_x}.
\end{equation*}
Thus the value in the upstream cell is used in either direction, and the
general formula above is the usual conservative upwind discretization.

A collision channel is an ordered quadruple $\gamma=(a,b;c,d)$: its
occurrence removes one mass quantum from each of the incoming cells $a,b$
and adds one to each of the outgoing cells $c,d$.  Write its four cell
indices as
\begin{equation*}
 a=(i_a,j),\qquad b=(i_b,k),\qquad
 c=(i_c,l),\qquad d=(i_d,m).
\end{equation*}
Collisions are local in $x$.  If the four spatial indices coincide,
$i_a=i_b=i_c=i_d=:i$, we define
\begin{align}
 K_\gamma^h
 :=\frac{1}{|X_i^h||V_j^h||V_k^h|}
 \int_{V_j^h\times V_k^h\times\mathbb S^{d-1}}
 &B_h(v,v_*,\sigma)
 \mathbf1_{V_l^h}(v')\mathbf1_{V_m^h}(v_*')
 \,dv\,dv_*\,d\sigma.
 \label{riemann-collision-quadrature}
\end{align}
If the four spatial indices are not all equal, we set $K_\gamma^h=0$.
Thus a collision can remove and add mass only within one and the same
spatial cell $X_i^h$.  Hence
$K_\gamma^h$ is explicitly determined by the collision kernel, the cutoff,
and the mesh: it integrates all collisions whose incoming velocities lie in
$V_j^h,V_k^h$ and whose outgoing velocities lie in $V_l^h,V_m^h$.
In particular, \eqref{riemann-collision-quadrature} is an exact cell
integration rather than an additional pointwise quadrature assumption.  It
inherits particle-exchange symmetry from $B_h$, while microreversibility
gives, with
$K_{i;jk\to lm}^h:=K_{((i,j),(i,k);(i,l),(i,m))}^h$,
\begin{align}
 |V_j^h||V_k^h|K_{i;jk\to lm}^h
 &=\frac1{|X_i^h|}
 \int_{\mathbb R^{2d}\times\mathbb S^{d-1}}
 B_h(v,v_*,\sigma)
 \mathbf1_{V_j^h}(v)\mathbf1_{V_k^h}(v_*)
 \mathbf1_{V_l^h}(v')\mathbf1_{V_m^h}(v_*')
 \,dv\,dv_*\,d\sigma\notag\\
 &=\frac1{|X_i^h|}
 \int_{\mathbb R^{2d}\times\mathbb S^{d-1}}
 B_h(\widetilde v,\widetilde v_*,\widetilde\sigma)
 \mathbf1_{V_l^h}(\widetilde v)\mathbf1_{V_m^h}(\widetilde v_*)
 \mathbf1_{V_j^h}(\widetilde v')\mathbf1_{V_k^h}(\widetilde v_*')
 \,d\widetilde v\,d\widetilde v_*\,d\widetilde\sigma\notag\\
 &=|V_l^h||V_m^h|K_{i;lm\to jk}^h.
 \label{volume-weighted-collision-microreversibility}
\end{align}
The first and last equalities follow directly from
\eqref{riemann-collision-quadrature}, after inserting the indicators of the
two incoming velocity cells.  For the middle equality, we use the
pre/post-collisional change of variables
\begin{equation*}
 (\widetilde v,\widetilde v_*,\widetilde\sigma)
 =\left(v',v_*',\frac{v-v_*}{|v-v_*|}\right),
 \qquad v\ne v_*.
\end{equation*}
Under the collision map \eqref{collision-map}, the post-collisional
velocities generated from
$(\widetilde v,\widetilde v_*,\widetilde\sigma)$ are
$(\widetilde v',\widetilde v_*')=(v,v_*)$.  The change of variables
preserves $dv\,dv_*\,d\sigma$, and the second identity in
\eqref{collision-kernel-symmetries} shows that the collision kernel is
unchanged.  The cutoff in the definition of $B_h$ is also unchanged because
it treats the four velocities $v,v_*,v',v_*'$ symmetrically.  This proves
\eqref{volume-weighted-collision-microreversibility}.
Its Riemann sums converge on smooth compactly supported profiles.  Exact
discrete momentum and energy conservation are not required for the
fixed-mesh theory; on compact velocity sets the corresponding consistency
errors vanish as $h_x+h_v\downarrow0$.

Figure~\ref{fig:coarse-transport-collision} summarizes the two kinds of
jumps.  It shows a one-dimensional projection of the phase-space mesh; in
higher dimension the same construction is applied in each spatial direction
and to the multidimensional velocity cells.
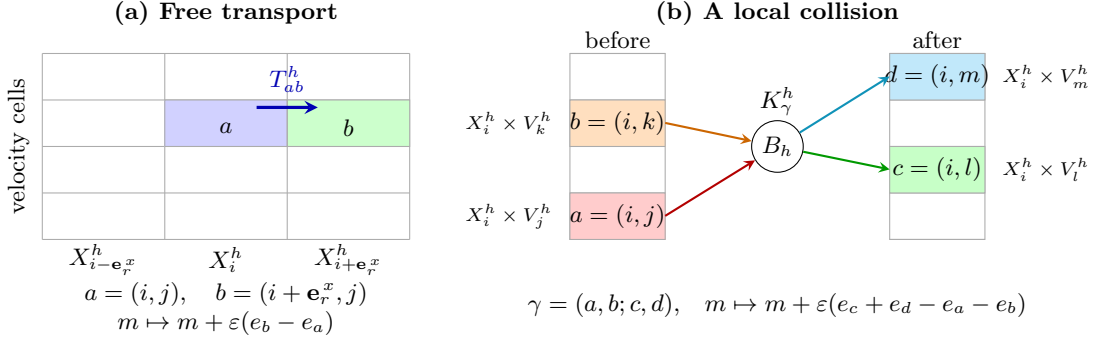
\begin{figure}[H]
\centering
\begin{tikzpicture}[x=0.90cm,y=0.82cm,>=stealth,
                    every node/.style={font=\small}]
 % Transport panel
 \node[font=\small\bfseries] at (2.7,4.15) {(a) Free transport};
 \foreach \x in {0,1.8,3.6,5.4}
   \draw[gray!65] (\x,0.5)--(\x,3.5);
 \foreach \y in {0.5,1.25,2,2.75,3.5}
   \draw[gray!65] (0,\y)--(5.4,\y);
 \fill[blue!18] (1.8,2) rectangle (3.6,2.75);
 \fill[green!20] (3.6,2) rectangle (5.4,2.75);
 \draw[gray!65] (1.8,2) rectangle (3.6,2.75);
 \draw[gray!65] (3.6,2) rectangle (5.4,2.75);
 \node at (2.7,2.30) {$a$};
 \node at (4.5,2.30) {$b$};
 \draw[->,very thick,blue!70!black]
   (3.15,2.63)--(4.05,2.63)
   node[midway,above=2pt] {$T_{ab}^h$};
 \node at (0.9,0.18) {$X_{i-\mathbf e_r^{\,x}}^h$};
 \node at (2.7,0.18) {$X_i^h$};
 \node at (4.5,0.18) {$X_{i+\mathbf e_r^{\,x}}^h$};
 \node[rotate=90] at (-0.35,2) {velocity cells};
 \node[align=center] at (2.7,-0.62)
   {$a=(i,j),\quad b=(i+\mathbf e_r^{\,x},j)$\\
    $m\mapsto m+\varepsilon(e_b-e_a)$};

 % Collision panel
 \begin{scope}[xshift=0.5cm]
 \node[font=\small\bfseries] at (10.25,4.15) {(b) A local collision};
 \node at (7.9,3.72) {before};
 \node at (12.6,3.72) {after};
 \foreach \y in {0.5,1.25,2,2.75,3.5}{
   \draw[gray!65] (7.2,\y)--(8.6,\y);
   \draw[gray!65] (11.9,\y)--(13.3,\y);
 }
 \draw[gray!65] (7.2,0.5) rectangle (8.6,3.5);
 \draw[gray!65] (11.9,0.5) rectangle (13.3,3.5);
 \fill[red!20] (7.2,0.5) rectangle (8.6,1.25);
 \fill[orange!25] (7.2,2) rectangle (8.6,2.75);
 \fill[green!22] (11.9,1.25) rectangle (13.3,2);
 \fill[cyan!22] (11.9,2.75) rectangle (13.3,3.5);
 \draw[gray!65] (7.2,0.5) rectangle (8.6,1.25);
 \draw[gray!65] (7.2,2) rectangle (8.6,2.75);
 \draw[gray!65] (11.9,1.25) rectangle (13.3,2);
 \draw[gray!65] (11.9,2.75) rectangle (13.3,3.5);
 \node at (7.9,0.88) {$a=(i,j)$};
 \node at (7.9,2.38) {$b=(i,k)$};
 \node at (12.6,1.63) {$c=(i,l)$};
 \node at (12.6,3.13) {$d=(i,m)$};
 \node[draw,circle,inner sep=2pt] (col) at (10.25,2) {$B_h$};
 \draw[->,thick,red!70!black] (8.6,0.88)--(col);
 \draw[->,thick,orange!80!black] (8.6,2.38)--(col);
 \draw[->,thick,green!60!black] (col)--(11.9,1.63);
 \draw[->,thick,cyan!70!black] (col)--(11.9,3.13);
 \node[fill=white,inner sep=1pt] at (10.25,2.73) {$K_\gamma^h$};
 \node[font=\scriptsize,anchor=east] at (7.08,0.88)
   {$X_i^h\times V_j^h$};
 \node[font=\scriptsize,anchor=east] at (7.08,2.38)
   {$X_i^h\times V_k^h$};
 \node[font=\scriptsize,anchor=west] at (13.42,1.63)
   {$X_i^h\times V_l^h$};
 \node[font=\scriptsize,anchor=west] at (13.42,3.13)
   {$X_i^h\times V_m^h$};
 \node[align=center] at (10.25,-0.55)
   {$\gamma=(a,b;c,d)$,\quad
     $m\mapsto m+\varepsilon(e_c+e_d-e_a-e_b)$};
 \end{scope}
\end{tikzpicture}
\caption{The two elementary jumps of the coarse model.  In (a), for a
positive $r$th velocity component, one mass quantum moves to the neighbouring
spatial cell without changing its velocity cell; the reverse arrow is used
for a negative component.  In (b), both grids have the same fixed spatial
cell $X_i^h$, while each horizontal strip represents one velocity cell.
If $a=b$, the collision rate contains
$m_a(m_a-\varepsilon)$, which is positive only when
$m_a/\varepsilon\geq2$; hence two available mass quanta are removed from
$E_a^h$ when the collision occurs.}
\label{fig:coarse-transport-collision}
\end{figure}

We denote the finite sets of active transport edges and collision channels by
\begin{equation*}
 \mathcal E_h=\{(a,b)\in\mathcal I_h^2:T_{ab}^h>0\},
 \qquad
 \Gamma_h=\{(a,b;c,d)\in\mathcal I_h^4:K_{(a,b;c,d)}^h>0\}
\end{equation*}
respectively.

The fixed-mesh analysis below uses only the finiteness of the channel sets
and the nonnegativity of $T_{ab}^h$ and $K_\gamma^h$.  To define the evolution of the cell masses, we first assign an instantaneous
rate to every active channel as a function of the current state.  For
$m=(m_a)_{a\in\mathcal I_h}$, define
\begin{align}
 \lambda_{ab}^{\varepsilon,h}(m)
 &=\varepsilon^{-1}m_aT_{ab}^h,
 \label{coarse-transport-rate}\\
 \lambda_\gamma^{\varepsilon,h}(m)
 &=\frac{K_\gamma^h}{2\varepsilon}
 m_a\bigl(m_b-\varepsilon\mathbf1_{\{a=b\}}\bigr).
 \label{coarse-boltzmann-collision-rate}
\end{align}
These are nonnegative rate functions on the state space.  They become
stochastic intensities when evaluated along the random path $m(t-)$.  Since
every mass quantum has size
$\varepsilon$, the number of available quanta in cell $a$ is
$m_a/\varepsilon$.  Hence \eqref{coarse-transport-rate} is the
single-quantum transport rate $T_{ab}^h$ multiplied by the number
$m_a/\varepsilon$ of available quanta in the departure cell.  Likewise,
\begin{equation*}
 \frac{m_a}{\varepsilon}
 \left(\frac{m_b}{\varepsilon}-\mathbf1_{\{a=b\}}\right)
\end{equation*}
counts ordered pairs of distinct incoming quanta in cells $a$ and $b$.
Multiplying this number by the rate $\varepsilon K_\gamma^h/2$ assigned to
each ordered pair gives exactly \eqref{coarse-boltzmann-collision-rate}.
The subtraction excludes the collision of a quantum with itself when
$a=b$, while the factor $1/2$ is the usual Boltzmann symmetry factor.

We realize the channel counts by random time changes.  On the filtered
probability space fixed above, let
\begin{equation*}
 (P_{ab}^{\rm tr})_{(a,b)\in\mathcal E_h},
 \qquad (P_\gamma^{\rm col})_{\gamma\in\Gamma_h},
\end{equation*}
be mutually independent unit-rate Poisson processes, independent of the
initial state.  Abbreviating $m(t)=m^{\varepsilon,h}(t)$, define the
physical channel counts jointly with the state process by
\begin{align}
 P_{ab}^{\varepsilon,h}(t)
 &=P_{ab}^{\rm tr}\left(
 \int_0^t\lambda_{ab}^{\varepsilon,h}(m(s-))\,ds\right),
 \label{coarse-transport-counts}\\
 P_\gamma^{\varepsilon,h}(t)
 &=P_\gamma^{\rm col}\left(
 \int_0^t\lambda_\gamma^{\varepsilon,h}(m(s-))\,ds\right).
 \label{coarse-collision-counts}
\end{align}
By the random-time-change theorem,
$P_{ab}^{\varepsilon,h}$ and $P_\gamma^{\varepsilon,h}$ have predictable
intensities
\begin{equation*}
 \lambda_{ab}^{\varepsilon,h}(m(t-))
 \qquad\text{and}\qquad
 \lambda_\gamma^{\varepsilon,h}(m(t-)),
\end{equation*}
respectively.  The left limit makes the intensities predictable.  The
time-changed counts are coupled through their common dependence on $m$,
even though the reference processes are independent.

Together with the time-change equations
\eqref{coarse-transport-counts}--\eqref{coarse-collision-counts}, the
coordinate vectors $(e_a)_{a\in\mathcal I_h}$ and a prescribed initial state
$m^{\varepsilon,h}(0)$ define the cell-mass process through the coupled
random-time-change equation
\begin{align}
 dm^{\varepsilon,h}(t)
 &=\varepsilon\sum_{(a,b)\in\mathcal E_h}(e_b-e_a)
 \,dP_{ab}^{\varepsilon,h}(t)\notag\\
 &\quad+\varepsilon\sum_{\gamma=(a,b;c,d)\in\Gamma_h}
 (e_c+e_d-e_a-e_b)
 \,dP_\gamma^{\varepsilon,h}(t).
 \label{coarse-boltzmann-model}
\end{align}
Each transport jump therefore changes $m$ by
$\varepsilon(e_b-e_a)$, and each collision jump changes it by
$\varepsilon(e_c+e_d-e_a-e_b)$.  Both jump vectors have zero coordinate
sum, so every jump preserves the total mass $\sum_a m_a$.
For $M\in\varepsilon\mathbb N_0$, denote the fixed-mass lattices and their
finite-mass union by
\begin{align}
 \mathsf S_{\varepsilon,h}(M)
 &=\left\{m\in(\varepsilon\mathbb N_0)^{\mathcal I_h}:
 \sum_{a\in\mathcal I_h}m_a=M\right\},\notag\\
 \mathsf S_{\varepsilon,h}^{\rm fin}
 &=\bigcup_{M\in\varepsilon\mathbb N_0}
 \mathsf S_{\varepsilon,h}(M),
 \qquad
 \mathsf S_{\varepsilon,h}=\mathsf S_{\varepsilon,h}(1).
 \label{intro-coarse-state-space}
\end{align}

\begin{definition}[Solution of the coarse-grained system]
\label{def:coarse-grained-solution}
Fix $\varepsilon>0$, $h>0$, and $T>0$, and let
$m_0^{\varepsilon,h}$ be an $\mathcal F_0$-measurable random variable with
values in $\mathsf S_{\varepsilon,h}^{\rm fin}$, independent of the
reference Poisson processes.  A \emph{global strong solution} of the
coarse-grained fluctuating Boltzmann system on $[0,T]$ is an adapted
c\`adl\`ag process
\begin{equation*}
 m^{\varepsilon,h}:[0,T]\times\Omega
 \longrightarrow\mathsf S_{\varepsilon,h}^{\rm fin}
\end{equation*}
which is constructed from $m_0^{\varepsilon,h}$ and the prescribed family
$(P_{ab}^{\rm tr},P_\gamma^{\rm col})$ and has the following properties.
For every $(a,b)\in\mathcal E_h$ and $\gamma\in\Gamma_h$, the cumulative
clocks
\begin{equation*}
 \int_0^t\lambda_{ab}^{\varepsilon,h}
       (m^{\varepsilon,h}(s-))\,ds,
 \qquad
 \int_0^t\lambda_\gamma^{\varepsilon,h}
       (m^{\varepsilon,h}(s-))\,ds
\end{equation*}
are finite almost surely for every $t\in[0,T]$, and, almost surely,
simultaneously for all $t\in[0,T]$,
\begin{align}
 m^{\varepsilon,h}(t)
 ={}&m_0^{\varepsilon,h}
 +\varepsilon\sum_{(a,b)\in\mathcal E_h}(e_b-e_a)
 P_{ab}^{\rm tr}\left(
 \int_0^t\lambda_{ab}^{\varepsilon,h}
       (m^{\varepsilon,h}(s-))\,ds\right)\notag\\
 &+\varepsilon\sum_{\gamma=(a,b;c,d)\in\Gamma_h}
 (e_c+e_d-e_a-e_b)
 P_\gamma^{\rm col}\left(
 \int_0^t\lambda_\gamma^{\varepsilon,h}
       (m^{\varepsilon,h}(s-))\,ds\right).
 \label{coarse-boltzmann-integral-equation}
\end{align}
Here ``strong'' means that the solution is constructed pathwise from the
given initial state and reference Poisson processes, rather than that its
sample paths possess additional regularity.  Pathwise uniqueness means that
two such solutions with the same initial state and the same reference
Poisson processes are indistinguishable.  If
$m_0^{\varepsilon,h}\in\mathsf S_{\varepsilon,h}(M)$ almost surely, then a
solution takes values in that fixed-mass state space.  The associated
piecewise constant density is the reconstruction
$f_t^{\varepsilon,h}=\sum_am_a^{\varepsilon,h}(t)\chi_a^h$ from
\eqref{coarse-boltzmann-reconstruction}; it is not an additional unknown.
\end{definition}

Equation~\eqref{coarse-boltzmann-integral-equation} is the integral form of
\eqref{coarse-transport-counts}--\eqref{coarse-boltzmann-model}.  The
finiteness of the cumulative clocks, together with the finiteness of
$\mathcal E_h$ and $\Gamma_h$, rules out an accumulation of jumps on
$[0,T]$.  Theorem~\ref{thm:intro-coarse-wellposedness-rigorous} below shows
that this global solution exists, is pathwise unique, and remains in the
nonnegative fixed-mass lattice for every finite time horizon.

Define the compensated counting processes by
\begin{align*}
 \widetilde P_{ab}^{\varepsilon,h}(t)
 &=P_{ab}^{\varepsilon,h}(t)
   -\int_0^t\lambda_{ab}^{\varepsilon,h}(m(s-))\,ds,\\
 \widetilde P_\gamma^{\varepsilon,h}(t)
 &=P_\gamma^{\varepsilon,h}(t)
   -\int_0^t\lambda_\gamma^{\varepsilon,h}(m(s-))\,ds.
\end{align*}
Applying the reconstruction map
$m\mapsto\sum_a m_a\chi_a^h$ to
\eqref{coarse-boltzmann-model} sends the coordinate vector $e_a$ to
$\chi_a^h$.  Before compensation, the density process
$f_t^{\varepsilon,h}=\sum_a m_a(t)\chi_a^h$ satisfies
\begin{align*}
 df_t^{\varepsilon,h}
 &=\varepsilon\sum_{(a,b)\in\mathcal E_h}
 (\chi_b^h-\chi_a^h)\,dP_{ab}^{\varepsilon,h}(t)\\
 &\quad+\varepsilon\sum_{\gamma=(a,b;c,d)\in\Gamma_h}
 (\chi_c^h+\chi_d^h-\chi_a^h-\chi_b^h)
 \,dP_\gamma^{\varepsilon,h}(t).
\end{align*}
By the definition of the compensated counts,
\begin{align*}
 dP_{ab}^{\varepsilon,h}(t)
 &=d\widetilde P_{ab}^{\varepsilon,h}(t)
   +\lambda_{ab}^{\varepsilon,h}(m(t-))\,dt,\\
 dP_\gamma^{\varepsilon,h}(t)
 &=d\widetilde P_\gamma^{\varepsilon,h}(t)
   +\lambda_\gamma^{\varepsilon,h}(m(t-))\,dt.
\end{align*}
The transport part of the compensator is therefore
\begin{align*}
 &\varepsilon\sum_{(a,b)\in\mathcal E_h}
 \lambda_{ab}^{\varepsilon,h}(m(t-))
 (\chi_b^h-\chi_a^h)\,dt\\
 &\qquad=\sum_{(a,b)\in\mathcal E_h}
 m_a(t-)T_{ab}^h(\chi_b^h-\chi_a^h)\,dt,
\end{align*}
where we used \eqref{coarse-transport-rate}.  Likewise,
\eqref{coarse-boltzmann-collision-rate} gives the collision compensator
\begin{align*}
 &\varepsilon\sum_{\gamma=(a,b;c,d)\in\Gamma_h}
 \lambda_\gamma^{\varepsilon,h}(m(t-))
 (\chi_c^h+\chi_d^h-\chi_a^h-\chi_b^h)\,dt\\
 &\qquad=\frac12\sum_{\gamma=(a,b;c,d)\in\Gamma_h}K_\gamma^h
 m_a(t-)\bigl(m_b(t-)-\varepsilon\mathbf1_{\{a=b\}}\bigr)
 (\chi_c^h+\chi_d^h-\chi_a^h-\chi_b^h)\,dt.
\end{align*}
Combining these two compensator terms with the two remaining martingale
terms yields the equivalent density equation
\begin{align}
 df_t^{\varepsilon,h}
 &=\Bigg[\mathcal T_h^*f_t^{\varepsilon,h}
 +\frac12\sum_{\gamma=(a,b;c,d)\in\Gamma_h}K_\gamma^h
 m_a(t-)\bigl(m_b(t-)-\varepsilon\mathbf1_{\{a=b\}}\bigr)
 (\chi_c^h+\chi_d^h-\chi_a^h-\chi_b^h)\Bigg]dt
 \notag\\
 &\quad+\varepsilon\sum_{(a,b)\in\mathcal E_h}(\chi_b^h-\chi_a^h)
 \,d\widetilde P_{ab}^{\varepsilon,h}(t)\notag\\
 &\quad+\varepsilon\sum_{\gamma=(a,b;c,d)\in\Gamma_h}
 (\chi_c^h+\chi_d^h-\chi_a^h-\chi_b^h)
 \,d\widetilde P_\gamma^{\varepsilon,h}(t),
 \label{coarse-boltzmann-density-equation}
\end{align}
where, if $f=\sum_{a\in\mathcal I_h}m_a\chi_a^h$, the transport drift is
\begin{equation}
 \mathcal T_h^*f
 =\sum_{(a,b)\in\mathcal E_h}m_aT_{ab}^h(\chi_b^h-\chi_a^h).
 \label{coarse-transport-drift}
\end{equation}

We next give a formal term-by-term comparison between the discrete and
continuum models, whose purpose is only to explain how the discrete drift and
martingale terms correspond to their continuum counterparts, without
assuming well-posedness of the continuum Poissonian equation or claiming
that the discrete process converges to it.  In the formal calculations below,
$h\to0$ means that $h_x,h_v\to0$ and $R(h)\to\infty$.  For a smooth test
function $\varphi$, write
\begin{equation*}
 \varphi_a^h:=\int_{\mathbb D_h}\varphi\chi_a^h\,dx\,dv,
 \qquad a\in\mathcal I_h,
\end{equation*}
and, for every collision channel $\gamma=(a,b;c,d)\in\Gamma_h$, define
\begin{equation*}
 \Delta_h\varphi_\gamma
 :=\varphi_c^h+\varphi_d^h-\varphi_a^h-\varphi_b^h.
\end{equation*}
Testing \eqref{coarse-boltzmann-density-equation} against $\varphi$ gives
the exact identity
\begin{align}
 d\langle f_t^{\varepsilon,h},\varphi\rangle
 &=\Bigg[
 \sum_{(a,b)\in\mathcal E_h}m_aT_{ab}^h
       (\varphi_b^h-\varphi_a^h)
 +\frac12\sum_{\gamma=(a,b;c,d)\in\Gamma_h}
 K_\gamma^hm_a(m_b-\varepsilon\mathbf1_{\{a=b\}})
       \Delta_h\varphi_\gamma\Bigg]dt\notag\\
 &\quad+\varepsilon\sum_{(a,b)\in\mathcal E_h}
 (\varphi_b^h-\varphi_a^h)\,
 d\widetilde P_{ab}^{\varepsilon,h}(t)\notag\\
 &\quad+\varepsilon\sum_{\gamma=(a,b;c,d)\in\Gamma_h}
 \Delta_h\varphi_\gamma\,
 d\widetilde P_\gamma^{\varepsilon,h}(t).
 \label{coarse-boltzmann-weak-comparison}
\end{align}
We denote the transport and collision martingales in the second and third
lines by $M_{\mathrm{tr}}^{\varepsilon,h}(\varphi)$ and
$M_{\mathrm{col}}^{\varepsilon,h}(\varphi)$, respectively.  We compare the
transport drift, collision drift, and martingale terms in this order with
their continuum counterparts.

First, the first sum in the drift is the upwind discretization of free
transport.  For cell masses
$m_a^h=\int_{E_a^h}f(x,v)\,dx\,dv$ obtained from a smooth profile, the
difference quotient in the definition of $T_{ab}^h$ gives formally
\begin{equation}
 \sum_{(a,b)\in\mathcal E_h}m_a^hT_{ab}^h
 (\varphi_b^h-\varphi_a^h)
 \xrightarrow[h\to0]{}
 \int_{\mathbb T^d\times\mathbb R^d}
 f(x,v)v\cdot\nabla_x\varphi(x,v)\,dx\,dv.
 \label{formal-transport-consistency}
\end{equation}
Thus $\mathcal T_h^*$ corresponds to the term
$-v\cdot\nabla_xf$ in the strong equation.

Second, consider the collision drift.  A channel
$\gamma=((i,j),(i,k);(i,l),(i,m))$ is the cell version of
a collision configuration $c=(x,v,v_*,\sigma)$.  By the definition of
$K_\gamma^h$, the second drift sum is a Riemann sum over all collision
configurations.  More precisely, if representative points $z_a\in E_a^h$
are chosen, then
$\varphi_a^h=\langle\chi_a^h,\varphi\rangle\to\varphi(z_a)$ as $h\to0$.
Thus, for a channel representing $c=(x,v,v_*,\sigma)$, $\Delta_h\varphi_\gamma
 \rightarrow
 \Delta\varphi(c)$.
 
For the auxiliary cell masses
$m_a^h=\int_{E_a^h}f(x,v)\,dx\,dv$ associated with a smooth density $f$,
the leading part of the collision drift therefore satisfies the formal
Riemann-sum limit
\begin{align}
 \frac12\sum_{\gamma=(a,b;c,d)\in\Gamma_h}
 K_\gamma^hm_a^hm_b^h
 \Delta_h\varphi_\gamma
 \xrightarrow[h\to0]{}
 \frac12\int_{\mathcal C}B(v-v_*,\sigma)f(x,v)f(x,v_*)
 \Delta\varphi(c)\,dc.
 \label{formal-collision-consistency}
\end{align}
This limit concerns only the term with coefficient $m_a^hm_b^h$.  At fixed
$\varepsilon$, the full collision drift also contains
\begin{equation*}
 -\frac{\varepsilon}{2}
 \sum_{\substack{\gamma=(a,b;c,d)\in\Gamma_h\\a=b}}
 K_\gamma^hm_a^h\Delta_h\varphi_\gamma,
\end{equation*}
which cannot be discarded merely by sending $h\to0$.  However, this term
does not contribute to the large-deviation rate function derived in
Section~\ref{sec:fixed-mesh-ldp-rigorous}, because that result first fixes
$h$ and then lets $\varepsilon\downarrow0$, under which the correction is of
order $\varepsilon$ and vanishes.  Moreover, for fixed $\varepsilon$, the
cell averages $m_a^h$ of a smooth density eventually fail
to belong to $\varepsilon\mathbb N_0$ as the mesh is refined.  Thus
\eqref{formal-collision-consistency} identifies the continuum form of the
leading collision term, but it is not a convergence statement for the full
fixed-$\varepsilon$ collision drift.

Finally, compare the two jump mechanisms and their martingale parts.  A
transport event on the edge $(a,b)$ changes the reconstructed density by
\begin{equation*}
 \varepsilon(\chi_b^h-\chi_a^h).
\end{equation*}
The associated transport martingale has no counterpart in the continuum
Poissonian Boltzmann model, as it is introduced by the randomized upwind
realization of free transport.  Applying the compensated-Poisson bracket
formula stated above, an event in the channel
$(a,b)$ produces the jump
$\varepsilon(\varphi_b^h-\varphi_a^h)$, while its predictable intensity is
$\lambda_{ab}^{\varepsilon,h}(m(t-))
=\varepsilon^{-1}m_a(t-)T_{ab}^h$.  Different channels are driven by
independent reference Poisson processes and have zero predictable
cross-variation.  The bracket formula for compensated counting processes
therefore gives
\begin{align*}
 d\langle M_{\mathrm{tr}}^{\varepsilon,h}(\varphi)\rangle_t
 &=\varepsilon^2\sum_{(a,b)\in\mathcal E_h}
 (\varphi_b^h-\varphi_a^h)^2
 \lambda_{ab}^{\varepsilon,h}(m(t-))\,dt\\
 &=\varepsilon\sum_{(a,b)\in\mathcal E_h}
 m_a(t-)T_{ab}^h(\varphi_b^h-\varphi_a^h)^2\,dt.
\end{align*}
To estimate the last sum, suppose that
$a=(i,j)$ and $b=(i\pm\mathbf e_r^{\,x},j)$.  The two cells then differ only
by a spatial translation of length $h_x$.  Hence, for smooth $\varphi$,
\begin{equation*}
 |\varphi_b^h-\varphi_a^h|
 \leq h_x\|\partial_{x_r}\varphi\|_\infty.
\end{equation*}
Moreover, since $|v|\leq R(h)$ on the truncated velocity domain, the
definition of the upwind rates implies
\begin{align*}
 T_{(i,j),(i+\mathbf e_r^{\,x},j)}^h
 +T_{(i,j),(i-\mathbf e_r^{\,x},j)}^h=\frac{1}{h_x|V_j^h|}
 \int_{V_j^h}|v_r|\,dv
 \leq\frac{R(h)}{h_x}.
\end{align*}
Consequently, on the fixed-mass state space
$\sum_a m_a(t)=M$,
\begin{align*}
 d\langle M_{\mathrm{tr}}^{\varepsilon,h}(\varphi)\rangle_t
 &\leq \varepsilon R(h) h_x
 \left(\sum_{r=1}^d\|\partial_{x_r}\varphi\|_\infty^2\right)
 \sum_{a\in\mathcal I_h}m_a(t-)\,dt\\
 &=\varepsilon M R(h) h_x
 \left(\sum_{r=1}^d\|\partial_{x_r}\varphi\|_\infty^2\right)\,dt
 =O(\varepsilon R(h)h_x)\,dt.
\end{align*}
It therefore disappears under the condition $R(h)h_x\to0$.

A collision event in the channel $\gamma=(a,b;c,d)$ changes the reconstructed
density by
\begin{equation*}
 \varepsilon(\chi_c^h+\chi_d^h-\chi_a^h-\chi_b^h).
\end{equation*}
Since the normalized cell density $\chi_a^h$ is the coarse replacement of a
point mass, this jump is the cell version of $\varepsilon J_c$ in
\eqref{collision-jump-measure}.  After testing against $\varphi$, its size is $\varepsilon\Delta_h\varphi_\gamma$, whose leading continuum form is $\varepsilon\Delta\varphi(c)$.  The
independent reference collision clocks
then play the role of the cell counts of the continuum Poisson random
measure.  The channel intensity is
\begin{equation*}
 \frac{K_\gamma^h}{2\varepsilon}
 m_a(m_b-\varepsilon\mathbf1_{\{a=b\}}).
\end{equation*}
Its leading term, obtained by replacing the last factor by $m_am_b$, has the
same cell-integrated form as the continuum compensator
\begin{equation*}
 \frac{1}{2\varepsilon}B(v-v_*,\sigma)f(x,v)f(x,v_*)
 \,dt\,dx\,dv\,dv_*\,d\sigma.
\end{equation*}
At fixed $\varepsilon$, however, the self-pair correction remains in both the
intensity and the predictable quadratic variation.  The latter is exactly
\begin{align*}
 d\langle M_{\mathrm{col}}^{\varepsilon,h}(\varphi)\rangle_t
 &=\frac{\varepsilon}{2}\sum_{\gamma=(a,b;c,d)\in\Gamma_h}
 K_\gamma^hm_a(m_b-\varepsilon\mathbf1_{\{a=b\}})
 (\Delta_h\varphi_\gamma)^2\,dt.
\end{align*}
The part containing $m_am_b$ has the formal continuum expression
\begin{equation*}
 \frac{\varepsilon}{2}\int_{\mathcal C}
 Bff_*(\Delta\varphi)^2\,dc\,dt,
\end{equation*}
which is the bracket of the collision noise in
\eqref{poisson-fluctuating-boltzmann-weak}; the full fixed-$\varepsilon$
bracket is not identified with this expression by mesh refinement alone.

The preceding calculation identifies the continuum forms of the leading
collision drift and quadratic variation, while keeping the finite-occupancy
correction explicit at fixed $\varepsilon$.
Section~\ref{sec:fixed-mesh-ldp-rigorous}
establishes the path large-deviation principle for the discrete model at a
fixed mesh, while Section~\ref{sec:collision-limit-rigorous} proves that,
at a fixed velocity cutoff and on biased regular paths, its rate function
recovers the corresponding continuum rate function as the mesh is refined.

\subsection{Main results}
\label{subsec:main-results-rigorous}

Fix $h>0$. For a transport edge $\kappa=(a,b)\in\mathcal E_h$ and a collision
channel $\gamma=(a,b;c,d)\in\Gamma_h$, set
\begin{equation}\label{intro-unified-channel-data}
\begin{aligned}
 z_\kappa&=e_b-e_a,
 &\beta_\kappa(m)&=m_aT_{ab}^h,\\
 z_\gamma&=e_c+e_d-e_a-e_b,
 &\beta_\gamma(m)&=\frac12K_\gamma^hm_am_b,
\end{aligned}
\end{equation}
where $(e_a)_{a\in\mathcal I_h}$ is the canonical basis of
$\mathbb R^{\mathcal I_h}$.  Write
$\mathcal K_h:=\mathcal E_h\sqcup\Gamma_h$.  In sums over
$\mathcal K_h$, $\kappa$ denotes a generic channel; when
$\kappa=\gamma\in\Gamma_h$, we set $z_\kappa=z_\gamma$ and
$\beta_\kappa=\beta_\gamma$.  The vector $z_\kappa$ is the change in the
mass vector caused by one transition, and $\beta_\kappa$ is its limiting
nominal intensity after the factor $\varepsilon^{-1}$ has been removed.
If $m_0^h\in(0,\infty)^{\mathcal I_h}$ has total mass one, define, for
$m\in D([0,T];\mathbb R_{\geq0}^{\mathcal I_h})$,
\begin{align}\label{intro-fixed-h-rigorous-rate}
 I_h^{m_0^h}(m)=
 \inf_{q=(q_\kappa)_{\kappa\in\mathcal K_h}}
 \Bigg\{&
 \int_0^T\sum_{\kappa\in\mathcal K_h}
 \beta_\kappa(m(t))\ell(q_\kappa(t))\,dt:
 q:[0,T]\to\mathbb R_{\geq0}^{\mathcal K_h}
 \text{ is measurable},\notag\\
 & (\beta_\kappa(m)q_\kappa)_{\kappa\in\mathcal K_h}
 \in L^1(0,T;\mathbb R^{\mathcal K_h}),\quad
 \eqref{intro-fixed-h-rigorous-continuity}\text{ holds}
 \Bigg\}.
\end{align}
Here $\ell(r)=r\log r-r+1$.  The function $q_\kappa(t)$ is the multiplier of the
nominal intensity $\beta_\kappa(m(t))$ on channel $\kappa$; the uncontrolled
dynamics corresponds to $q_\kappa\equiv1$ and has zero cost.  The entire
vector-valued multiplier
$q=(q_\kappa)_{\kappa\in\mathcal K_h}$ is constrained by the single
skeleton equation
\begin{equation}\label{intro-fixed-h-rigorous-continuity}
 \dot m(t)=\sum_{\kappa\in\mathcal K_h}
 \beta_\kappa(m(t))q_\kappa(t)z_\kappa,
 \quad\text{for a.e. }t\in[0,T],
 \qquad m(0)=m_0^h.
\end{equation}
The infimum over the empty set is $+\infty$; in particular,
$I_h^{m_0^h}(m)=+\infty$ unless $m$ is absolutely continuous and has the
prescribed initial value.
This is the finite-channel form of the Poisson control cost in
\eqref{BCD-control-cost}: integration over the discrete channel variable
becomes the sum in \eqref{intro-fixed-h-rigorous-rate}, while the nominal
channel activity supplies the weight $\beta_\kappa(m(t))$.

\iffalse
The deterministic coarse equation and its Gaussian covariance are described
by
\begin{align}
 F_h(m)
 &=\sum_{\kappa\in\mathcal E_h}\beta_\kappa(m)z_\kappa
   +\sum_{\gamma\in\Gamma_h}\beta_\gamma(m)z_\gamma,
 \label{intro-fixed-mesh-drift}\\
 A_h(m)
 &=\sum_{\kappa\in\mathcal E_h}\beta_\kappa(m)
   z_\kappa z_\kappa^{\mathsf T}
   +\sum_{\gamma\in\Gamma_h}\beta_\gamma(m)
     z_\gamma z_\gamma^{\mathsf T}.
 \label{intro-fixed-mesh-covariance}
\end{align}
Let $\bar m^h$ solve
\begin{equation}\label{intro-fixed-mesh-deterministic-equation}
 \dot{\bar m}^h(t)=F_h(\bar m^h(t)),
 \qquad \bar m^h(0)=\bar m_0^h.
\end{equation}
Set
\begin{equation*}
 \bar f_t^h=\sum_{a\in\mathcal I_h}\bar m_a^h(t)\chi_a^h.
\end{equation*}
The fluctuation vector and the associated density field are
\begin{equation}\label{intro-fixed-mesh-fluctuation-field}
 Z_t^{\varepsilon,h}
 =\varepsilon^{-1/2}\bigl(m^{\varepsilon,h}(t)-\bar m^h(t)\bigr),
 \qquad
 Y_t^{\varepsilon,h}
 =\varepsilon^{-1/2}\bigl(f_t^{\varepsilon,h}-\bar f_t^h\bigr)
 =\sum_{a\in\mathcal I_h}Z_a^{\varepsilon,h}(t)\chi_a^h.
\end{equation}
\fi

\begin{theorem}[Finite-mesh well-posedness]
\label{thm:intro-coarse-wellposedness-rigorous}
Let $h>0$ and $\varepsilon>0$.  Suppose that $m(0)$ is an
$\mathcal F_0$-measurable random variable with values in
$\mathsf S_{\varepsilon,h}^{\rm fin}$ and is independent of the reference
Poisson processes.  Then the coarse equation
\eqref{coarse-boltzmann-model} has a pathwise unique nonexplosive strong
c\`adl\`ag solution.  Almost surely, the solution remains nonnegative and
\begin{equation*}
 \sum_{a\in\mathcal I_h}m_a(t)
 =\sum_{a\in\mathcal I_h}m_a(0),
 \qquad t\geq0.
\end{equation*}
%In particular, if $\varepsilon^{-1}\in\mathbb N$ and
%$m(0)\in\mathsf S_{\varepsilon,h}$ almost surely, then the total mass remains
%equal to one.
\end{theorem}

\begin{theorem}[Finite-mesh path LDP]
\label{thm:intro-fixed-mesh-ldp-rigorous}
Fix $h>0$.  Suppose that $m_0^h$ is strictly positive and that
$\sum_{a\in\mathcal I_h}m_{0,a}^h=1$. Suppose further that
$m_0^{\varepsilon,h}\in\mathsf S_{\varepsilon,h}$ and
$m_0^{\varepsilon,h}\to m_0^h$ as $\varepsilon\downarrow0$ along values
satisfying $\varepsilon^{-1}\in\mathbb N$.
Then $m^{\varepsilon,h}$ satisfies on
$D([0,T];\mathbb R_{\geq0}^{\mathcal I_h})$ an LDP with speed
$\varepsilon^{-1}$ and good rate \eqref{intro-fixed-h-rigorous-rate}.
The continuous reconstruction
$m\mapsto\sum_am_a\chi_a^h$ yields the corresponding density-path LDP by
the contraction principle on $D([0,T];L^1(\mathbb D_h))$ with its
Skorokhod $J_1$ topology.
Furthermore, let $m_*^h\in(0,\infty)^{\mathcal I_h}$ and suppose that the
initial cell masses are given by
\begin{equation*}
 m_a^{\varepsilon,h}(0)=\varepsilon Z_a^\varepsilon,
 \qquad
 Z_a^\varepsilon\sim\operatorname{Pois}
 \left(\frac{m_{*,a}^h}{\varepsilon}\right),
 \qquad a\in\mathcal I_h,
\end{equation*}
where the random variables $(Z_a^\varepsilon)_{a\in\mathcal I_h}$ are
independent of one another and of the dynamical Poisson noises.  Theorem
\ref{thm:intro-coarse-wellposedness-rigorous} then gives a well-defined
process with values in $\mathsf S_{\varepsilon,h}^{\rm fin}$, whose laws
satisfy an LDP with the good rate
\begin{equation}\label{intro-fixed-h-poisson-rate}
 \mathcal I_h^{\rm Poi}(m)
 =\sum_{a\in\mathcal I_h}m_{*,a}^h
 \ell\left(\frac{m_a(0)}{m_{*,a}^h}\right)+I_h(m),
\end{equation}
where $I_h(m)$ is defined by the right-hand side of
\eqref{intro-fixed-h-rigorous-rate}, subject to the skeleton equation in
\eqref{intro-fixed-h-rigorous-continuity} but without prescribing $m(0)$.
\end{theorem}

The final result concerns a fixed velocity cutoff and uses a more regular
kernel than the fixed-mesh theory.  Let $R<\infty$, set
\begin{equation*}
 \mathbb D_R=\mathbb T^d\times B_R,
 \qquad
 \mathcal C_R=\{c=(x,v,v_*,\sigma)\in\mathcal C:
 |v|\vee|v_*|\vee|v'|\vee|v_*'|<R\},
\end{equation*}
We assume that
\begin{equation}\label{intro-fixed-cutoff-kernel-assumption}
 \mathcal B_R\in
 C_b^{0,1}(\mathbb R^d\times\mathbb R^d\times\mathbb S^{d-1}),
 \qquad \mathcal B_R\geq0,
 \qquad
 \operatorname{supp}\mathcal B_R
 \Subset\left\{(v,v_*,\sigma):
 |v|\vee|v_*|\vee|v'|\vee|v_*'|<R\right\}.
\end{equation}
We also require $\mathcal B_R$ to satisfy the two identities in
\eqref{collision-kernel-symmetries}, with
$B(v-v_*,\sigma)$ replaced throughout by
$\mathcal B_R(v,v_*,\sigma)$.  These are
the fixed-cutoff kernel assumptions used in the mesh-limit result.  They are
stronger than \eqref{collision-kernel-assumption}: the latter is sufficient
for the fixed-mesh process, but not for the consistency estimates under mesh
refinement.  For a density path $g$ with endpoint traces define
\begin{align}
 \mathcal I_{\mathcal B_R}^{\rm Ham}(g)
 :=\sup_{r\in C^1([0,T];W^{2,\infty}(\mathbb D_R))}\Bigg\{
 &\langle g_T,r_T\rangle-\langle g_0,r_0\rangle
 -\int_0^T\langle g_t,\partial_tr_t+v\cdot\nabla_xr_t\rangle\,dt
 \notag\\
 &-\frac12\int_0^T\int_{\mathcal C_R}
 \mathcal B_Rg_tg_{t,*}(e^{\Delta r_t}-1)\,dc\,dt\Bigg\},
 \label{intro-fixed-cutoff-dynamical-action}\\
 I_{\mathcal B_R}(g)
 ={}&\int_{\mathbb D_R}
 \left[g_0\log\frac{g_0}{f_0}-g_0+f_0\right]dz
 +\mathcal I_{\mathcal B_R}^{\rm Ham}(g).
 \label{intro-fixed-cutoff-full-action}
\end{align}
The value is $+\infty$ when the displayed quantities or endpoint traces are
not defined.  Here a controlled pair $(g,q)$ is regular if $g$ and $q$ are
continuously differentiable in time, have bounded Lipschitz derivatives on
the truncated domains, satisfy $0<c\leq g\leq C$ and $0\leq q\leq C$, and
$g$ solves \eqref{controlled-boltzmann-equation} with control $q$ and
collision kernel $\mathcal B_R$.  The path $g$ is called biased regular if
there exist a control $q$ and a function
$p\in C^1([0,T];W^{2,\infty}(\mathbb D_R))$ such that $(g,q)$ is a regular
controlled pair and $q(t,c)=e^{\Delta p_t(c)}$ for
$(t,c)\in[0,T]\times\mathcal C_R$. Proposition~\ref{prop:fixed-cutoff-biased-duality} then gives
\begin{equation}\label{intro-fixed-cutoff-biased-action}
 I_{\mathcal B_R}(g)
 =\int_{\mathbb D_R}
 \left[g_0\log\frac{g_0}{f_0}-g_0+f_0\right]dz
 +\frac12\int_0^T\int_{\mathcal C_R}
 \mathcal B_Rgg_*\ell(q)\,dc\,dt.
\end{equation}
For the mesh on $\mathbb D_R$, set
\begin{equation*}
 (\Pi_hf)_a=\int_{E_a^h}f(z)\,dz,
 \qquad
 \mathsf R_hm=\sum_{a\in\mathcal I_h}m_a\chi_a^h.
\end{equation*}
For a discrete path $m$ define
\begin{equation*}
 d_{\infty,R}^h(m,g)
 =\sup_{t\in[0,T]}
 \|\mathsf R_hm(t)-g_t\|_{L^\infty(\mathbb D_R)}.
\end{equation*}

\begin{theorem}[Continuum limit of the rate function]
\label{thm:intro-regular-path-rate-convergence}
Assume that $f_0$ is Lipschitz on $\mathbb D_R$ and satisfies
$0<c_0\leq f_0\leq C_0<\infty$.  Let $g$ be a
biased regular path, use the kernel $\mathcal B_R$ satisfying
\eqref{intro-fixed-cutoff-kernel-assumption} in the cell-integrated mesh
collision rates, and take $m_{*,a}^h=(\Pi_hf_0)_a$ in
\eqref{intro-fixed-h-poisson-rate}.  For every sequence $\widetilde m^h$ such
that $d_{\infty,R}^h(\widetilde m^h,g)\to0$ as
$h_x+h_v\downarrow0$,
\begin{equation}\label{intro-regular-path-rate-limit}
 \liminf_{h_x+h_v\downarrow0}
 \mathcal I_h^{\rm Poi}(\widetilde m^h)
 \geq I_{\mathcal B_R}(g).
\end{equation}
Conversely, there is a sequence $m^h$ such that
$d_{\infty,R}^h(m^h,g)\to0$ and
\begin{equation*}
 \lim_{h_x+h_v\downarrow0}\mathcal I_h^{\rm Poi}(m^h)
 =I_{\mathcal B_R}(g).
\end{equation*}
\end{theorem}

The lower bound and recovery sequence identify the continuum action locally
near every biased regular path.  Thus the coarse process reproduces the
candidate large-deviation cost associated with the collision dynamics of
the underlying particles.  Since $R$ is fixed, the limiting functional is
the action of the truncated dynamics; removing the cutoff would require
additional uniform control of the high-energy tails. Theorem~\ref{thm:intro-regular-path-rate-convergence} is proved in
Theorem~\ref{thm:regular-path-rate-convergence-rigorous}.

\section{Well-posedness of the coarse model}
\label{sec:coarse-wellposedness}

Fix $h>0$ and $\varepsilon>0$ satisfying $\varepsilon^{-1}\in\mathbb N$.
Recall from
\eqref{intro-coarse-state-space} the fixed-mass
lattices $\mathsf S_{\varepsilon,h}(M)$,
$M\in\varepsilon\mathbb N_0$, the normalized state space
$\mathsf S_{\varepsilon,h}=\mathsf S_{\varepsilon,h}(1)$, and their union
$\mathsf S_{\varepsilon,h}^{\rm fin}$.  Every fixed-mass layer is finite:
after division by $\varepsilon$, it is
the set of weak compositions of $M/\varepsilon$ into
$|\mathcal I_h|$ parts.  In particular,
$\mathsf S_{\varepsilon,h}^{\rm fin}
=(\varepsilon\mathbb N_0)^{\mathcal I_h}$.

At fixed $(\varepsilon,h)$, the mass equation is a finite-dimensional pure
jump SDE, and its restriction to each fixed-mass layer is a finite-state
continuous-time Markov chain.  Its well-posedness includes two ingredients.  Lemma~\ref{lem:coarse-total-intensity} bounds the total
state-dependent jump intensity uniformly on every fixed-mass layer, which
prevents infinitely many jumps from accumulating in finite time.
Lemma~\ref{lem:coarse-invariant-state-space} shows that a channel can have
positive intensity only when its incoming cells contain enough mass, and
that every permitted jump preserves nonnegativity and total mass.  These
two ingredients yield the pathwise construction, nonexplosion, and pathwise
uniqueness in
Theorem~\ref{thm:coarse-boltzmann-wellposedness}.

Corollary~\ref{cor:coarse-random-finite-mass} extends the construction from
unit total mass to arbitrary finite, possibly random, total mass.  Finally,
Proposition~\ref{prop:coarse-boltzmann-generator} identifies the generator,
the associated martingale decomposition, and the predictable brackets of the
tested density process. 

\begin{lemma}\label{lem:coarse-total-intensity}
For $m\in\mathsf S_{\varepsilon,h}^{\rm fin}$, define the total
instantaneous jump intensity by
\begin{equation}\label{coarse-total-jump-intensity}
 \Lambda_{\varepsilon,h}(m)
 :=\sum_{(a,b)\in\mathcal E_h}
       \lambda_{ab}^{\varepsilon,h}(m)
   +\sum_{\gamma\in\Gamma_h}
       \lambda_\gamma^{\varepsilon,h}(m).
\end{equation}
For every $M\in\varepsilon\mathbb N_0$, there is a finite deterministic
constant $C_{\varepsilon,h}(M)$ such that
\begin{equation}\label{coarse-uniform-total-rate}
 \sup_{m\in\mathsf S_{\varepsilon,h}(M)}
 \Lambda_{\varepsilon,h}(m)
 \leq C_{\varepsilon,h}(M)<\infty.
\end{equation}
One may take
\begin{equation*}
 C_{\varepsilon,h}(M)
 =\frac1\varepsilon\left(MT_h+\frac12M^2K_h\right),
 \qquad
 T_h:=\max_{a\in\mathcal I_h}\sum_{b:(a,b)\in\mathcal E_h}T_{ab}^h,
 \quad
 K_h:=\sum_{\gamma\in\Gamma_h}K_\gamma^h.
\end{equation*}
\end{lemma}

\begin{proof}
The index sets $\mathcal I_h$, $\mathcal E_h$, and $\Gamma_h$ are finite,
and each coefficient $T_{ab}^h$ and $K_\gamma^h$ is finite.  Therefore the
constants $T_h$ and $K_h$ defined in the statement are finite.  Fix
$M\in\varepsilon\mathbb N_0$ and let
$m\in\mathsf S_{\varepsilon,h}(M)$.  Then
$m_a\geq0$ and $\sum_am_a=M$.  Hence
\begin{align*}
 \sum_{(a,b)\in\mathcal E_h}\lambda_{ab}^{\varepsilon,h}(m)
 &=\frac1\varepsilon\sum_{a\in\mathcal I_h}m_a
   \sum_{b:(a,b)\in\mathcal E_h}T_{ab}^h
 \leq\frac{MT_h}{\varepsilon}.
\end{align*}
Because every $m_a$ is an integer multiple of $\varepsilon$,
\begin{equation*}
 0\leq m_a\bigl(m_b-\varepsilon\mathbf1_{\{a=b\}}\bigr)
 \leq m_am_b\leq M^2.
\end{equation*}
Consequently,
\begin{align*}
 \sum_{\gamma\in\Gamma_h}\lambda_\gamma^{\varepsilon,h}(m)
 &\leq\frac{M^2}{2\varepsilon}
       \sum_{\gamma\in\Gamma_h}K_\gamma^h
 =\frac{M^2K_h}{2\varepsilon}.
\end{align*}
Adding the two estimates proves \eqref{coarse-uniform-total-rate}.
\end{proof}

\begin{lemma}
\label{lem:coarse-invariant-state-space}
Let $M\in\varepsilon\mathbb N_0$ and
$m\in\mathsf S_{\varepsilon,h}(M)$.  If a transport edge $(a,b)$ has
$\lambda_{ab}^{\varepsilon,h}(m)>0$, then
\begin{equation*}
 m+\varepsilon(e_b-e_a)\in\mathsf S_{\varepsilon,h}(M).
\end{equation*}
If a collision channel $\gamma=(a,b;c,d)$ has
$\lambda_\gamma^{\varepsilon,h}(m)>0$, then
\begin{equation*}
 m+\varepsilon(e_c+e_d-e_a-e_b)
 \in\mathsf S_{\varepsilon,h}(M).
\end{equation*}
Thus every jump permitted by the state-dependent intensities preserves
nonnegativity, the mass lattice, and total mass.

Moreover, let $\eta=(\eta_a)_{a\in\mathcal I_h}$ satisfy
\begin{equation}\label{coarse-discrete-linear-invariant}
 \eta_b=\eta_a\quad\text{whenever }T_{ab}^h>0,
 \qquad
 \eta_c+\eta_d=\eta_a+\eta_b
 \quad\text{whenever }K_{(a,b;c,d)}^h>0.
\end{equation}
If $m'$ is obtained from $m$ by any transport or collision jump having
positive intensity at $m$, then
\begin{equation*}
 \sum_{i\in\mathcal I_h}\eta_i m_i'
 =\sum_{i\in\mathcal I_h}\eta_i m_i.
\end{equation*}
\end{lemma}

\begin{proof}
Consider first a transport jump from $a$ to $b$. Its rate is
$\varepsilon^{-1}m_aT_{ab}^h$, so a positive rate implies
$m_a\geq\varepsilon$.  After adding $\varepsilon(e_b-e_a)$, every coordinate
therefore remains in $\varepsilon\mathbb N_0$.  The sum of the coordinates is
unchanged.

For a collision channel $\gamma=(a,b;c,d)$ with $a\ne b$, positivity of
$\lambda_\gamma^{\varepsilon,h}(m)$ implies
$m_a,m_b\geq\varepsilon$.  Thus one quantum can be removed from both incoming
cells.  If $a=b$, positivity of the falling-factorial rate
\begin{equation*}
 \frac{K_\gamma^h}{2\varepsilon}m_a(m_a-\varepsilon)
\end{equation*}
and the fact that $m_a\in\varepsilon\mathbb N_0$ imply
$m_a\geq2\varepsilon$.  Thus cell $a$ contains at least two mass units,
each of size $\varepsilon$, and subtracting the two incoming units leaves
$m_a-2\varepsilon\geq0$.  The factor $m_a(m_a-\varepsilon)$ therefore
excludes the selection of the same mass unit twice, which is the
coarse-grained analogue of excluding self-collisions.  The same conclusion
holds when some incoming and outgoing cell labels coincide,
after the corresponding basis vectors are cancelled in the jump vector.
The collision jump preserves total mass because
$\sum_{i\in\mathcal I_h}(e_c+e_d-e_a-e_b)_i=0$.  Hence every jump maps
$\mathsf S_{\varepsilon,h}(M)$ into itself.

It remains to verify the last assertion.  If $m'$ is obtained from $m$ by a
transport jump from $a$ to $b$, then
\begin{equation*}
 \sum_{i\in\mathcal I_h}\eta_i(m_i'-m_i)
 =\varepsilon(\eta_b-\eta_a)=0.
\end{equation*}
If $m'$ is obtained from $m$ by the collision channel $(a,b;c,d)$, then
\begin{equation*}
 \sum_{i\in\mathcal I_h}\eta_i(m_i'-m_i)
 =\varepsilon(\eta_c+\eta_d-\eta_a-\eta_b)=0.
\end{equation*}
Thus the corresponding linear functional has the same value immediately
before and after every admissible jump.
\end{proof}

\begin{theorem}[Well-posedness and nonnegativity]
\label{thm:coarse-boltzmann-wellposedness}
For every $h>0$, every $\varepsilon>0$ satisfying
$\varepsilon^{-1}\in\mathbb N$, and every
$m(0)\in\mathsf S_{\varepsilon,h}$, the mass equation
\eqref{coarse-boltzmann-model} has a pathwise unique global strong solution
in the sense of Definition~\ref{def:coarse-grained-solution}.  Its
reconstruction \eqref{coarse-boltzmann-reconstruction} satisfies
\eqref{coarse-boltzmann-density-equation}.  Almost surely,
\begin{equation*}
 f_t^{\varepsilon,h}\geq0,
 \qquad
 \int_{\mathbb T^d\times\mathbb R^d} f_t^{\varepsilon,h}(z)\,dz=1
 \quad\text{for every }t\geq0.
\end{equation*}
More generally, for every weight
$\eta=(\eta_a)_{a\in\mathcal I_h}$ satisfying both conditions in
\eqref{coarse-discrete-linear-invariant},
\begin{equation*}
 \sum_{a\in\mathcal I_h}\eta_am_a(t)
 =\sum_{a\in\mathcal I_h}\eta_am_a(0)
 \qquad\text{for every }t\geq0
\end{equation*}
almost surely.  Thus every weight satisfying
\eqref{coarse-discrete-linear-invariant} defines a pathwise conserved linear
functional of the cell masses.  The constant choice $\eta_a=1$ satisfies
both conditions and gives the total-mass conservation displayed above.
\end{theorem}

\begin{proof}
Starting from $m(0)$, we use the independent reference Poisson processes in
\eqref{coarse-transport-counts}--\eqref{coarse-collision-counts} to define
the jump times and jump channels recursively.  Recall that
$\mathcal K_h=\mathcal E_h\sqcup\Gamma_h$, and write
$\kappa\in\mathcal K_h$ for a generic channel.  To distinguish channel
labels from the canonical basis vectors $(e_a)_{a\in\mathcal I_h}$, define
\begin{equation*}
 (P_\kappa^{\rm ref},\lambda_\kappa,z_\kappa)
 :=
 \begin{cases}
 \bigl(P_{ab}^{\rm tr},\lambda_{ab}^{\varepsilon,h},e_b-e_a\bigr),
 &\kappa=(a,b)\in\mathcal E_h,\\[1mm]
 \bigl(P_\gamma^{\rm col},\lambda_\gamma^{\varepsilon,h},
 e_c+e_d-e_a-e_b\bigr),
 &\kappa=\gamma=(a,b;c,d)\in\Gamma_h.
 \end{cases}
\end{equation*}
Here $P_\kappa^{\rm ref}$ is the independent unit-rate reference process
introduced before
\eqref{coarse-transport-counts}--\eqref{coarse-collision-counts}, whereas
$\lambda_\kappa$ is the state-dependent rate from
\eqref{coarse-transport-rate} or
\eqref{coarse-boltzmann-collision-rate}.  For a constructed portion of the
path, set
\begin{equation*}
 A_\kappa(t):=\int_0^t\lambda_\kappa(m(s-))\,ds
\end{equation*}
on that time interval.  The accumulated internal time $A_\kappa(t)$ is the
cumulative intensity that converts physical time into the time parameter of
the reference process.  Its instantaneous speed is
$\lambda_\kappa(m(t-))$, so the clock stops whenever this rate vanishes.
The corresponding time-changed process
\begin{equation*}
 P_\kappa^{\rm ref}(A_\kappa(t))
\end{equation*}
equals $P_{ab}^{\varepsilon,h}(t)$ for
$\kappa=(a,b)\in\mathcal E_h$ and $P_\gamma^{\varepsilon,h}(t)$ for
$\kappa=\gamma\in\Gamma_h$.  It has compensator $A_\kappa(t)$ and counts
the jumps of channel $\kappa$ up to physical time $t$.  The corresponding
increment of the mass vector is $\varepsilon z_\kappa$.

Suppose that the process has been constructed through its $n$th physical
jump time $\tau_n$, and write $m_n=m(\tau_n)$ for the state immediately
after that jump.  For every channel $\kappa\in\mathcal K_h$, the constructed
path on $[0,\tau_n]$ determines $A_\kappa(\tau_n)$.  Since
$P_\kappa^{\rm ref}$ is fixed independently of $m$, let
$S_{\kappa,n}$ be its smallest jump time strictly larger than
$A_\kappa(\tau_n)$.

Figure~\ref{fig:channel-random-time-change} illustrates the conversion from
internal time to physical time.  The candidate physical time shown in
panel~(b) is computed below.
\begin{figure}[H]
\centering
\begin{tikzpicture}[x=0.78cm,y=0.68cm,>=stealth,
                    every node/.style={font=\footnotesize}]
 % Reference process on the internal-time axis
 \node[anchor=west,font=\footnotesize\bfseries] at (0.2,6.85)
   {(a) Reference Poisson process};
 \draw[->] (0.8,4.70)--(10.7,4.70)
   node[right] {$u$ (internal time)};
 \draw[->] (0.8,4.70)--(0.8,6.55);
 \node[rotate=90] at (-0.72,5.65) {jump count};
 \node[anchor=east] at (0.70,5.38) {$j$};
 \node[anchor=east] at (0.70,6.08) {$j+1$};
 \draw[very thick,blue!70!black]
   (0.8,5.38)--(7.8,5.38)--(7.8,6.08)--(10.25,6.08);
 \draw[dashed,gray!80] (3.2,4.70)--(3.2,5.38);
 \draw[dashed,gray!80] (7.8,4.70)--(7.8,6.08);
 \draw (3.2,4.64)--(3.2,4.76);
 \draw (7.8,4.64)--(7.8,4.76);
 \node[below=2pt] at (3.2,4.70) {$A_\kappa(\tau_n)$};
 \node[below=2pt] at (7.8,4.70) {$S_{\kappa,n}$};
 \draw[<->] (3.2,4.98)--(7.8,4.98)
   node[midway,above=3pt,fill=white,inner sep=1.2pt]
   {$\Delta u_\kappa$};

 % The random time change
 \draw[->,gray!75] (9.35,4.35)--(9.35,3.72)
   node[midway,right] {$u=A_\kappa(t)$};

 % Time-changed count on the physical-time axis
 \node[anchor=west,font=\footnotesize\bfseries] at (0.2,3.38)
   {(b) Time-changed counting process};
 \draw[->] (0.8,1.20)--(10.7,1.20)
   node[right] {$t$ (physical time)};
 \draw[->] (0.8,1.20)--(0.8,3.05);
 \node[rotate=90] at (-0.72,2.15) {jump count};
 \node[anchor=east] at (0.70,1.88) {$j$};
 \node[anchor=east] at (0.70,2.58) {$j+1$};
 \draw[very thick,green!45!black]
   (0.8,1.88)--(7.8,1.88)--(7.8,2.58)--(10.25,2.58);
 \draw[dashed,gray!80] (3.2,1.20)--(3.2,1.88);
 \draw[dashed,gray!80] (7.8,1.20)--(7.8,2.58);
 \draw (3.2,1.14)--(3.2,1.26);
 \draw (7.8,1.14)--(7.8,1.26);
 \node[below=2pt] at (3.2,1.20) {$\tau_n$};
 \node[below=2pt] at (7.8,1.20) {$\tau_{\kappa,n+1}$};
 \draw[<->] (3.2,1.48)--(7.8,1.48)
   node[midway,above=3pt,fill=white,inner sep=1.2pt]
   {$\Delta t_\kappa$};
\end{tikzpicture}
\caption{Random time change for channel $\kappa$ while the state remains
$m_n$.  The upper and lower panels use internal time $u$ and physical time
$t$, respectively.  With
$\Delta u_\kappa=S_{\kappa,n}-A_\kappa(\tau_n)$ and
$\Delta t_\kappa=\tau_{\kappa,n+1}-\tau_n$, the time change gives
$\Delta u_\kappa=\lambda_\kappa(m_n)\Delta t_\kappa$.  }
\label{fig:channel-random-time-change}
\end{figure}
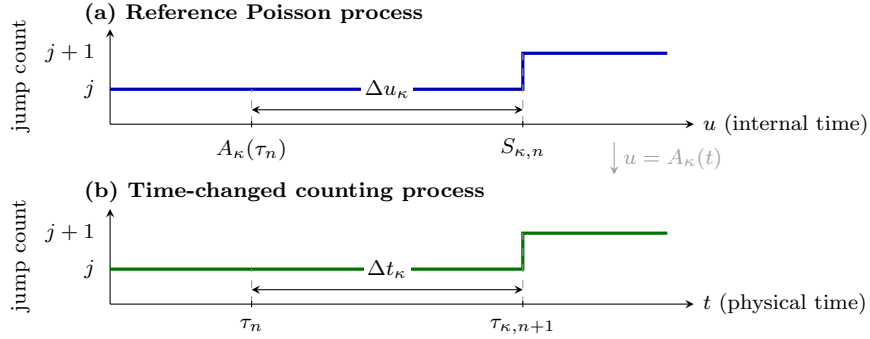

To determine the next channel, suppose that no channel jumps
between $\tau_n$ and a later physical time $t$.  The state is
then equal to $m_n$ throughout this interval, and hence the intensity of
channel $\kappa$ is the constant $\lambda_\kappa(m_n)$.  Its internal time
at physical time $t$ would therefore be
\begin{equation*}
 A_\kappa(\tau_n)+\lambda_\kappa(m_n)(t-\tau_n).
\end{equation*}
Channel $\kappa$ reaches its next reference Poisson point when this quantity
equals $S_{\kappa,n}$.  If $\lambda_\kappa(m_n)>0$, solving
\begin{equation*}
 A_\kappa(\tau_n)
 +\lambda_\kappa(m_n)(\tau_{\kappa,n+1}-\tau_n)
 =S_{\kappa,n}
\end{equation*}
gives its candidate physical jump time
\begin{equation*}
 \tau_{\kappa,n+1}
 =\tau_n+
 \frac{S_{\kappa,n}-A_\kappa(\tau_n)}{\lambda_\kappa(m_n)}.
\end{equation*}
When $\lambda_\kappa(m_n)=0$, no jump of this channel can occur while the
state is $m_n$, and we set $\tau_{\kappa,n+1}=+\infty$.

The next actual jump is the first among these competing candidate jumps:
\begin{equation*}
 \tau_{n+1}:=\min_{\kappa\in\mathcal K_h}\tau_{\kappa,n+1}.
\end{equation*}
For every $t<\tau_{n+1}$, no channel has reached its next reference Poisson
point, so the state remains equal to $m_n$.  If $\kappa_{n+1}$ attains the
minimum, then a jump
of channel $\kappa_{n+1}$ occurs at $\tau_{n+1}$ and the state is updated to
\begin{equation*}
 m_{n+1}=m_n+\varepsilon z_{\kappa_{n+1}}.
\end{equation*}
The independence of the reference Poisson processes implies that ties have
probability zero.  This recursion is the random-time-change construction of
the integral equation \eqref{coarse-boltzmann-integral-equation}.  It is
defined until
\begin{equation*}
 \tau_\infty:=\lim_{n\to\infty}\tau_n.
\end{equation*}
The event $\{\tau_\infty<\infty\}$ would mean that infinitely many jumps
accumulate in a finite time interval and is called explosion.

Lemma~\ref{lem:coarse-invariant-state-space} shows inductively that every
state produced by this construction belongs to
$\mathsf S_{\varepsilon,h}$.  Lemma~\ref{lem:coarse-total-intensity}, with
$M=1$, therefore bounds its total jump intensity by
$C_{\varepsilon,h}(1)$ at every step.  Let $J_t$ be the total number of
jumps by time $t$.  For $n\geq1$, the stopped counting process
$J_{t\wedge\tau_n}$ has compensator
\begin{equation*}
 \int_0^{t\wedge\tau_n}
 \Lambda_{\varepsilon,h}(m(s-))\,ds
 \leq C_{\varepsilon,h}(1)t.
\end{equation*}
It follows that
\begin{equation*}
 n\,\mathbb P(\tau_n\leq T)
 \leq\mathbb E[J_{T\wedge\tau_n}]
 \leq C_{\varepsilon,h}(1)T.
\end{equation*}
Letting $n\to\infty$ yields
$\mathbb P(\tau_\infty\leq T)=0$ for every $T<\infty$.  Thus
$\tau_\infty=\infty$ almost surely, and the construction gives a global
adapted c\`adl\`ag solution.

The same recursion gives pathwise uniqueness.  If two solutions
have the same initial state and the same reference Poisson processes, then
at each common jump time they have the same state, the same accumulated
internal clocks, and the same next unused jump time $S_{\kappa,n}$ for every
channel $\kappa\in\mathcal K_h$.  Hence their candidate times
$\tau_{\kappa,n+1}$ agree, and they make
the same next jump.  Induction shows that their paths coincide for all
times.  The recursive construction is a measurable adapted functional of
the initial state and the reference Poisson processes, and is therefore a
strong solution.

Finally, $\chi_a^h\geq0$ and $m_a(t)\geq0$ imply
$f_t^{\varepsilon,h}\geq0$. Since
$\int_{\mathbb T^d\times\mathbb R^d}\chi_a^h(z)\,dz=1$,
\begin{equation*}
 \int_{\mathbb T^d\times\mathbb R^d}f_t^{\varepsilon,h}(z)\,dz
 =\sum_{a\in\mathcal I_h}m_a(t)=1.
\end{equation*}
More generally, suppose that $\eta$ satisfies
\eqref{coarse-discrete-linear-invariant}.  The process is constant between
successive jump times, and Lemma~\ref{lem:coarse-invariant-state-space}
shows that the quantity
\begin{equation*}
 L_\eta(m(t)):=\sum_{a\in\mathcal I_h}\eta_am_a(t)
\end{equation*}
is unchanged at every jump.  Since nonexplosion gives only finitely many
jumps on each bounded time interval,
\begin{equation*}
 L_\eta(m(t))=L_\eta(m(0))
 \qquad\text{for every }t\geq0
\end{equation*}
almost surely.  Taking $\eta_a=1$ gives mass conservation.  For an
arbitrary admissible weight $\eta$, the same argument proves the linear
conservation law stated in the theorem.

Decomposing each accepted counting process into its compensator and
compensated part transforms
\eqref{coarse-boltzmann-model} into
\eqref{coarse-boltzmann-density-equation}, so the two displayed forms define
the same path.
\end{proof}

\begin{corollary}
\label{cor:coarse-random-finite-mass}
Fix $\varepsilon>0$ and $h>0$.  For every
$M\in\varepsilon\mathbb N_0$ and every deterministic initial state
$m_0\in\mathsf S_{\varepsilon,h}(M)$, the coupled random-time-change equation
\eqref{coarse-boltzmann-integral-equation}, driven by the reference Poisson
processes $(P_{ab}^{\rm tr})_{(a,b)\in\mathcal E_h}$ and
$(P_\gamma^{\rm col})_{\gamma\in\Gamma_h}$, has a pathwise unique global
strong solution.  Almost surely,
\begin{equation*}
 m(t)\in\mathsf S_{\varepsilon,h}(M)
 \qquad\text{for every }t\geq0.
\end{equation*}

More generally, let $m_0$ be an $\mathcal F_0$-measurable random variable
with values in $\mathsf S_{\varepsilon,h}^{\rm fin}$, independent of all
the reference Poisson processes, and define its random total mass by
\begin{equation*}
 M_0:=\sum_{a\in\mathcal I_h}m_{0,a}.
\end{equation*}
Then the same equation has a pathwise unique global strong solution and,
almost surely,
\begin{equation*}
 m(t)\in\mathsf S_{\varepsilon,h}(M_0)
 \qquad\text{for every }t\geq0.
\end{equation*}
In particular, this applies when
$m_{0,a}=\varepsilon N_a$, where $(N_a)_{a\in\mathcal I_h}$ are independent
Poisson random variables with finite means and are independent of the driving
Poisson processes.
\end{corollary}

\begin{proof}
For deterministic $m_0\in\mathsf S_{\varepsilon,h}(M)$, the proof of
Theorem~\ref{thm:coarse-boltzmann-wellposedness} applies without change on
the finite state space $\mathsf S_{\varepsilon,h}(M)$.  Lemma~\ref{lem:coarse-invariant-state-space}
keeps every successive state in this space, and
Lemma~\ref{lem:coarse-total-intensity} bounds the total jump intensity there
by
\begin{equation*}
 C_{\varepsilon,h}(M)
 =\varepsilon^{-1}\left(MT_h+\frac12M^2K_h\right)<\infty.
\end{equation*}
The same stopped-counting-process argument rules out explosion and gives a
global pathwise unique solution.

For random $m_0$, condition on its value.  Independence of $m_0$ from the
reference Poisson processes leaves their laws unchanged under this
conditioning, and every realization of $m_0$ belongs to one of the finite
state spaces $\mathsf S_{\varepsilon,h}(M_0)$.  Applying the deterministic
conclusion conditionally and then removing the conditioning proves the
claim.  Finally, because $\mathcal I_h$ is finite, a finite family of
Poisson random variables with finite means has finite sum almost surely, so
the last example satisfies the hypotheses.
\end{proof}

\begin{proposition}
\label{prop:coarse-boltzmann-generator}
For $M\in\varepsilon\mathbb N_0$ and bounded
$F$ on $\mathsf S_{\varepsilon,h}(M)$, the generator is
\begin{align}
 \mathcal A_{\varepsilon,h}F(m)
 &=\frac1\varepsilon\sum_{(a,b)\in\mathcal E_h}m_aT_{ab}^h
 [F(m+\varepsilon(e_b-e_a))-F(m)]\notag\\
 &\quad+\sum_{\gamma=(a,b;c,d)\in\Gamma_h}
 \lambda_\gamma^{\varepsilon,h}(m)
 [F(m+\varepsilon(e_c+e_d-e_a-e_b))-F(m)].
 \label{coarse-boltzmann-generator}
\end{align}
By
Lemma~\ref{lem:coarse-invariant-state-space}, every channel with positive
rate has its endpoint in $\mathsf S_{\varepsilon,h}(M)$, so all the displayed
values of $F$ are then well defined.
Moreover, for the solution $m$ taking values in
$\mathsf S_{\varepsilon,h}(M)$, the process
\begin{equation}
 M_F(t):=F(m(t))-F(m(0))
 -\int_0^t\mathcal A_{\varepsilon,h}F(m(r))\,dr
 \label{coarse-generator-martingale}
\end{equation}
is a c\`adl\`ag square-integrable martingale.

More explicitly, let $\varphi$ be a smooth test function and set
\begin{equation*}
 \varphi_a^h:=\int_{\mathbb T^d\times\mathbb R^d}
 \chi_a^h(z)\varphi(z)\,dz,
 \qquad
 \Delta_\gamma^h\varphi
 :=\varphi_c^h+\varphi_d^h-\varphi_a^h-\varphi_b^h,
 \quad \gamma=(a,b;c,d).
\end{equation*}
Then the reconstructed density
$f_t^{\varepsilon,h}=\sum_{a\in\mathcal I_h}m_a(t)\chi_a^h$ satisfies
\begin{align}
 \langle f_t^{\varepsilon,h},\varphi\rangle
 &=\langle f_0^{\varepsilon,h},\varphi\rangle
 +\int_0^t\sum_{(a,b)\in\mathcal E_h}
 m_a(r)T_{ab}^h(\varphi_b^h-\varphi_a^h)\,dr\notag\\
 &\quad+\frac12\int_0^t
 \sum_{\gamma=(a,b;c,d)\in\Gamma_h}K_\gamma^h
 m_a(r)\bigl(m_b(r)-\varepsilon\mathbf1_{\{a=b\}}\bigr)
 \Delta_\gamma^h\varphi\,dr
 +M_\varphi^{\varepsilon,h}(t),
 \label{coarse-tested-density-martingale-decomposition}
\end{align}
where
\begin{align}
 M_\varphi^{\varepsilon,h}(t)
 :=\varepsilon\sum_{(a,b)\in\mathcal E_h}\int_0^t
 (\varphi_b^h-\varphi_a^h)\,
 d\widetilde P_{ab}^{\varepsilon,h}(r)+\varepsilon\sum_{\gamma\in\Gamma_h}\int_0^t
 \Delta_\gamma^h\varphi\,
 d\widetilde P_\gamma^{\varepsilon,h}(r)
 \label{coarse-tested-density-martingale}
\end{align}
is a square-integrable martingale.  For two smooth test functions
$\varphi$ and $\psi$, its predictable quadratic covariation is
\begin{align}
 \big\langle M_\varphi^{\varepsilon,h},
 M_\psi^{\varepsilon,h}\big\rangle_t
 &=\varepsilon\int_0^t\sum_{(a,b)\in\mathcal E_h}
 m_a(r)T_{ab}^h
 (\varphi_b^h-\varphi_a^h)(\psi_b^h-\psi_a^h)\,dr\notag\\
 &\quad+\frac\varepsilon2\int_0^t
 \sum_{\gamma=(a,b;c,d)\in\Gamma_h}K_\gamma^h
 m_a(r)\bigl(m_b(r)-\varepsilon\mathbf1_{\{a=b\}}\bigr)
 \Delta_\gamma^h\varphi\,\Delta_\gamma^h\psi\,dr.
 \label{coarse-tested-density-bracket}
\end{align}
\end{proposition}

\begin{proof}
Recall the notation introduced at the beginning of
Subsection~\ref{subsec:main-results-rigorous}: for a transport edge
$(a,b)\in\mathcal E_h$ the jump vector is $e_b-e_a$, whereas for a
collision channel $\gamma=(a,b;c,d)\in\Gamma_h$ it is
\begin{equation*}
 z_\gamma=e_c+e_d-e_a-e_b.
\end{equation*}
Conditional on $m(t)=m\in\mathsf S_{\varepsilon,h}(M)$, during a time interval
of length $s$ a transport jump $(a,b)\in\mathcal E_h$ occurs with probability
$s\lambda_{ab}^{\varepsilon,h}(m)+o(s)$, a jump of type $\gamma$ occurs with
probability $s\lambda_\gamma^{\varepsilon,h}(m)+o(s)$, and the probability
of two or more jumps is $o(s)$.  Indeed,
Lemma~\ref{lem:coarse-total-intensity} bounds the total jump intensity on
$\mathsf S_{\varepsilon,h}(M)$ by $C_{\varepsilon,h}(M)$.  Hence the number
of jumps during an interval of length $s$ is stochastically dominated by a
random variable $Y_s$ with Poisson distribution of mean
$C_{\varepsilon,h}(M)s$.  Writing $C=C_{\varepsilon,h}(M)$, we obtain
\begin{equation*}
 \mathbb P(Y_s\geq2)
 =1-\mathbb P(Y_s=0)-\mathbb P(Y_s=1)
 =1-e^{-Cs}(1+Cs)=O(s^2)=o(s),
\end{equation*}
which proves the estimate for two or more jumps.
Let $\mathbb E_m$ be the expectation of the process started from the
deterministic state $m(0)=m$.  The preceding short-time probabilities give
\begin{align*}
 \mathbb E_m\bigl[F(m(s))-F(m)\bigr]
 &=s\sum_{(a,b)\in\mathcal E_h}\lambda_{ab}^{\varepsilon,h}(m)
   [F(m+\varepsilon(e_b-e_a))-F(m)]\\
 &\quad+s\sum_{\gamma=(a,b;c,d)\in\Gamma_h}
   \lambda_\gamma^{\varepsilon,h}(m)
   [F(m+\varepsilon(e_c+e_d-e_a-e_b))-F(m)]+o(s).
\end{align*}
Consequently, dividing by $s$ and letting $s\downarrow0$ yields
\begin{align*}
 \lim_{s\downarrow0}\frac{\mathbb E_mF(m(s))-F(m)}s
 &=\sum_{(a,b)\in\mathcal E_h}\lambda_{ab}^{\varepsilon,h}(m)
 [F(m+\varepsilon(e_b-e_a))-F(m)]\\
 &\quad+\sum_{\gamma=(a,b;c,d)\in\Gamma_h}
 \lambda_\gamma^{\varepsilon,h}(m)
 [F(m+\varepsilon(e_c+e_d-e_a-e_b))-F(m)],
\end{align*}
which is \eqref{coarse-boltzmann-generator}.  Since the state space is finite
and the total rates are bounded, every bounded $F$ belongs to the generator
domain.  Applying the jump-chain rule to
\eqref{coarse-boltzmann-model} and subtracting the predictable compensator
of each channel count gives
\begin{align*}
 M_F(t)
 &=\sum_{(a,b)\in\mathcal E_h}\int_0^t
 \bigl[F(m(r-)+\varepsilon(e_b-e_a))-F(m(r-))\bigr]
 \,d\widetilde P_{ab}^{\varepsilon,h}(r)\\
 &\quad+\sum_{\gamma=(a,b;c,d)\in\Gamma_h}\int_0^t
 \bigl[F(m(r-)+\varepsilon(e_c+e_d-e_a-e_b))-F(m(r-))\bigr]
 \,d\widetilde P_\gamma^{\varepsilon,h}(r).
\end{align*}
The integrands are predictable and bounded by $2\lVert F\rVert_\infty$.
Moreover, Lemma~\ref{lem:coarse-total-intensity} gives
\begin{equation*}
 \int_0^t\left(
 \sum_{(a,b)\in\mathcal E_h}\lambda_{ab}^{\varepsilon,h}(m(r-))
 +\sum_{\gamma\in\Gamma_h}\lambda_\gamma^{\varepsilon,h}(m(r-))
 \right)dr
 \leq C_{\varepsilon,h}(M)t.
\end{equation*}
Hence both stochastic sums above are square-integrable martingales, and so
is $M_F$ in \eqref{coarse-generator-martingale}.

For the final assertions, take
\begin{equation*}
 F_\varphi(m):=\sum_{a\in\mathcal I_h}m_a\varphi_a^h
 =\langle f^{\varepsilon,h},\varphi\rangle.
\end{equation*}
This function is bounded on $\mathsf S_{\varepsilon,h}(M)$, and its
increments along the two types of jumps are
\begin{align*}
 F_\varphi(m+\varepsilon(e_b-e_a))-F_\varphi(m)
 &=\varepsilon(\varphi_b^h-\varphi_a^h),\\
 F_\varphi(m+\varepsilon(e_c+e_d-e_a-e_b))-F_\varphi(m)
 &=\varepsilon\Delta_\gamma^h\varphi,
 \qquad \gamma=(a,b;c,d).
\end{align*}
Substituting these identities and the definitions of the rates into
\eqref{coarse-boltzmann-generator} gives
\begin{align*}
 \mathcal A_{\varepsilon,h}F_\varphi(m)
 &=\sum_{(a,b)\in\mathcal E_h}m_aT_{ab}^h
 (\varphi_b^h-\varphi_a^h)\\
 &\quad+\frac12\sum_{\gamma=(a,b;c,d)\in\Gamma_h}K_\gamma^h
 m_a(m_b-\varepsilon\mathbf1_{\{a=b\}})
 \Delta_\gamma^h\varphi.
\end{align*}
The general martingale identity \eqref{coarse-generator-martingale} now
gives \eqref{coarse-tested-density-martingale-decomposition}, while the
stochastic-integral representation of $M_F$ above gives
\eqref{coarse-tested-density-martingale}.

It remains to compute the predictable covariation.  Distinct channel counts
have zero predictable cross-covariation, whereas
\begin{align*}
 d\big\langle\widetilde P_{ab}^{\varepsilon,h}\big\rangle_r
 &=\lambda_{ab}^{\varepsilon,h}(m(r-))\,dr,\\
 d\big\langle\widetilde P_\gamma^{\varepsilon,h}\big\rangle_r
 &=\lambda_\gamma^{\varepsilon,h}(m(r-))\,dr.
\end{align*}
Applying the predictable-covariation formula to
\eqref{coarse-tested-density-martingale}, and then inserting
\eqref{coarse-transport-rate}--\eqref{coarse-boltzmann-collision-rate},
yields \eqref{coarse-tested-density-bracket}.
\end{proof}

\section{The finite-mesh path large-deviation principle}
\label{sec:fixed-mesh-ldp-rigorous}

This section establishes the path large-deviation principle for the
coarse-grained process at a fixed mesh size $h$.  We first study the
deterministic control problem defining the rate function, including its
stability, behavior near the boundary of the mass simplex, compactness, and
dual representation.  We then represent the jump process by a marked
Poisson random measure and prove the Laplace principle through controlled
compactness and recovery.  Finally, we allow independent Poisson initial
occupations and show that their initial entropy is added to the dynamical
action.

\subsection{Analysis of rate functions}

We begin with the analytic properties of the finite-mesh action.  The main
tasks are to relate the finite-$\varepsilon$ jump rates to their limiting
nominal rates, establish stability of the controlled skeleton equation,
regularize controlled paths that meet the boundary of the mass simplex,
and prove that the resulting rate function is good.  We also record a dual
formula for the scalar Poisson cost that will be used in the lower-bound
argument.

Throughout this section $h>0$ is fixed.  We use the unified channel notation
from \eqref{intro-unified-channel-data}: recall that
$\mathcal K_h=\mathcal E_h\sqcup\Gamma_h$, and let $\kappa$ denote a
generic element of $\mathcal K_h$, while $z_\kappa$ and $\beta_\kappa$
denote its jump vector and limiting nominal rate.  We also recall the rate function
$I_h^{m_0^h}$ from \eqref{intro-fixed-h-rigorous-rate}, subject to the
skeleton equation \eqref{intro-fixed-h-rigorous-continuity}.
The control multiplier and its cost are
$q_\kappa$ and $\beta_\kappa(m)\ell(q_\kappa)$ from
\eqref{intro-fixed-h-rigorous-rate}.

To connect the finite-$\varepsilon$ process with its limiting action, we
also use the finite-$\varepsilon$ nominal rates.  For
$m\in\mathbb R_{\geq0}^{\mathcal I_h}$ and
$\kappa\in\mathcal K_h$, define
\begin{equation}\label{finite-particle-beta}
 \beta_\kappa^\varepsilon(m)
 =\begin{cases}
 m_aT_{ab}^h,
 &\kappa=(a,b)\in\mathcal E_h,\\[1mm]
 \displaystyle\frac12K_\kappa^hm_a
 \bigl(m_b-\varepsilon\mathbf1_{\{a=b\}}\bigr)_+,
 &\kappa=(a,b;c,d)\in\Gamma_h.
 \end{cases}
\end{equation}
If $m$ belongs to a lattice state space
$\mathsf S_{\varepsilon,h}(M)$, the positive part in the second line is
redundant: when $a=b$, either $m_a=0$ or $m_a\geq\varepsilon$.
Consequently, the stochastic intensities defined in
\eqref{coarse-transport-rate}--\eqref{coarse-boltzmann-collision-rate}
satisfy
\begin{equation*}
 \lambda_{ab}^{\varepsilon,h}(m)
 =\varepsilon^{-1}\beta_{(a,b)}^\varepsilon(m),
 \qquad
 \lambda_\gamma^{\varepsilon,h}(m)
 =\varepsilon^{-1}\beta_\gamma^\varepsilon(m).
\end{equation*}
In the large-deviation argument, the controlled paths take values
in the full continuous state space $\mathbb R_{\geq0}^{\mathcal I_h}$.  Without
the positive part, a diagonal collision channel with $a=b$ would have a
negative nominal rate whenever $0<m_a<\varepsilon$.  Formula
\eqref{finite-particle-beta} therefore provides the nonnegative continuous
extension needed for the Poisson acceptance rate.  On each fixed-mass sublevel, and in
particular on the unit-mass simplex,
\begin{equation}\label{uniform-beta-epsilon-limit}
 \max_{\kappa\in\mathcal K_h}
 \sup_{\substack{m\geq0\\\sum_am_a\leq M}}
 |\beta_\kappa^\varepsilon(m)-\beta_\kappa(m)|
 \longrightarrow0
 \qquad(\varepsilon\downarrow0),
\end{equation}
for every fixed $M<\infty$.
The equality $\beta_\kappa^\varepsilon=\beta_\kappa$ holds for every transport
channel and every collision channel with $a\neq b$.  If
$\kappa=(a,a;c,d)\in\Gamma_h$, then
\begin{align*}
 0
 &\leq \beta_\kappa(m)-\beta_\kappa^\varepsilon(m)\\
 &=\frac12K_\kappa^hm_a
 \bigl[m_a-(m_a-\varepsilon)_+\bigr]
 \leq\frac12K_\kappa^h\varepsilon m_a
 \leq\frac12K_\kappa^h\varepsilon M.
\end{align*}
Since $\Gamma_h$ is finite for fixed $h$, taking the supremum over the
mass sublevel and then the maximum over the channels proves
\eqref{uniform-beta-epsilon-limit}.

\begin{lemma}
\label{lem:bounded-control-skeleton-stability}
Fix $n\in\mathbb N$.  Suppose that
$q=(q_\kappa)_{\kappa\in\mathcal K_h}$ is measurable and
$0\leq q_\kappa\leq n$ for every $\kappa\in\mathcal K_h$.  Then the skeleton equation
\eqref{intro-fixed-h-rigorous-continuity} has a unique solution in
$AC([0,T];\mathbb R_{\geq0}^{\mathcal I_h})$, and that solution remains in
the unit-mass simplex.

Let $(q^r)_{r\geq1}$ be a sequence with the same bound.  If
\begin{equation}\label{bounded-control-weak-convergence}
 q_\kappa^r\rightharpoonup q_\kappa
 \quad\text{weakly in }L^1(0,T)
 \text{ for every }\kappa\in\mathcal K_h,
\end{equation}
then the corresponding skeleton paths satisfy
\begin{equation}\label{bounded-skeleton-uniform-stability}
 \sup_{t\in[0,T]}|m^{q^r}(t)-m^q(t)|
 \longrightarrow0.
\end{equation}
\end{lemma}

\begin{proof}
For fixed $q$, the right-hand side of the skeleton equation is
measurable in time.  Since every $\beta_\kappa$ is Lipschitz on the
unit-mass simplex and $0\leq q_\kappa\leq n$, it is Lipschitz in $m$,
with a constant depending only on $h$ and $n$.  The Carath\'eodory theorem
therefore gives local existence and uniqueness.
Every jump vector has zero coordinate sum, so the total mass is constant.
Moreover, if $m_a=0$, every channel having cell $a$ as an incoming cell
has zero nominal rate, whereas channels having $a$ as an outgoing cell
contribute a nonnegative amount to the $a$-th coordinate.  The vector
field is therefore inward pointing on every boundary face of the
nonnegative orthant.  Hence the solution stays nonnegative and has total
mass one.  Compactness of the unit-mass simplex extends it to $[0,T]$.

For stability, the paths $m^{q^r}$ are uniformly bounded and have a common
Lipschitz constant, hence are relatively compact in
$C([0,T];\mathbb R^{\mathcal I_h})$.  If a subsequence converges uniformly
to $\widetilde m$, then, for every $\kappa\in\mathcal K_h$, continuity of
$\beta_\kappa$ gives
\begin{align*}
 &\int_0^t\beta_\kappa(m^{q^r}(s))q_\kappa^r(s)\,ds
 -\int_0^t\beta_\kappa(\widetilde m(s))q_\kappa(s)\,ds\\
 &\quad=\int_0^t
 \bigl[\beta_\kappa(m^{q^r}(s))-\beta_\kappa(\widetilde m(s))\bigr]
 q_\kappa^r(s)\,ds
 +\int_0^t\beta_\kappa(\widetilde m(s))
 \bigl(q_\kappa^r(s)-q_\kappa(s)\bigr)\,ds
 \longrightarrow0.
\end{align*}
Here the first term tends to zero uniformly in $t$ by uniform convergence
of the paths and boundedness of $q^r$, while the second tends to zero by
\eqref{bounded-control-weak-convergence}.  Passing to the limit in the
integral skeleton equation shows that $\widetilde m=m^q$ by uniqueness.
Since every convergent subsequence has the
same limit, the whole sequence converges as in
\eqref{bounded-skeleton-uniform-stability}.
\end{proof}

The collision part of the finite-mesh action
\eqref{intro-fixed-h-rigorous-rate} has the same relative-entropy structure
as the continuum functional $\mathcal I_B$ in
\eqref{boltzmann-path-rate-control}.  For a collision channel
$\gamma=(a,b;c,d)\in\Gamma_h$,
\begin{equation*}
 \beta_\gamma(m)\ell(q_\gamma)
 =\frac12K_\gamma^hm_am_b\ell(q_\gamma).
\end{equation*}
This is the cellwise analogue of the integrand
$\frac12Bgg_*\ell(q)$ in \eqref{boltzmann-path-rate-control}: the collision
measure $B\,dc$ is replaced by the discrete weight $K_\gamma^h$, the product
$gg_*$ by $m_am_b$, and the relative collision activity by the channel
multiplier $q_\gamma$.  The precise equivalence between $\mathcal I_B$ and
the continuum Poissonian rate
$I_{X_0}^{\rm FB}$ is proved in
Lemma~\ref{lem:FB-control-rate-equivalence} of
Appendix~\ref{app:fluctuating-boltzmann-rate}.

\begin{lemma}
\label{lem:fixed-h-vacuum-regularization}
Suppose that $(m,q)$ satisfies the skeleton equation
\eqref{intro-fixed-h-rigorous-continuity}, without necessarily prescribing
its initial point, and has finite action. Let
$\bar m\in(0,\infty)^{\mathcal I_h}$ be fixed. For $\delta\in(0,1)$ set
\begin{equation}\label{fixed-h-buffer-path}
 c_\delta=1-\delta,
 \qquad m^\delta(\cdot)=c_\delta m(\cdot)+\delta\bar m.
\end{equation}
For every $\kappa\in\mathcal K_h$ define
\begin{equation}\label{fixed-h-buffer-control}
 q_\kappa^\delta(t):=
 \begin{cases}
 \displaystyle
 \frac{c_\delta\beta_\kappa(m(t))q_\kappa(t)}
      {\beta_\kappa(m^\delta(t))},
 &\beta_\kappa(m^\delta(t))>0,\\[3mm]
 1,&\beta_\kappa(m^\delta(t))=0.
 \end{cases}
\end{equation}
Then $m^\delta(0)=c_\delta m(0)+\delta\bar m$,
\begin{equation*}
 \inf_{t\in[0,T]}m_a^\delta(t)
 \geq\delta\bar m_a>0
 \qquad\text{for every } a\in\mathcal I_h,
\end{equation*}
and $(m^\delta,q^\delta)$ satisfies the skeleton equation
\eqref{intro-fixed-h-rigorous-continuity}.  Moreover,
\begin{equation}\label{fixed-h-buffer-action-limit}
 \limsup_{\delta\downarrow0}
 \int_0^T\sum_{\kappa\in\mathcal K_h}
 \beta_\kappa(m^\delta)\ell(q_\kappa^\delta)\,dt
 \leq
 \int_0^T\sum_{\kappa\in\mathcal K_h}
 \beta_\kappa(m)\ell(q_\kappa)\,dt .
\end{equation}
\end{lemma}

\begin{proof}
The lower bound follows directly from
\eqref{fixed-h-buffer-path}.  Since
\begin{equation*}
 m_a^\delta=c_\delta m_a+\delta\bar m_a
 \geq c_\delta m_a
 \qquad (a\in\mathcal I_h),
\end{equation*}
the two types of nominal rates satisfy
\begin{align*}
 \beta_{(a,b)}(m^\delta)
 &=T_{ab}^h m_a^\delta
 \geq c_\delta T_{ab}^h m_a
 =c_\delta\beta_{(a,b)}(m),\\[2mm]
 \beta_{(a,b;c,d)}(m^\delta)
 &=\frac12K_{(a,b;c,d)}^h m_a^\delta m_b^\delta
 \geq c_\delta^2\frac12K_{(a,b;c,d)}^h m_am_b
 =c_\delta^2\beta_{(a,b;c,d)}(m).
\end{align*}
Thus, for $\kappa\in\mathcal K_h$, set $d_\kappa=1$ when
$\kappa\in\mathcal E_h$ and $d_\kappa=2$ when
$\kappa\in\Gamma_h$.  Then
\begin{equation}\label{buffer-nominal-lower-bound}
 \beta_\kappa(m^\delta)\geq
 c_\delta^{d_\kappa}\beta_\kappa(m).
\end{equation}
In particular, if $\beta_\kappa(m)q_\kappa>0$, then
$\beta_\kappa(m^\delta)>0$.  Thus \eqref{fixed-h-buffer-control} gives
\begin{equation*}
 \beta_\kappa(m^\delta)q_\kappa^\delta
 =c_\delta\beta_\kappa(m)q_\kappa
 \quad\text{a.e. on }[0,T].
\end{equation*}
Consequently,
\begin{equation*}
 \sum_{\kappa\in\mathcal K_h}
 \beta_\kappa(m^\delta)q_\kappa^\delta z_\kappa
 =c_\delta\sum_{\kappa\in\mathcal K_h}
 \beta_\kappa(m)q_\kappa z_\kappa
 =\dot m^\delta,
\end{equation*}
which proves the skeleton equation for $(m^\delta,q^\delta)$.

For $\kappa\in\mathcal K_h$, set
$J_\kappa:=\beta_\kappa(m)q_\kappa$.  If
$\beta_\kappa(m)>0$, a direct calculation gives
\begin{align*}
 \beta_\kappa(m^\delta)\ell(q_\kappa^\delta)
 &=c_\delta\beta_\kappa(m)\ell(q_\kappa)
 +c_\delta J_\kappa
 \log\frac{c_\delta\beta_\kappa(m)}{\beta_\kappa(m^\delta)}\\
 &\quad+\beta_\kappa(m^\delta)-c_\delta\beta_\kappa(m)\\
 &\leq c_\delta\beta_\kappa(m)\ell(q_\kappa)
 +c_\delta(d_\kappa-1)|\log c_\delta|J_\kappa\\
 &\quad+\beta_\kappa(m^\delta)-c_\delta\beta_\kappa(m).
\end{align*}
If $\beta_\kappa(m)=0$, then $J_\kappa=0$, and the same inequality follows
directly from $q_\kappa^\delta=0$ when
$\beta_\kappa(m^\delta)>0$ (the value assigned when both rates vanish is
irrelevant).  The scalar inequality
\begin{equation}\label{fixed-h-flux-fenchel}
 s\beta_\kappa(m(t))q_\kappa(t)
 \leq\beta_\kappa(m(t))\ell(q_\kappa(t))
 +\beta_\kappa(m(t))(e^s-1),
 \qquad s>0,
\end{equation}
follows from $sq\leq\ell(q)+e^s-1$.  The finite-action assumption implies
that the first term on the right-hand side is integrable.  Moreover, the nominal
rates are bounded on the compact mass simplex containing $m$ and $\bar m$.
Integrating \eqref{fixed-h-flux-fenchel} over $[0,T]$ therefore gives
$J_\kappa\in L^1(0,T)$.  The last three terms in the preceding bound consequently
vanish in $L^1(0,T)$ as $\delta\downarrow0$.  Summation over the finitely
many channels proves \eqref{fixed-h-buffer-action-limit}.
\end{proof}

In the following, we show that the rate function is a good rate function. 
\begin{lemma}
\label{lem:fixed-h-action-good-rigorous}
For fixed $m_0^h$, the functional $I_h^{m_0^h}$ defined by
\eqref{intro-fixed-h-rigorous-rate} is lower semicontinuous on
$D([0,T];\mathbb R_{\geq0}^{\mathcal I_h})$ and has compact sublevel sets.
Every finite-rate path is absolutely continuous, and hence continuous.
\end{lemma}

\begin{proof}
Fix $L<\infty$, and take a sequence
\begin{equation*}
 (m^n)_{n\geq1}\subset
 D([0,T];\mathbb R_{\geq0}^{\mathcal I_h})
 \qquad\text{such that}\qquad
 \sup_{n\geq1}I_h^{m_0^h}(m^n)\leq L.
\end{equation*}
By the definition of the infimum in
\eqref{intro-fixed-h-rigorous-rate}, for each $n$ there is an admissible
control $q^n=(q_\kappa^n)_{\kappa\in\mathcal K_h}$ such that
\begin{equation}\label{near-minimizing-control-action}
 \int_0^T\sum_{\kappa\in\mathcal K_h}
 \beta_\kappa(m^n(t))\ell(q_\kappa^n(t))\,dt
 \leq I_h^{m_0^h}(m^n)+\frac1n
 \leq L+\frac1n.
\end{equation}
Introduce, only for this compactness argument, the effective channel fluxes
\begin{equation}\label{near-minimizing-effective-flux}
 J_\kappa^n(t):=\beta_\kappa(m^n(t))q_\kappa^n(t),
 \qquad \kappa\in\mathcal K_h.
\end{equation}
The skeleton equation then reads
$\dot m^n=\sum_{\kappa\in\mathcal K_h}J_\kappa^nz_\kappa$.
Since every jump vector has zero total mass,
\eqref{intro-fixed-h-rigorous-continuity} and
$\sum_am_{0,a}^h=1$ imply
\begin{equation*}
 \sum_{a\in\mathcal I_h}m_a^n(t)=1,
 \qquad t\in[0,T].
\end{equation*}
Consequently, $0\leq m_a^n(t)\leq1$ for every $a$, $n$, and $t$, and the
nominal rates are uniformly bounded by
\begin{equation*}
 A_h:=
 \max_{\kappa\in\mathcal K_h}
 \sup_{\substack{m\in\mathbb R_{\geq0}^{\mathcal I_h}\\\sum_am_a=1}}
 \beta_\kappa(m)<\infty.
\end{equation*}
Applying
\eqref{fixed-h-flux-fenchel} on a measurable set $E\subset[0,T]$ gives
\begin{equation*}
 \int_EJ_\kappa^n(t)\,dt
 \leq\frac1s\int_E
 \beta_\kappa(m^n(t))\ell(q_\kappa^n(t))\,dt
 +\frac{A_h(e^s-1)}s|E|.
\end{equation*}
By \eqref{near-minimizing-control-action}, for every $s>0$,
\begin{equation}\label{uniform-integrability-flux-bound}
 \sup_{n\geq1}\int_EJ_\kappa^n(t)\,dt
 \leq\frac{L+1}{s}
 +\frac{A_h(e^s-1)}s|E|.
\end{equation}
Taking $E=[0,T]$ and, for instance, $s=1$ shows that
\begin{equation*}
 \sup_{n\geq1}\|J_\kappa^n\|_{L^1(0,T)}
 \leq L+1+A_h(e-1)T<\infty.
\end{equation*}
Moreover, given $\eta>0$, first choose $s>0$ so that
$(L+1)/s<\eta/2$, and then choose $\delta>0$ so that
\begin{equation*}
 \frac{A_h(e^s-1)}s\,\delta<\frac\eta2.
\end{equation*}
It follows from \eqref{uniform-integrability-flux-bound} that
\begin{equation*}
 |E|<\delta
 \quad\Longrightarrow\quad
 \sup_{n\geq1}\int_EJ_\kappa^n(t)\,\d t<\eta.
\end{equation*}
Thus $\{J_\kappa^n:n\geq1\}$ is uniformly integrable in $L^1(0,T)$ for
every $\kappa\in\mathcal K_h$.  The skeleton equation
\eqref{intro-fixed-h-rigorous-continuity} gives, for
$0\leq s\leq t\leq T$,
\begin{align*}
 \|m^n(t)-m^n(s)\|_{\ell^1(\mathcal I_h)}
 &\leq
 \sum_{\kappa\in\mathcal K_h}
 \|z_\kappa\|_{\ell^1(\mathcal I_h)}
 \int_s^tJ_\kappa^n(r)\,\d r.
\end{align*}
Define
\begin{equation*}
 \omega_h(\delta):=
 \sum_{\kappa\in\mathcal K_h}
 \|z_\kappa\|_{\ell^1(\mathcal I_h)}
 \sup_{\substack{n\geq1,\ E\subset[0,T]\ {\rm measurable}\\|E|\leq\delta}}
 \int_EJ_\kappa^n(r)\,\d r .
\end{equation*}
There are only finitely many channels, so the uniform integrability proved
above implies $\omega_h(\delta)\to0$ as $\delta\downarrow0$.  Hence
\begin{equation*}
 \sup_{n\geq1}\|m^n(t)-m^n(s)\|_{\ell^1(\mathcal I_h)}
 \leq\omega_h(|t-s|),
\end{equation*}
which is a common modulus of continuity for the paths.  Arzel\`a--Ascoli and
Dunford--Pettis yield, along a subsequence,
\begin{equation*}
 m^n\longrightarrow m\quad\text{uniformly},
 \qquad J_\kappa^n\rightharpoonup J_\kappa
 \quad\text{weakly in }L^1(0,T)
 \quad\text{for every }\kappa\in\mathcal K_h.
\end{equation*}

We first pass to the limit in the skeleton equation.  Its integral form is
\begin{equation}\label{near-minimizer-integrated-skeleton}
 m^n(t)=m_0^h+
 \sum_{\kappa\in\mathcal K_h}z_\kappa
 \int_0^tJ_\kappa^n(r)\,\d r,
 \qquad t\in[0,T].
\end{equation}
For each fixed $t$, the indicator
$\boldsymbol1_{[0,t]}$ belongs to $L^\infty(0,T)$.  Hence the weak
$L^1$ convergence of $J_\kappa^n$ implies
\begin{equation*}
 \int_0^tJ_\kappa^n(r)\,\d r
 \longrightarrow
 \int_0^tJ_\kappa(r)\,\d r.
\end{equation*}
Letting $n\to\infty$ in
\eqref{near-minimizer-integrated-skeleton} therefore gives
\begin{equation}\label{limit-integrated-skeleton}
 m(t)=m_0^h+
 \sum_{\kappa\in\mathcal K_h}z_\kappa
 \int_0^tJ_\kappa(r)\,\d r.
\end{equation}
It remains to recover a control for the limiting path and compare the
actions.  For every $\kappa\in\mathcal K_h$, write
\begin{equation*}
 a_\kappa^n(t):=\beta_\kappa(m^n(t)),
 \qquad a_\kappa(t):=\beta_\kappa(m(t)).
\end{equation*}
Since $m^n\to m$ uniformly and each $\beta_\kappa$ is continuous,
$a_\kappa^n\to a_\kappa$ uniformly on $[0,T]$.  Applying
Lemma~\ref{lem:integral-relative-flux-duality} to each
$\kappa\in\mathcal K_h$, using the
weak $L^1$ convergence of $J_\kappa^n$ and the uniform convergence of
$a_\kappa^n$, gives
\begin{align}\label{rate-lower-semicontinuity-chain}
 &\int_0^T\sum_{\kappa\in\mathcal K_h}
 a_\kappa(t)\ell\left(\frac{J_\kappa(t)}{a_\kappa(t)}\right)\,dt
 \notag\\
 &\qquad\leq\liminf_{n\to\infty}
 \int_0^T\sum_{\kappa\in\mathcal K_h}
 a_\kappa^n(t)\ell\left(\frac{J_\kappa^n(t)}{a_\kappa^n(t)}\right)\,dt
 \leq L.
\end{align}
Here and below the integrand is defined to be $0$ when $a=J=0$ and
$+\infty$ when $a=0<J$.  In particular, \eqref{rate-lower-semicontinuity-chain}
implies $J_\kappa=0$ almost everywhere on $\{a_\kappa=0\}$.  For every
$\kappa\in\mathcal K_h$, define
\begin{equation}\label{limit-control-from-effective-flux}
 q_\kappa(t):=
 \begin{cases}
 J_\kappa(t)/a_\kappa(t),&a_\kappa(t)>0,\\
 1,&a_\kappa(t)=0.
 \end{cases}
\end{equation}
Then $J_\kappa=a_\kappa q_\kappa$ almost everywhere, so
\eqref{limit-integrated-skeleton} is precisely the integral form of the
skeleton equation for $(m,q)$.  Therefore
\begin{equation*}
 I_h^{m_0^h}(m)
 \leq\int_0^T\sum_{\kappa\in\mathcal K_h}
 \beta_\kappa(m(t))\ell(q_\kappa(t))\,dt
 \leq\liminf_{n\to\infty}I_h^{m_0^h}(m^n)
 \leq L.
\end{equation*}
Thus every sequence in the level set
$\{I_h^{m_0^h}\leq L\}$ has a uniformly convergent subsequence whose limit
belongs to the same level set.  The level set is therefore compact in the
uniform topology, and hence also in the weaker Skorokhod $J_1$ topology.
In particular, every level set is $J_1$-closed, which proves lower
semicontinuity.

Finally, fix $m$ with $I_h^{m_0^h}(m)<\infty$ and choose admissible
controls $q^n$ whose actions decrease to $I_h^{m_0^h}(m)$.  Applying the
same compactness argument to
$J_\kappa^n=\beta_\kappa(m)q_\kappa^n$ yields a weak limit $J$ and then,
by \eqref{limit-control-from-effective-flux}, an admissible control $q$
whose action is no larger than the limiting actions.  Hence the infimum
in \eqref{intro-fixed-h-rigorous-rate} is attained.  Since there are
finitely many channels and
$\beta_\kappa(m)q_\kappa\in L^1(0,T)$, the skeleton equation gives
\begin{equation*}
 \dot m=\sum_{\kappa\in\mathcal K_h}
 \beta_\kappa(m)q_\kappa z_\kappa
 \in L^1(0,T;\mathbb R^{\mathcal I_h}).
\end{equation*}
Thus every finite-rate path is absolutely continuous, and hence continuous.
\end{proof}

For later use, we record the following dual representation.

\begin{lemma}\label{lem:integral-relative-flux-duality}
Let $a,J\in L^1(0,T)$ be nonnegative.  Then
\begin{align}
 \int_0^T a(t)\ell\left(\frac{J(t)}{a(t)}\right)\,dt
 =\sup_{\theta\in C([0,T])}\int_0^T
 \bigl\{\theta(t)J(t)-a(t)(e^{\theta(t)}-1)\bigr\}\,dt.
 \label{integral-relative-flux-duality}
\end{align}
Both sides are allowed to take the value $+\infty$, and the integrand on
the left is defined to be $0$ when $a=J=0$ and $+\infty$ when $a=0<J$.
\end{lemma}

\begin{proof}
For every $j,a\geq0$ and $s\in\mathbb R$, direct differentiation in $s$
gives
\begin{equation*}
 sj-a(e^s-1)\leq a\ell(j/a),
\end{equation*}
with the same convention at $a=0$.
Applying the preceding inequality with
$(j,a,s)=(J(t),a(t),\theta(t))$ for almost every $t$ and then integrating
over $[0,T]$ gives the ``$\geq$''
direction of \eqref{integral-relative-flux-duality}. 

We prove the reverse inequality.  For $M\in\mathbb N$, choose a bounded
measurable function $\theta_M:[0,T]\to[-M,M]$ such that, for almost every
$t\in [0,T]$,
\begin{equation*}
 \theta_M(t)\in\operatorname*{arg\,max}_{|s|\leq M}
 \bigl\{sJ(t)-a(t)(e^s-1)\bigr\}.
\end{equation*}
It may be chosen explicitly as
\begin{equation*}
 \theta_M(t)=
 \begin{cases}
  \max\{-M,\min\{M,\log(J(t)/a(t))\}\},&J(t)>0, a(t)>0,\\
  -M,&J(t)=0, a(t)>0,\\
  M,&J(t)>0, a(t)=0,\\
  0,&J(t)=a(t)=0.
 \end{cases}
\end{equation*}
Consequently,
\begin{equation*}
 G_M(t):=\theta_M(t)J(t)-a(t)(e^{\theta_M(t)}-1)
 =\max_{|s|\leq M}\bigl\{sJ(t)-a(t)(e^s-1)\bigr\}.
\end{equation*}
The sequence $(G_M)_{M\in\mathbb{N}}$ is nonnegative and nondecreasing in $M\in\mathbb{N}$, and the preceding
scalar maximization gives
\begin{equation*}
 G_M(t)\uparrow a(t)\ell\left(\frac{J(t)}{a(t)}\right)
 \qquad\text{for a.e. }t\in[0,T].
\end{equation*}
Hence the monotone convergence theorem yields
\begin{equation*}
 \int_0^T a(t)\ell\left(\frac{J(t)}{a(t)}\right)\,dt
 =\lim_{M\to\infty}\int_0^T G_M(t)\,dt.
\end{equation*}

It remains to replace the measurable function $\theta_M$ by continuous
functions.  For fixed $M$, the measure
\begin{equation*}
 \mu(dt):=(J(t)+a(t))\,dt
\end{equation*}
is a finite Borel measure on $[0,T]$.  By the density of continuous
functions in $L^1([0,T],\mu)$, followed by truncation to $[-M,M]$, there
exist $\theta_{M,k}\in C([0,T])$, with
$|\theta_{M,k}|\leq M$, such that
\begin{equation*}
 \int_0^T|\theta_{M,k}(t)-\theta_M(t)|(J(t)+a(t))\,dt
 \longrightarrow0.
\end{equation*}
Since the exponential function is $e^M$-Lipschitz on $[-M,M]$,
\begin{align*}
 &\left|\int_0^T
 \bigl\{\theta_{M,k}J-a(e^{\theta_{M,k}}-1)\bigr\}\,dt
 -\int_0^T G_M(t)\,dt\right|\\
 &\qquad\leq
 \int_0^T|\theta_{M,k}-\theta_M|J\,dt
 +e^M\int_0^T|\theta_{M,k}-\theta_M|a\,dt
 \longrightarrow0.
\end{align*}
Therefore the supremum over $\theta\in C([0,T])$ is at least
$\int_0^T G_M(t)\,dt$ for every $M$.  Letting $M\to\infty$ in the
preceding monotone-convergence identity proves the reverse inequality and
completes the proof.
\end{proof}

We also use the dynamical action without fixing the initial point:
for $m\in D([0,T];\mathbb R_{\geq0}^{\mathcal I_h})$, set
\begin{align}\label{fixed-h-free-initial-action}
 I_h(m)
 &=\inf_{\substack{
 q=(q_\kappa)_{\kappa\in\mathcal K_h}:
 [0,T]\to\mathbb R_{\geq0}^{\mathcal K_h}
 \text{ measurable},\\
 (\beta_\kappa(m)q_\kappa)_{\kappa\in\mathcal K_h}
 \in L^1(0,T;\mathbb R^{\mathcal K_h})}}
 \left\{
 \int_0^T\sum_{\kappa\in\mathcal K_h}
 \beta_\kappa(m(t))\ell(q_\kappa(t))\,dt:
 \right.\notag\\
 &\hspace{28mm}\left.
 m\in AC([0,T];\mathbb R_{\geq0}^{\mathcal I_h}),\quad
 \dot m(t)=\sum_{\kappa\in\mathcal K_h}
 \beta_\kappa(m(t))q_\kappa(t)z_\kappa
 \ \text{for a.e. }t\in[0,T]
 \right\}.
\end{align}
As in \eqref{intro-fixed-h-rigorous-rate}, the infimum is taken over one
vector-valued control subject to the skeleton equation.  Here the initial
point $m(0)$ is free.  The functional $I_h$ need
not be good by itself, since the total
mass is not fixed; the initial entropy in
Theorem~\ref{thm:fixed-h-poisson-initial-ldp} restores coercivity.

\subsection{Large deviations}

We now prove the finite-mesh path LDP for deterministic initial data.  After
stating the result, we construct a marked-Poisson realization of the coarse
process and identify its law with that of the original random-time-change
model.  We then analyze entropy-bounded controlled processes: the first
control lemma supplies compactness and the liminf inequality, while the
second constructs recovery controls.  These ingredients are combined with
the Poisson variational representation in the proof of the theorem.

\begin{theorem}
\label{thm:fixed-h-rigorous-ldp}
Fix $h>0$ and suppose that
$m_0^h\in(0,\infty)^{\mathcal I_h}$ has total mass one.  Let
$m_0^{\varepsilon,h}\in\mathsf S_{\varepsilon,h}$ satisfy
$m_0^{\varepsilon,h}\to m_0^h$ as $\varepsilon\downarrow0$ along values
satisfying $\varepsilon^{-1}\in\mathbb N$, and let
$m^{\varepsilon,h}$ denote the coarse process with initial condition
\begin{equation*}
 m^{\varepsilon,h}(0)=m_0^{\varepsilon,h}.
\end{equation*}
Then the laws of $m^{\varepsilon,h}$ satisfy on
$D([0,T];\mathbb R_{\geq0}^{\mathcal I_h})$, with the Skorokhod $J_1$ topology,
an LDP with speed $\varepsilon^{-1}$ and good rate
$I_h^{m_0^h}$ defined by \eqref{intro-fixed-h-rigorous-rate}.
\end{theorem}

We first prepare the Poisson-random-measure representation used in the
proof.  Recall the notation
$\mathcal K_h:=\mathcal E_h\sqcup\Gamma_h$.  For every lattice state
$m\in\mathsf S_{\varepsilon,h}$, the rates of the coarse model from
Section~\ref{sec:coarse-setup-main}, defined in
\eqref{coarse-transport-rate}--\eqref{coarse-boltzmann-collision-rate},
satisfy
\begin{equation*}
 \lambda_\kappa^{\varepsilon,h}(m)
 =\varepsilon^{-1}\beta_\kappa^\varepsilon(m),
 \qquad \kappa\in\mathcal K_h,
\end{equation*}
where $\beta_\kappa^\varepsilon$ is defined in
\eqref{finite-particle-beta}.  For a diagonal collision channel, the
positive part in \eqref{finite-particle-beta} is redundant on the lattice
state space, so this identity also holds when the two incoming cells
coincide.

For the variational representation, we use a marked Poisson random measure
that includes the auxiliary mark $u$.  We construct it on a possibly
different probability space.  Lemma~\ref{lem:fixed-h-equivalent-prm-realization}
shows that the resulting mass process has the same law as the original
coarse process.  Set
\begin{align*}
 \nu_h(d\kappa,du)
 &:=\sum_{\kappa'\in\mathcal K_h}
 \delta_{\kappa'}(d\kappa)\,du,\\
 F_h^\varepsilon(m)
 &:=\sum_{\kappa\in\mathcal K_h}
 \beta_\kappa^\varepsilon(m)z_\kappa,
 &G_h^\varepsilon(m;\kappa,u)
 &:=z_\kappa
 \boldsymbol1_{\{0\leq u\leq\beta_\kappa^\varepsilon(m)\}}.
\end{align*}
Let $N_h^{\varepsilon^{-1}}$ be a Poisson random measure on
$[0,T]\times\mathcal K_h\times[0,\infty)$ with intensity
$\varepsilon^{-1}ds\,\nu_h(d\kappa,du)$.  For each
$\kappa\in\mathcal K_h$, denote the restriction of this random measure to
the channel slice $\{\kappa\}$ by
\begin{equation*}
 N_{h,\kappa}^{\varepsilon^{-1}}(A\times B)
 :=N_h^{\varepsilon^{-1}}(A\times\{\kappa\}\times B),
 \qquad
 A\in\mathcal B([0,T]),\quad B\in\mathcal B([0,\infty)).
\end{equation*}
Then $N_{h,\kappa}^{\varepsilon^{-1}}$ is a Poisson random measure on
$[0,T]\times[0,\infty)$ with intensity
$\varepsilon^{-1}ds\,du$, and these restrictions are independent for
distinct channels.  We define a c\`adl\`ag mass process
$\widehat m^{\varepsilon,h}$ and its accepted channel counts
$(\widehat P_\kappa^{\varepsilon,h})_{\kappa\in\mathcal K_h}$ recursively by
\begin{equation*}
 \widehat P_\kappa^{\varepsilon,h}(t)
 :=\int_{(0,t]\times[0,\infty)}
 \boldsymbol1_{\{0\leq u\leq
 \beta_\kappa^\varepsilon(\widehat m^{\varepsilon,h}(s-))\}}
 N_{h,\kappa}^{\varepsilon^{-1}}(ds,du),
 \qquad \kappa\in\mathcal K_h.
\end{equation*}
Together with these counts, set
\begin{equation}\label{fixed-h-prm-mass-process}
 \widehat m^{\varepsilon,h}(t)
 =m_0^{\varepsilon,h}
 +\varepsilon\sum_{\kappa\in\mathcal K_h}
 z_\kappa\widehat P_\kappa^{\varepsilon,h}(t).
\end{equation}
Between successive accepted marks the state is constant, so these equations
determine the next accepted mark and the resulting state recursively.  Since
there are finitely many channels and their total rate is bounded on the
unit-mass state space, the construction is nonexplosive and determines
$\widehat m^{\varepsilon,h}$ on all of $[0,T]$.  The predictable compensator
of $\widehat P_\kappa^{\varepsilon,h}$ is
\begin{align*}
 &\varepsilon^{-1}\int_0^t\int_0^\infty
 \boldsymbol1_{\{0\leq u\leq
 \beta_\kappa^\varepsilon(\widehat m^{\varepsilon,h}(s-))\}}\,du\,ds\\
 &\qquad=\varepsilon^{-1}\int_0^t
 \beta_\kappa^\varepsilon(\widehat m^{\varepsilon,h}(s-))\,ds
 =\int_0^t\lambda_\kappa^{\varepsilon,h}
 (\widehat m^{\varepsilon,h}(s-))\,ds.
\end{align*}

\begin{lemma}
\label{lem:fixed-h-equivalent-prm-realization}
The process $\widehat m^{\varepsilon,h}$ constructed in
\eqref{fixed-h-prm-mass-process} has the same law on
$D([0,T];\mathbb R_{\geq0}^{\mathcal I_h})$ as the coarse process
$m^{\varepsilon,h}$ defined by
\eqref{coarse-transport-counts}--\eqref{coarse-boltzmann-integral-equation}.
\end{lemma}

\begin{proof}
Suppose that the state has been constructed
up to a jump time $\tau_n$ and equals $m_n$ there.  Until the next jump, it
remains equal to $m_n$.  For each channel
$\kappa\in\mathcal K_h$, the original
random-time-change construction uses a unit-rate Poisson process whose
internal clock advances at speed
$\lambda_\kappa^{\varepsilon,h}(m_n)$.  Hence its candidate waiting time is
exponential with rate $\lambda_\kappa^{\varepsilon,h}(m_n)$.  In the present
Poisson-random-measure construction, the accepted region on that channel has
intensity
\begin{equation*}
 \varepsilon^{-1}\int_0^{\beta_\kappa^\varepsilon(m_n)}du
 =\varepsilon^{-1}\beta_\kappa^\varepsilon(m_n)
 =\lambda_\kappa^{\varepsilon,h}(m_n),
\end{equation*}
so its candidate waiting time has the same exponential law.  In both
constructions these candidate times are independent over the finitely many
channels.  Consequently, conditional on the history up to $\tau_n$, the
time and channel of the next jump have the same joint law.  These conditional
transition laws, together with the common initial law, recursively determine
the joint law of the channel counts and of the mass process.  The uniform
bound on the total rate from Lemma~\ref{lem:coarse-total-intensity} rules out
accumulation of jump times, so the recursion continues on the whole interval
$[0,T]$.  The Poisson-random-measure representation and the original
random-time-change representation are therefore equivalent in law.
\end{proof}

From now on, we work with the Poisson-random-measure realization and denote it
again by $m^{\varepsilon,h}$.  Using the accepted counts, the integral form
\eqref{coarse-boltzmann-integral-equation} becomes

\begin{align*}
 m_t^{\varepsilon,h}
 &=m_0^{\varepsilon,h}
 +\varepsilon\sum_{\kappa\in\mathcal K_h}
 z_\kappa\widehat P_\kappa^{\varepsilon,h}(t)\\
 &=m_0^{\varepsilon,h}
 +\varepsilon\int_0^t
 \int_{(\mathcal E_h\sqcup\Gamma_h)\times[0,\infty)}
 G_h^\varepsilon(m_{s-}^{\varepsilon,h};\kappa,u)
 N_h^{\varepsilon^{-1}}(ds,d\kappa,du).
\end{align*}
Now set
\begin{equation*}
 \widetilde N_h^{\varepsilon^{-1}}(dt,d\kappa,du)
 :=N_h^{\varepsilon^{-1}}(dt,d\kappa,du)
 -\varepsilon^{-1}dt\,\nu_h(d\kappa,du).
\end{equation*}
Since
\begin{equation*}
 \int_{(\mathcal E_h\sqcup\Gamma_h)\times[0,\infty)}
 G_h^\varepsilon(m;\kappa,u)
 \nu_h(d\kappa,du)
 =\sum_{\kappa\in\mathcal K_h}
 \beta_\kappa^\varepsilon(m)z_\kappa
 =F_h^\varepsilon(m),
\end{equation*}
compensating the preceding Poisson integral gives
\begin{align*}
 &\varepsilon\int_0^t
 \int_{(\mathcal E_h\sqcup\Gamma_h)\times[0,\infty)}
 G_h^\varepsilon(m_{s-}^{\varepsilon,h};\kappa,u)
 N_h^{\varepsilon^{-1}}(ds,d\kappa,du)\\
 &\quad=\varepsilon\int_0^t
 \int_{(\mathcal E_h\sqcup\Gamma_h)\times[0,\infty)}
 G_h^\varepsilon(m_{s-}^{\varepsilon,h};\kappa,u)
 \widetilde N_h^{\varepsilon^{-1}}(ds,d\kappa,du)
 +\int_0^tF_h^\varepsilon(m_s^{\varepsilon,h})\,ds.
\end{align*}
Here $m_{s-}^{\varepsilon,h}$ may be replaced by
$m_s^{\varepsilon,h}$ in the Lebesgue integral because the two differ only
at jump times.  We have therefore derived the exact compensated
representation
\begin{align}\label{fixed-h-BCD-representation}
 m_t^{\varepsilon,h}
 ={}&m_0^{\varepsilon,h}
 +\int_0^tF_h^\varepsilon(m_s^{\varepsilon,h})\,ds\notag\\
 &+\varepsilon\int_0^t
 \int_{(\mathcal E_h\sqcup\Gamma_h)\times[0,\infty)}
 G_h^\varepsilon(m_{s-}^{\varepsilon,h};\kappa,u)
 \widetilde N_h^{\varepsilon^{-1}}(ds,d\kappa,du).
\end{align}
This is the finite-dimensional analogue of
\eqref{BCD-abstract-Poisson-SPDE}.  Its coefficients depend on
$\varepsilon$ through the falling-factorial correction; consequently, we
will use \eqref{uniform-beta-epsilon-limit} when passing to the limit in the
controlled equations.

We next introduce the controlled Poisson random measures used in the
compactness and recovery arguments.  We use the controlled-Poisson construction
\eqref{BCD-controlled-PRM} with mark space
$\mathcal K_h\times[0,\infty)$ and intensity measure $\nu_h$.  Thus, on the
canonical augmented Poisson space of
Appendix~\ref{app:poisson-spde-ldp-framework}, a
control changes the intensity of the mark $(\kappa,u)$ without changing the
graphical jump rule.  Let $\mathcal P_h$ denote the predictable sigma-field
of the filtration on this augmented space.  Recall from
\eqref{poisson-entropy-density} that $\ell(r)=r\log r-r+1$, $r\geq0$,
with the convention $0\log0=0$.  Define the admissible control class by
\begin{align*}
 \mathcal A_h:=\Bigg\{&\varphi=(\varphi_\kappa)_{\kappa\in\mathcal K_h}:
 \varphi_\kappa:\Omega\times[0,T]\times[0,\infty)\to[0,\infty)
 \text{ is }\mathcal P_h\otimes\mathcal B([0,\infty))\text{-measurable},\\
 &\mathbb E\int_0^T\sum_{\kappa\in\mathcal K_h}\int_0^\infty
 \ell(\varphi_\kappa(t,u))\,du\,dt<\infty\Bigg\}.
\end{align*}
An element of $\mathcal A_h$ is therefore a nonnegative predictable mark
control with one component for each channel.  For
$\varphi\in\mathcal A_h$, let $N_h^{\varepsilon^{-1}\varphi}$ denote the
controlled Poisson random measure defined by
\eqref{BCD-controlled-PRM}.  In the coordinates $(t,\kappa,u)$, its
predictable compensator is
\begin{equation*}
 \varepsilon^{-1}\varphi_\kappa(t,u)dt\,\nu_h(d\kappa,du).
\end{equation*}

The recursive construction in the proof of
Theorem~\ref{thm:coarse-boltzmann-wellposedness} defines a measurable
graphical solution map $\mathcal G_{\varepsilon,h}$.  For a
locally finite marked measure $n$ on
$[0,T]\times\mathcal K_h\times[0,\infty)$, the path
$g=\mathcal G_{\varepsilon,h}(n)$ is the unique c\`adl\`ag solution of
\begin{align}
 g_t
 ={}&m_0^{\varepsilon,h}
 +\varepsilon\int_{(0,t]\times\mathcal K_h\times[0,\infty)}
 z_\kappa
 \boldsymbol1_{\{0\leq u\leq
 \beta_\kappa^\varepsilon(g_{s-})\}}
 n(ds,d\kappa,du).
 \label{fixed-h-graphical-solution-map}
\end{align}
The original and controlled mass processes are therefore represented by
\begin{align}
 m^{\varepsilon,h}
 &\stackrel{\mathrm{law}}=
 \mathcal G_{\varepsilon,h}(N_h^{\varepsilon^{-1}}),
 &m^{\varepsilon,h,\varphi}
 &:=\mathcal G_{\varepsilon,h}
 (N_h^{\varepsilon^{-1}\varphi}).
 \label{fixed-h-original-controlled-solution-maps}
\end{align}
For a fixed $\kappa\in\mathcal K_h$, the control $\varphi_\kappa$ changes
the intensity of the available marks on channel $\kappa$, while the indicator in
\eqref{fixed-h-graphical-solution-map} retains only the marks accepted by
the state-dependent rate.  In particular, the accepted channel-$\kappa$
marked count has predictable compensator
\begin{equation*}
 \varepsilon^{-1}
 \boldsymbol1_{\{0\leq u\leq
 \beta_\kappa^\varepsilon(m_{t-}^{\varepsilon,h,\varphi})\}}
 \varphi_\kappa(t,u)dt\,du,
\end{equation*}
and hence total predictable jump rate
\begin{equation*}
 \varepsilon^{-1}\int_0^{\beta_\kappa^\varepsilon(
 m_{t-}^{\varepsilon,h,\varphi})}
 \varphi_\kappa(t,u)\,du.
\end{equation*}
The finite-cost condition defining $\mathcal A_h$ makes this rate locally
integrable.  The control changes only the frequency of the proposed jumps,
not their jump vectors or acceptance rules.  Therefore
Lemma~\ref{lem:coarse-invariant-state-space} shows that the controlled path
remains in the unit-mass state space.  The elementary inequality
$r\leq\ell(r)+e-1$ for $r\geq0$, together with the uniform boundedness of
$\beta_\kappa^\varepsilon$ on the unit-mass simplex, bounds the expected
total accepted intensity on $[0,T]$.  Therefore the controlled graphical
equation is nonexplosive.
For $\varphi_\kappa\equiv1$ for every $\kappa\in\mathcal K_h$,
the two representations in
\eqref{fixed-h-original-controlled-solution-maps} have the same law.

The proof of Theorem~\ref{thm:fixed-h-rigorous-ldp} will apply the Poisson
variational representation to this graphical solution map.  Before doing
so, we establish the two estimates needed to pass to the small-noise limit:
first, compactness and a lower bound for arbitrary sequences of controls
with bounded entropy cost; second, a recovery construction for any path of
finite rate.  For a predictable mark control
$\varphi\in\mathcal A_h$, define its effective
channel flux and multiplier by
\begin{align}
 J_\kappa^{\varepsilon,\varphi}(t)
 &:=\int_0^{\beta_\kappa^\varepsilon(
 m^{\varepsilon,h,\varphi}(t))}
 \varphi_\kappa(t,u)\,du,
 \label{effective-mark-control-flux}\\
 q_\kappa^{\varepsilon,\varphi}(t)
 &:=
 \begin{cases}
 \displaystyle
 \frac{J_\kappa^{\varepsilon,\varphi}(t)}
 {\beta_\kappa^\varepsilon(m^{\varepsilon,h,\varphi}(t))},
 &\beta_\kappa^\varepsilon(m^{\varepsilon,h,\varphi}(t))>0,\\[3mm]
 1,&\beta_\kappa^\varepsilon(m^{\varepsilon,h,\varphi}(t))=0,
 \end{cases}
 \qquad \kappa\in\mathcal K_h.
 \label{effective-mark-control-multiplier}
\end{align}
Here $J_\kappa^{\varepsilon,\varphi}$ is the controlled channel intensity
after integration over the acceptance mark $u$ and removal of the common
factor $\varepsilon^{-1}$, while
$q_\kappa^{\varepsilon,\varphi}$ is its ratio to the nominal intensity.

The next lemma identifies every subsequential limit of controlled processes
whose expected entropy costs remain bounded.  It provides the compactness
and lower-bound ingredient in the Laplace principle.

\begin{lemma}
\label{lem:fixed-h-controlled-compactness-liminf}
Let $\varepsilon_n\downarrow0$ and let
$\varphi^n\in\mathcal A_h$ satisfy
\begin{equation}\label{fixed-h-uniform-controlled-entropy-bound}
 \sup_{n\geq1}\mathbb E\left[
 \int_0^T\sum_{\kappa\in\mathcal K_h}\int_0^\infty
 \ell(\varphi_\kappa^n(t,u))\,du\,dt\right]<\infty.
\end{equation}
Then the laws of
\begin{equation*}
 \left(m^{\varepsilon_n,h,\varphi^n},
 \bigl(J_\kappa^{\varepsilon_n,\varphi^n}(t)\,dt
 \bigr)_{\kappa\in\mathcal K_h}\right)
\end{equation*}
are tight.  If a subsequence converges in distribution to
$\bigl(m,(J_\kappa(t)dt)_{\kappa\in\mathcal K_h}\bigr)$, then, with
\begin{equation}\label{fixed-h-limit-control-defined}
 q_\kappa(t):=
 \begin{cases}
 J_\kappa(t)/\beta_\kappa(m(t)),&\beta_\kappa(m(t))>0,\\
 1,&\beta_\kappa(m(t))=0,
 \end{cases}
 \qquad \kappa\in\mathcal K_h,
\end{equation}
the limit path $m$ satisfies the controlled equation
\eqref{intro-fixed-h-rigorous-continuity} with control $(q_\kappa)_{\kappa\in\mathcal K_h}$; moreover, the limiting
control cost satisfies
\begin{align}
 \mathbb E\int_0^T\sum_{\kappa\in\mathcal K_h}
 \beta_\kappa(m(t))\ell(q_\kappa(t))\,dt
 \leq\liminf_{n\to\infty}
 \mathbb E\int_0^T\sum_{\kappa\in\mathcal K_h}\int_0^\infty
 \ell(\varphi_\kappa^n(t,u))\,du\,dt.
 \label{fixed-h-controlled-condition-cost-liminf}
\end{align}
\end{lemma}

\begin{proof}
Fix $\varphi\in\mathcal A_h$.  For each $\kappa\in\mathcal K_h$, define
\begin{equation}\label{control-restricted-to-acceptance-interval}
 \widehat\varphi_\kappa(t,u)
 :=\begin{cases}
 \varphi_\kappa(t,u),
 &0\leq u\leq
 \beta_\kappa^\varepsilon(m^{\varepsilon,h,\varphi}(t-)),\\
 1,&u>\beta_\kappa^\varepsilon
 (m^{\varepsilon,h,\varphi}(t-)).
 \end{cases}
\end{equation}
The process $m^{\varepsilon,h,\varphi}(t-)$ is predictable, so
$\widehat\varphi\in\mathcal A_h$.  Moreover,
\begin{equation*}
 \boldsymbol1_{\{0\leq u\leq
 \beta_\kappa^\varepsilon(m^{\varepsilon,h,\varphi}(t-))\}}
 \widehat\varphi_\kappa(t,u)
 =
 \boldsymbol1_{\{0\leq u\leq
 \beta_\kappa^\varepsilon(m^{\varepsilon,h,\varphi}(t-))\}}
 \varphi_\kappa(t,u).
\end{equation*}
Thus the original controlled path also solves the graphical equation
\eqref{fixed-h-graphical-solution-map} with control
$\widehat\varphi$; uniqueness of the recursive graphical
construction gives
$m^{\varepsilon,h,\widehat\varphi}
=m^{\varepsilon,h,\varphi}$ almost surely.  Since $\ell\geq0$ and
$\ell(1)=0$,
\begin{align*}
 \int_0^\infty\ell(\widehat\varphi_\kappa(t,u))\,du
 &=\int_0^{\beta_\kappa^\varepsilon(
 m^{\varepsilon,h,\varphi}(t-))}
 \ell(\varphi_\kappa(t,u))\,du\\
 &\leq\int_0^\infty\ell(\varphi_\kappa(t,u))\,du.
\end{align*}
Consequently, we may assume without loss of generality that the control
equals one outside the current acceptance interval.
By \eqref{effective-mark-control-flux}--\eqref{effective-mark-control-multiplier},
$q_\kappa^{\varepsilon,\varphi}(t)$ is the average of
$\varphi_\kappa(t,\cdot)$ over the current acceptance interval whenever
that interval has positive length.  Jensen's inequality for the convex
function $\ell$, applied with respect to normalized Lebesgue measure on
that interval, therefore gives
\begin{equation}\label{mark-control-jensen}
 \int_0^\infty\ell(\varphi_\kappa(t,u))\,du
 \geq \beta_\kappa^\varepsilon(m^{\varepsilon,h,\varphi}(t))
 \ell(q_\kappa^{\varepsilon,\varphi}(t)).
\end{equation}
When the acceptance interval has length zero,
$q_\kappa^{\varepsilon,\varphi}(t)=1$ by definition, and both sides of
\eqref{mark-control-jensen} are zero after the preceding normalization of
the control.

For the sequence of controls satisfying
\eqref{fixed-h-uniform-controlled-entropy-bound},
\eqref{mark-control-jensen} gives uniform integrability of the effective
fluxes $J_\kappa^{\varepsilon_n,\varphi^n}$ on
$\Omega\times[0,T]$.  Indeed, for every fixed
$\kappa\in\mathcal K_h$, the nominal rates satisfy
\begin{equation*}
 B_{\kappa,h}:=
 \sup_{0<\varepsilon\leq1}\ 
 \sup_{\substack{m\in\mathbb R_{\geq0}^{\mathcal I_h}\\
                  \sum_{a\in\mathcal I_h}m_a=1}}
 \beta_\kappa^\varepsilon(m)<\infty.
\end{equation*}
Since
$J_\kappa^{\varepsilon_n,\varphi^n}
=\beta_\kappa^{\varepsilon_n}(m^{\varepsilon_n,h,\varphi^n})
q_\kappa^{\varepsilon_n,\varphi^n}$, whenever the nominal rate is
positive we have
\begin{align*}
 &\beta_\kappa^{\varepsilon_n}(m^{\varepsilon_n,h,\varphi^n})
 \ell(q_\kappa^{\varepsilon_n,\varphi^n})\\
 &\quad=
 J_\kappa^{\varepsilon_n,\varphi^n}
 \log\left(\frac{J_\kappa^{\varepsilon_n,\varphi^n}}
 {\beta_\kappa^{\varepsilon_n}(m^{\varepsilon_n,h,\varphi^n})}\right)
 -J_\kappa^{\varepsilon_n,\varphi^n}
 +\beta_\kappa^{\varepsilon_n}(m^{\varepsilon_n,h,\varphi^n}).
\end{align*}
Therefore, on the set
$\{J_\kappa^{\varepsilon_n,\varphi^n}(t)>R\}$, with
$R>eB_{\kappa,h}$, we have
\begin{align*}
 \beta_\kappa^{\varepsilon_n}(m^{\varepsilon_n,h,\varphi^n}(t))
 \ell(q_\kappa^{\varepsilon_n,\varphi^n}(t))\geq
 \left(\log\frac{R}{B_{\kappa,h}}-1\right)
 J_\kappa^{\varepsilon_n,\varphi^n}(t).
\end{align*}
Consequently, \eqref{mark-control-jensen} and
\eqref{fixed-h-uniform-controlled-entropy-bound} imply
\begin{align*}
 &\left(\log\frac{R}{B_{\kappa,h}}-1\right)
 \sup_{n\geq1}\mathbb E\int_0^T
 J_\kappa^{\varepsilon_n,\varphi^n}(t)
 \boldsymbol1_{\{J_\kappa^{\varepsilon_n,\varphi^n}(t)>R\}}\,dt\\
 &\quad\leq
 \sup_{n\geq1}\mathbb E\int_0^T
 \int_0^\infty\ell(\varphi_\kappa^n(t,u))\,du\,dt<\infty.
\end{align*}
Because the coefficient on the left tends to infinity with $R$, this
proves
\begin{equation*}
 \lim_{R\to\infty}\sup_{n\geq1}\mathbb E\int_0^T
 J_\kappa^{\varepsilon_n,\varphi^n}(t)
 \boldsymbol1_{\{J_\kappa^{\varepsilon_n,\varphi^n}(t)>R\}}\,dt=0.
\end{equation*}
The controlled paths have the decomposition
\begin{equation}\label{fixed-h-controlled-semimartingale}
 m^{\varepsilon_n,h,\varphi^n}(t)=m_0^{\varepsilon_n,h}
 +\int_0^t\sum_{\kappa\in\mathcal K_h}
 J_\kappa^{\varepsilon_n,\varphi^n}(s)z_\kappa\,ds
 +R^{\varepsilon_n,\varphi^n}(t).
\end{equation}
The martingale $R^{\varepsilon_n,\varphi^n}$ is given by the following
compensated integral.  Set
\begin{align*}
 \widetilde N_h^{\varepsilon_n^{-1}\varphi^n}
 (ds,d\kappa,du):=N_h^{\varepsilon_n^{-1}\varphi^n}(ds,d\kappa,du)
 -\varepsilon_n^{-1}\varphi_\kappa^n(s,u)
 \,ds\,\nu_h(d\kappa,du).
\end{align*}
Then
\begin{align}
 R^{\varepsilon_n,\varphi^n}(t)
 :={}&\varepsilon_n
 \int_{(0,t]\times\mathcal K_h\times[0,\infty)}
 z_\kappa
 \boldsymbol1_{\{0\leq u\leq
 \beta_\kappa^{\varepsilon_n}(
 m^{\varepsilon_n,h,\varphi^n}(s-))\}}
 \widetilde N_h^{\varepsilon_n^{-1}\varphi^n}
 (ds,d\kappa,du).
 \label{fixed-h-controlled-martingale-explicit}
\end{align}
It is an $\mathbb R^{\mathcal I_h}$-valued square-integrable martingale,
and its predictable quadratic-covariation matrix is
\begin{align*}
 \big\langle R^{\varepsilon_n,\varphi^n}\big\rangle_t
 =\varepsilon_n\int_0^t\sum_{\kappa\in\mathcal K_h}
 J_\kappa^{\varepsilon_n,\varphi^n}(s)
 z_\kappa z_\kappa^{\mathsf T}\,ds.
\end{align*}
Since the jump vectors are fixed and finite in number, the
Burkholder--Davis--Gundy inequality and the last identity give
\begin{equation}\label{fixed-h-controlled-martingale-vanishing}
 \mathbb E\sup_{t\leq T}|R^{\varepsilon_n,\varphi^n}(t)|^2
 \leq C_h\varepsilon_n\,
 \mathbb E\int_0^T\sum_{\kappa\in\mathcal K_h}
 J_\kappa^{\varepsilon_n,\varphi^n}(s)\,ds
 \longrightarrow0
 \qquad\text{as }n\to\infty.
\end{equation}
The expectation-valued version of the uniform-integrability compactness
argument in the proof of
Lemma~\ref{lem:fixed-h-action-good-rigorous}, together with
\eqref{fixed-h-controlled-martingale-vanishing}, shows that the joint laws
of
\begin{equation*}
 \left(m^{\varepsilon_n,h,\varphi^n},
 \bigl(J_\kappa^{\varepsilon_n,\varphi^n}(t)\,dt
 \bigr)_{\kappa\in\mathcal K_h}\right)
\end{equation*}
are tight on
\begin{equation*}
 D([0,T];\mathbb R^{\mathcal I_h})
 \times\mathcal M_+([0,T])^{\mathcal K_h}.
\end{equation*}
Here the path space carries the Skorokhod $J_1$ topology and the measure
space carries the topology of weak convergence.  The same argument shows
that the path laws are $C$-tight and that every limiting flux measure is
almost surely absolutely continuous with respect to $dt$.  Hence, after
passing to a subsequence and using a Skorokhod representation, the limit
has the form
$(m,(J_\kappa(t)dt)_{\kappa\in\mathcal K_h})$, where $m$ is continuous,
and the path convergence is uniform.  We may therefore pass to the limit
in \eqref{fixed-h-controlled-semimartingale}; using
\eqref{fixed-h-controlled-martingale-vanishing} gives
\begin{equation}\label{controlled-limit-effective-flux-equation}
 m(t)=m_0^h+\sum_{\kappa\in\mathcal K_h}z_\kappa
 \int_0^tJ_\kappa(s)\,ds.
\end{equation}
The limiting equation is first written in terms of the effective fluxes
$J_\kappa$.  Define the corresponding multiplier $q_\kappa$ from
$(m,J_\kappa)$ by
\eqref{fixed-h-limit-control-defined}.  It remains to verify that
$J_\kappa=0$ almost everywhere on
$\{t:\beta_\kappa(m(t))=0\}$; the entropy lower bound below gives
this fact and hence shows that
$J_\kappa=\beta_\kappa(m)q_\kappa$ almost everywhere.
For $\kappa\in\mathcal K_h$, set
\begin{equation*}
 a_{\kappa,n}(t):=
 \beta_\kappa^{\varepsilon_n}
 (m^{\varepsilon_n,h,\varphi^n}(t)),
 \qquad
 a_\kappa(t):=\beta_\kappa(m(t)).
\end{equation*}
The uniform convergence of the paths, continuity of $\beta_\kappa$, and
\eqref{uniform-beta-epsilon-limit} imply
\begin{equation*}
 \sup_{t\in[0,T]}|a_{\kappa,n}(t)-a_\kappa(t)|
 \longrightarrow0
 \qquad\text{almost surely}.
\end{equation*}
Moreover, weak convergence of the flux measures gives, for every
$\theta\in C([0,T])$,
\begin{equation*}
 \int_0^T\theta(t)J_\kappa^{\varepsilon_n,\varphi^n}(t)\,dt
 \longrightarrow
 \int_0^T\theta(t)J_\kappa(t)\,dt
 \qquad\text{almost surely}.
\end{equation*}
Consequently,
\begin{align*}
 &\int_0^T\left\{
 \theta(t)J_\kappa^{\varepsilon_n,\varphi^n}(t)
 -a_{\kappa,n}(t)(e^{\theta(t)}-1)\right\}\,dt\\
 &\quad\longrightarrow
 \int_0^T\left\{
 \theta(t)J_\kappa(t)
 -a_\kappa(t)(e^{\theta(t)}-1)\right\}\,dt
 \qquad\text{almost surely}.
\end{align*}
Taking the supremum over $\theta$ in the dual representation
\eqref{integral-relative-flux-duality}, and then summing over the finitely
many channels, yields the pathwise lower-semicontinuity estimate
\begin{align*}
 \int_0^T\sum_{\kappa\in\mathcal K_h}
 a_\kappa(t)\ell\left(\frac{J_\kappa(t)}{a_\kappa(t)}\right)\,dt\leq\liminf_{n\to\infty}
 \int_0^T\sum_{\kappa\in\mathcal K_h}
 a_{\kappa,n}(t)
 \ell\left(\frac{J_\kappa^{\varepsilon_n,\varphi^n}(t)}
 {a_{\kappa,n}(t)}\right)\,dt.
\end{align*}
Here the extended-value convention in
Lemma~\ref{lem:integral-relative-flux-duality} is used when a nominal rate
vanishes.  By \eqref{effective-mark-control-multiplier}, the integrand on
the right is
$a_{\kappa,n}\ell(q_\kappa^{\varepsilon_n,\varphi^n})$.
Fatou's lemma, followed by \eqref{mark-control-jensen}, therefore gives
\begin{align*}
 \mathbb E\int_0^T\sum_{\kappa\in\mathcal K_h}
 a_\kappa(t)\ell\left(\frac{J_\kappa(t)}{a_\kappa(t)}\right)\,dt&\leq\liminf_{n\to\infty}
 \mathbb E\int_0^T\sum_{\kappa\in\mathcal K_h}
 a_{\kappa,n}(t)
 \ell(q_\kappa^{\varepsilon_n,\varphi^n}(t))\,dt\\
 &\leq\liminf_{n\to\infty}
 \mathbb E\int_0^T\sum_{\kappa\in\mathcal K_h}\int_0^\infty
 \ell(\varphi_\kappa^n(t,u))\,du\,dt.
\end{align*}
The last quantity is finite by
\eqref{fixed-h-uniform-controlled-entropy-bound}; hence
$J_\kappa=0$ almost everywhere on $\{a_\kappa=0\}$.  For the limiting
multiplier $q$ defined in \eqref{fixed-h-limit-control-defined}, we thus
have $J_\kappa=a_\kappa q_\kappa$ almost everywhere, and the preceding
inequality becomes
\begin{equation}\label{fixed-h-expected-action-liminf}
 \mathbb E\int_0^T\sum_{\kappa\in\mathcal K_h}
 \beta_\kappa(m(t))\ell(q_\kappa(t))\,dt
 \leq\liminf_{n\to\infty}
 \mathbb E\int_0^T\sum_{\kappa\in\mathcal K_h}\int_0^\infty
 \ell(\varphi_\kappa^n(t,u))\,du\,dt.
\end{equation}
Substituting
$J_\kappa=\beta_\kappa(m)q_\kappa$ into
\eqref{controlled-limit-effective-flux-equation} gives
\begin{equation*}
 m(t)=m_0^h+\sum_{\kappa\in\mathcal K_h}z_\kappa
 \int_0^t\beta_\kappa(m(s))q_\kappa(s)\,ds,
 \qquad t\in[0,T],
\end{equation*}
which is the integral form of
\eqref{intro-fixed-h-rigorous-continuity} with control
$(q_\kappa)_{\kappa\in\mathcal K_h}$.  Hence the limit is a solution of the
deterministic skeleton equation obtained after the martingale part vanishes
as $\varepsilon\to0$.  Every limit of a
bounded-cost sequence is an admissible controlled path in the rate
function \eqref{intro-fixed-h-rigorous-rate} and satisfies the required
cost lower bound.
\end{proof}

The preceding lemma provides the compactness and liminf estimates.  For the
opposite Laplace bound, we construct predictable mark controls whose
paths and entropy costs approximate any prescribed path of finite rate.

\begin{lemma}
\label{lem:fixed-h-controlled-recovery}
For every $m$ with $I_h^{m_0^h}(m)<\infty$ and every $\eta>0$, there is a
family $(\varphi^\varepsilon)_\varepsilon\subset\mathcal A_h$, indexed by
$\varepsilon\downarrow0$ with $\varepsilon^{-1}\in\mathbb N$, such that
\begin{equation*}
 m^{\varepsilon,h,\varphi^\varepsilon}\Longrightarrow m
 \quad\text{in }
 D([0,T];\mathbb R_{\geq0}^{\mathcal I_h})
 \text{ endowed with the Skorokhod $J_1$ topology},
\end{equation*}
and
\begin{equation*}
 \limsup_{\varepsilon\downarrow0}\mathbb E\left[
 \int_0^T\sum_{\kappa\in\mathcal K_h}\int_0^\infty
 \ell(\varphi_\kappa^\varepsilon(t,u))\,du\,dt\right]
 \leq I_h^{m_0^h}(m)+\eta.
\end{equation*}
\end{lemma}

\begin{proof}
By the definition of $I_h^{m_0^h}(m)$, choose an admissible multiplier
$q=(q_\kappa)_{\kappa\in\mathcal K_h}$ such that
\begin{equation*}
 \int_0^T\sum_{\kappa\in\mathcal K_h}
 \beta_\kappa(m(t))\ell(q_\kappa(t))\,dt
 \leq I_h^{m_0^h}(m)+\eta.
\end{equation*}
Apply
Lemma~\ref{lem:fixed-h-vacuum-regularization} with
$\bar m=m_0^h$, and recall that the regularized path $m^\delta$ is defined
by \eqref{fixed-h-buffer-path}.  For fixed $\delta>0$ all
coordinates of $m^\delta$ are bounded away from zero.  The regularized
control $q^\delta$ is defined by \eqref{fixed-h-buffer-control}.  Since
$m_a^\delta(t)\geq\delta m_{0,a}^h$ for every
$a\in\mathcal I_h$ and every $t\in[0,T]$, each active channel has a
strictly positive nominal rate along $m^\delta$.  More precisely,
\begin{equation*}
 \inf_{t\in[0,T]}\beta_\kappa(m^\delta(t))
 \geq
 \begin{cases}
  \delta m_{0,a}^hT_{ab}^h,
  &\kappa=(a,b)\in\mathcal E_h,\\[1mm]
  \displaystyle\frac{\delta^2}{2}K_\kappa^h
  m_{0,a}^hm_{0,b}^h,
  &\kappa=(a,b;c,d)\in\Gamma_h,
 \end{cases}
 \quad >0.
\end{equation*}
The estimate proved in
Lemma~\ref{lem:fixed-h-vacuum-regularization} gives
\begin{equation*}
 \int_0^T\beta_\kappa(m^\delta(t))
 \ell(q_\kappa^\delta(t))\,dt<\infty,
\end{equation*}
and hence the preceding lower bound implies
$\ell(q_\kappa^\delta)\in L^1(0,T)$.  Finally, the scalar inequality
$r\leq\ell(r)+e-1$ for $r\geq0$ yields
\begin{equation*}
 q_\kappa^\delta,\ \ell(q_\kappa^\delta)\in L^1(0,T),
 \qquad \kappa\in\mathcal K_h.
\end{equation*}
Moreover, since the nominal rates are bounded on the unit simplex,
\begin{equation*}
 \int_0^T\sup_{m:\,\sum_am_a=1}
 \beta_\kappa^\varepsilon(m)q_\kappa^\delta(t)\,dt<\infty,
 \qquad \kappa\in\mathcal K_h.
\end{equation*}
Therefore the controlled graphical equation is nonexplosive.

For fixed $\varepsilon,\delta>0$, define a feedback mark control jointly
with its controlled path $m^{\varepsilon,h,\delta}$ by
\begin{align}
 \varphi_\kappa^{\varepsilon,\delta}(t,u)
 :=q_\kappa^\delta(t)
 \boldsymbol1_{\{0\leq u\leq
 \beta_\kappa^\varepsilon(m^{\varepsilon,h,\delta}(t-))\}}
 +\boldsymbol1_{\{u>
 \beta_\kappa^\varepsilon(m^{\varepsilon,h,\delta}(t-))\}},
 \qquad \kappa\in\mathcal K_h,
 \label{fixed-h-recovery-mark-control}
\end{align}
where
$m^{\varepsilon,h,\delta}:=
m^{\varepsilon,h,\varphi^{\varepsilon,\delta}}$ is the solution of the
graphical equation \eqref{fixed-h-graphical-solution-map}.  This control is
predictable because the left-limit process
$m^{\varepsilon,h,\delta}(t-)$ is predictable.  Moreover,
\begin{align}
 \int_0^{\beta_\kappa^\varepsilon(
 m^{\varepsilon,h,\delta}(t-))}
 \varphi_\kappa^{\varepsilon,\delta}(t,u)\,du
 &=\beta_\kappa^\varepsilon(m^{\varepsilon,h,\delta}(t-))
 q_\kappa^\delta(t),
 \label{fixed-h-recovery-effective-flux}\\
 \int_0^\infty
 \ell(\varphi_\kappa^{\varepsilon,\delta}(t,u))\,du
 &=\beta_\kappa^\varepsilon(m^{\varepsilon,h,\delta}(t-))
 \ell(q_\kappa^\delta(t)),
 \label{fixed-h-recovery-mark-cost}
\end{align}
where the second identity uses $\ell(1)=0$.  The boundedness of the nominal
rates on the unit-mass simplex and
$\ell(q_\kappa^\delta)\in L^1(0,T)$ show that
$\varphi^{\varepsilon,\delta}\in\mathcal A_h$.

To derive the equation for the controlled path explicitly, write
\begin{align*}
 \widetilde N_h^{\varepsilon^{-1}\varphi^{\varepsilon,\delta}}
 (ds,d\kappa,du)
 :={}&N_h^{\varepsilon^{-1}\varphi^{\varepsilon,\delta}}
 (ds,d\kappa,du)\\
 &-\varepsilon^{-1}
 \varphi_\kappa^{\varepsilon,\delta}(s,u)
 \,ds\,\nu_h(d\kappa,du),
\end{align*}
and set
\begin{align*}
 R^{\varepsilon,\delta}(t)
 :={}&\varepsilon
 \int_{(0,t]\times\mathcal K_h\times[0,\infty)}
 z_\kappa
 \boldsymbol1_{\{0\leq u\leq
 \beta_\kappa^\varepsilon(
 m^{\varepsilon,h,\delta}(s-))\}}
 \widetilde N_h^{\varepsilon^{-1}\varphi^{\varepsilon,\delta}}
 (ds,d\kappa,du).
\end{align*}
This is the compensated martingale part of the graphical equation
\eqref{fixed-h-graphical-solution-map}.  Separating the controlled Poisson
random measure into its compensated part and its predictable compensator
gives
\begin{align*}
 m^{\varepsilon,h,\delta}(t)
 ={}&m_0^{\varepsilon,h}+R^{\varepsilon,\delta}(t)\\
 &+\int_0^t\sum_{\kappa\in\mathcal K_h}
 z_\kappa\left[
 \int_0^{\beta_\kappa^\varepsilon(
 m^{\varepsilon,h,\delta}(s-))}
 \varphi_\kappa^{\varepsilon,\delta}(s,u)\,du
 \right]ds.
\end{align*}
The factor $\varepsilon$ in each jump cancels the factor
$\varepsilon^{-1}$ in the compensator.  By
\eqref{fixed-h-recovery-effective-flux}, the expression in square brackets
is
\begin{equation*}
 \beta_\kappa^\varepsilon(m^{\varepsilon,h,\delta}(s-))
 q_\kappa^\delta(s).
\end{equation*}
Since a c\`adl\`ag path and its left-limit path differ only at its jump
times, the left limit may be replaced by the value at $s$ in the Lebesgue
integral.  We therefore obtain
\begin{align}
 m^{\varepsilon,h,\delta}(t)
 ={}&m_0^{\varepsilon,h}
 +\int_0^t\sum_{\kappa\in\mathcal K_h}
 \beta_\kappa^\varepsilon(m^{\varepsilon,h,\delta}(s))
 q_\kappa^\delta(s)z_\kappa\,ds
 +R^{\varepsilon,\delta}(t),
 \label{fixed-h-recovery-controlled-equation}
\end{align}
which is \eqref{fixed-h-controlled-semimartingale} specialized to the
recovery control.  The predictable quadratic covariation of
$R^{\varepsilon,\delta}$ gives
\begin{equation}
 \mathbb E\sup_{t\leq T}|R^{\varepsilon,\delta}(t)|^2
 \leq C_h\varepsilon\int_0^T
 \sum_{\kappa\in\mathcal K_h}q_\kappa^\delta(s)\,ds
 \longrightarrow0.
 \label{fixed-h-recovery-martingale-bound}
\end{equation}
On the other hand, the skeleton equation for $(m^\delta,q^\delta)$ reads
\begin{equation}
 m^\delta(t)=m_0^h+\int_0^t
 \sum_{\kappa\in\mathcal K_h}
 \beta_\kappa(m^\delta(s))q_\kappa^\delta(s)z_\kappa\,ds.
 \label{fixed-h-regularized-skeleton-integral}
\end{equation}
The nominal rates are Lipschitz on the unit-mass simplex, and
\eqref{uniform-beta-epsilon-limit} gives, uniformly on that simplex,
\begin{equation*}
 |\beta_\kappa^\varepsilon(x)-\beta_\kappa(y)|
 \leq C_h\bigl(\varepsilon+|x-y|\bigr),
 \qquad \kappa\in\mathcal K_h.
\end{equation*}
Subtracting \eqref{fixed-h-regularized-skeleton-integral} from
\eqref{fixed-h-recovery-controlled-equation} therefore yields
\begin{align*}
 \sup_{r\leq t}|m^{\varepsilon,h,\delta}(r)-m^\delta(r)|
 &\leq |m_0^{\varepsilon,h}-m_0^h|
 +\sup_{r\leq t}|R^{\varepsilon,\delta}(r)|\\
 &\quad+C_h\int_0^t
 \bigl(\varepsilon+
 \sup_{u\leq s}|m^{\varepsilon,h,\delta}(u)-m^\delta(u)|\bigr)
 \sum_{\kappa\in\mathcal K_h}q_\kappa^\delta(s)\,ds.
\end{align*}
Since $\sum_{\kappa\in\mathcal K_h}q_\kappa^\delta\in L^1(0,T)$,
the integral form of Gronwall's inequality and
\eqref{fixed-h-recovery-martingale-bound} imply
\begin{align}
 &\mathbb E\sup_{t\leq T}
 |m^{\varepsilon,h,\delta}(t)-m^\delta(t)|\notag\\
 &\quad\leq
 \left(|m_0^{\varepsilon,h}-m_0^h|
 +\bigl(\mathbb E\sup_{t\leq T}
 |R^{\varepsilon,\delta}(t)|^2\bigr)^{1/2}
 +C_h\varepsilon\int_0^T\sum_{\kappa\in\mathcal K_h}
 q_\kappa^\delta(s)\,ds\right)\notag\\
 &\qquad\times
 \exp\left\{C_h\int_0^T\sum_{\kappa\in\mathcal K_h}
 q_\kappa^\delta(s)\,ds\right\}
 \longrightarrow0.
 \label{fixed-h-recovery-path-convergence}
\end{align}
Finally, \eqref{fixed-h-recovery-mark-cost}, the Lipschitz estimate for the
nominal rates, and \eqref{fixed-h-recovery-path-convergence} give
\begin{align*}
 &\left|\mathbb E\int_0^T\sum_{\kappa\in\mathcal K_h}
 \int_0^\infty\ell(\varphi_\kappa^{\varepsilon,\delta}(t,u))\,du\,dt
 -\int_0^T\sum_{\kappa\in\mathcal K_h}
 \beta_\kappa(m^\delta(t))\ell(q_\kappa^\delta(t))\,dt\right|\\
 &\quad\leq C_h\left(\varepsilon+
 \mathbb E\sup_{t\leq T}
 |m^{\varepsilon,h,\delta}(t)-m^\delta(t)|\right)
 \int_0^T\sum_{\kappa\in\mathcal K_h}
 \ell(q_\kappa^\delta(t))\,dt
 \longrightarrow0.
\end{align*}
Thus, for every fixed $\delta>0$, the controlled paths converge to
$m^\delta$ and their entropy costs converge to the action of
$(m^\delta,q^\delta)$.
Let $\delta\downarrow0$ and use
\eqref{fixed-h-buffer-action-limit}.  Since $m^\delta\to m$ uniformly,
a diagonal choice $\delta=\delta(\varepsilon)\downarrow0$, with
$\varphi^\varepsilon:=
\varphi^{\varepsilon,\delta(\varepsilon)}$, proves the lemma.
\end{proof}

The Poisson-random-measure realization and the two controlled-limit lemmas
now give the main large-deviation result.

\begin{proof}[Proof of Theorem~\ref{thm:fixed-h-rigorous-ldp}]
Fix $F\in C_b(D([0,T];\mathbb R_{\geq0}^{\mathcal I_h}))$.
By Lemma~\ref{lem:fixed-h-equivalent-prm-realization} and
\eqref{fixed-h-original-controlled-solution-maps}, the law of
$m^{\varepsilon,h}$ is the image of
$N_h^{\varepsilon^{-1}}$ under the graphical solution map
$\mathcal G_{\varepsilon,h}$.  We may therefore apply
\cite[Theorem~2.1 and Remark~2.2, p.~728]{BudhirajaDupuisMaroulas}
with mark space $\mathcal K_h\times[0,\infty)$, intensity measure $\nu_h$,
intensity parameter $\varepsilon^{-1}$, and functional
$\varepsilon^{-1}F\circ\mathcal G_{\varepsilon,h}$.  It gives the exact
variational representation
\begin{align}
 -\varepsilon\log\mathbb E
 e^{-F(m^{\varepsilon,h})/\varepsilon}
 =\inf_{\varphi\in\mathcal A_h}\mathbb E\Bigg[
 F(m^{\varepsilon,h,\varphi})
 +\int_0^T\sum_{\kappa\in\mathcal K_h}\int_0^\infty
 \ell(\varphi_\kappa(t,u))\,du\,dt\Bigg].
 \label{fixed-h-exact-variational-representation}
\end{align}

We first prove the lower Laplace bound.  Choose a sequence
$\varepsilon_n\downarrow0$
such that
\begin{align*}
 \lim_{n\to\infty}
 \left[-\varepsilon_n\log\mathbb E
 \exp\left\{-\frac{F(m^{\varepsilon_n,h})}{\varepsilon_n}\right\}\right]=
 \liminf_{\varepsilon\downarrow0}
 \left[-\varepsilon\log\mathbb E
 \exp\left\{-\frac{F(m^{\varepsilon,h})}{\varepsilon}\right\}\right].
\end{align*}
Choose controls
$\varphi^n\in\mathcal A_h$ whose values in
\eqref{fixed-h-exact-variational-representation} are within
$\varepsilon_n$ of the infimum; explicitly,
\begin{align*}
 &-\varepsilon_n\log\mathbb E
 \exp\left\{-\frac{F(m^{\varepsilon_n,h})}{\varepsilon_n}\right\}
 +\varepsilon_n
 \geq\mathbb E\Bigg[F(m^{\varepsilon_n,h,\varphi^n})
 +\int_0^T\sum_{\kappa\in\mathcal K_h}\int_0^\infty
 \ell(\varphi_\kappa^n(t,u))\,du\,dt\Bigg].
\end{align*}
For the uncontrolled choice $\varphi_\kappa\equiv1$ for every
$\kappa\in\mathcal K_h$, the entropy cost vanishes because $\ell(1)=0$,
and $m^{\varepsilon_n,h,\varphi}$ has the same law as
$m^{\varepsilon_n,h}$.  Taking this choice in
\eqref{fixed-h-exact-variational-representation} gives
\begin{align*}
 -\varepsilon_n\log\mathbb E
 \exp\left\{-\frac{F(m^{\varepsilon_n,h})}{\varepsilon_n}\right\}
 \leq \mathbb E F(m^{\varepsilon_n,h})
 \leq \|F\|_\infty.
\end{align*}
Combining this bound with the fact that $\mathbb E F(m^{\varepsilon_n,h,\varphi^n})\geq-\|F\|_\infty$ gives
\begin{align*}
 \mathbb E\int_0^T\sum_{\kappa\in\mathcal K_h}\int_0^\infty
 \ell(\varphi_\kappa^n(t,u))\,du\,dt\leq
 2\|F\|_\infty+\varepsilon_n.
\end{align*}
Thus the controls $(\varphi^n)_{n\geq1}$ have uniformly bounded expected
entropy cost.  By
Lemma~\ref{lem:fixed-h-controlled-compactness-liminf}, after passing to a
subsequence, the controlled paths converge in distribution to a path $m$
satisfying
 \eqref{intro-fixed-h-rigorous-continuity} for some multiplier $q$.
Using bounded continuity of $F$ and
\eqref{fixed-h-controlled-condition-cost-liminf}, we obtain
\begin{align*}
 \liminf_{\varepsilon\downarrow0}
 \left[-\varepsilon\log\mathbb E
 \exp\left\{-\frac{F(m^{\varepsilon,h})}{\varepsilon}\right\}\right]&\geq
 \mathbb E\left[F(m)+
 \int_0^T\sum_{\kappa\in\mathcal K_h}
 \beta_\kappa(m(t))\ell(q_\kappa(t))\,dt\right]\\
 &\geq \mathbb E\bigl[F(m)+I_h^{m_0^h}(m)\bigr]\\
 &\geq \inf_{g\in D([0,T];\mathbb R_{\geq0}^{\mathcal I_h})}
 \bigl\{F(g)+I_h^{m_0^h}(g)\bigr\}.
\end{align*}

For the reverse bound, fix a path $m$ with
$I_h^{m_0^h}(m)<\infty$ and $\eta>0$.  Let
$(\varphi^\varepsilon)_\varepsilon$ be the controls supplied by
Lemma~\ref{lem:fixed-h-controlled-recovery}.  Substituting
$\varphi^\varepsilon$ into
\eqref{fixed-h-exact-variational-representation} gives, for every
$\varepsilon>0$,
\begin{align*}
 -\varepsilon\log\mathbb E
 \exp\left\{-\frac{F(m^{\varepsilon,h})}{\varepsilon}\right\}\leq
 \mathbb E F(m^{\varepsilon,h,\varphi^\varepsilon})
 +\mathbb E\int_0^T\sum_{\kappa\in\mathcal K_h}\int_0^\infty
 \ell(\varphi_\kappa^\varepsilon(t,u))\,du\,dt.
\end{align*}
The recovery lemma gives
\begin{align*}
 m^{\varepsilon,h,\varphi^\varepsilon}
 &\Longrightarrow m,\\
 \limsup_{\varepsilon\downarrow0}
 \mathbb E\int_0^T\sum_{\kappa\in\mathcal K_h}\int_0^\infty
 \ell(\varphi_\kappa^\varepsilon(t,u))\,du\,dt
 &\leq I_h^{m_0^h}(m)+\eta.
\end{align*}
Since the limiting path $m$ is deterministic and $F$ is bounded and
continuous on the Skorokhod space, the first convergence implies
\begin{equation*}
 \mathbb E F(m^{\varepsilon,h,\varphi^\varepsilon})
 \longrightarrow F(m).
\end{equation*}
Taking the upper limit in the preceding variational bound therefore yields
\begin{align*}
 &\limsup_{\varepsilon\downarrow0}
 \left[-\varepsilon\log\mathbb E
 \exp\left\{-\frac{F(m^{\varepsilon,h})}{\varepsilon}\right\}\right]
 \leq F(m)+I_h^{m_0^h}(m)+\eta.
\end{align*}
Letting $\eta\downarrow0$ yields
\begin{align*}
 &\limsup_{\varepsilon\downarrow0}
 \left[-\varepsilon\log\mathbb E
 \exp\left\{-\frac{F(m^{\varepsilon,h})}{\varepsilon}\right\}\right]
 \leq F(m)+I_h^{m_0^h}(m).
\end{align*}
Taking the infimum over $m$ proves the opposite Laplace bound.  Therefore
$m^{\varepsilon,h}$ satisfies the Laplace principle with rate
$I_h^{m_0^h}$.  Lemma~\ref{lem:fixed-h-action-good-rigorous} shows that
this rate is good.  Since
$D([0,T];\mathbb R_{\geq0}^{\mathcal I_h})$ with the Skorokhod $J_1$
topology is Polish, the good-rate Laplace principle is equivalent to the
stated LDP.
\end{proof}

\subsection{Poissonian initial data}

The preceding theorem starts from a deterministic approximation of a fixed
initial mass profile.  We now replace it by independent Poisson
occupations in the cells.  Controlling the initial Poisson variables jointly
with the dynamical Poisson random measure adds the initial relative-entropy
cost to the dynamical action, leading to the following extension.

\begin{theorem}
\label{thm:fixed-h-poisson-initial-ldp}
Let $m_*^h\in(0,\infty)^{\mathcal I_h}$. For each $\varepsilon>0$, set
\begin{equation*}
 m_a^{\varepsilon,h}(0)=\varepsilon Z_a^\varepsilon,
 \qquad
 Z_a^\varepsilon\sim\operatorname{Pois}(m_{*,a}^h/\varepsilon),
 \qquad a\in\mathcal I_h,
\end{equation*}
independently over $a$ and independently of the dynamical noises.  Then the
extended process on $\mathsf S_{\varepsilon,h}^{\rm fin}$ given by
Corollary~\ref{cor:coarse-random-finite-mass} has path laws on
$D([0,T];\mathbb R_{\geq0}^{\mathcal I_h})$ satisfying an LDP with speed
$\varepsilon^{-1}$ and good rate
\begin{equation}\label{fixed-h-poisson-full-rate-rigorous}
 \mathcal I_h^{\rm Poi}(m)
 =\sum_{a\in\mathcal I_h}
 m_{*,a}^h\ell\left(\frac{m_a(0)}{m_{*,a}^h}\right)+I_h(m),
\end{equation}
where $I_h$ is the control action without a prescribed initial value.
\end{theorem}

\begin{proof}
Write
\begin{equation*}
 I_{0,h}(g)=\sum_{a\in\mathcal I_h}
 m_{*,a}^h\ell\left(\frac{g_a}{m_{*,a}^h}\right).
\end{equation*}
For each $a\in\mathcal I_h$, let $N_{0,a}^{\varepsilon^{-1}}$ be a Poisson
random measure on $[0,m_{*,a}^h]$ with intensity
$\varepsilon^{-1}du$.  Take these random measures to be mutually
independent and independent of the dynamical Poisson random measures, and
define
\begin{equation}\label{poisson-initial-prm-realization}
 \widehat m_a^{\varepsilon,h}(0)
 :=\varepsilon
 N_{0,a}^{\varepsilon^{-1}}([0,m_{*,a}^h]),
 \qquad a\in\mathcal I_h.
\end{equation}
Since
$N_{0,a}^{\varepsilon^{-1}}([0,m_{*,a}^h])$ is Poisson distributed with
mean $m_{*,a}^h/\varepsilon$, independence gives
\begin{equation*}
 \bigl(\widehat m_a^{\varepsilon,h}(0)\bigr)_{a\in\mathcal I_h}
 \stackrel{\mathrm{law}}=
 \bigl(m_a^{\varepsilon,h}(0)\bigr)_{a\in\mathcal I_h}.
\end{equation*}
We henceforth use the realization
\eqref{poisson-initial-prm-realization}.

Set $U_h:=\max_{a\in\mathcal I_h}m_{*,a}^h$.  For
$u\in[0,U_h]$, let $\mathcal F_{0,u}^\varepsilon$ be the usual augmentation
of the sigma-field generated by
\begin{equation*}
 \left\{
 N_{0,a}^{\varepsilon^{-1}}
 ([0,r]\cap[0,m_{*,a}^h]):
 a\in\mathcal I_h,\ 0\leq r\leq u
 \right\}.
\end{equation*}
Thus $(\mathcal F_{0,u}^\varepsilon)_{0\leq u\leq U_h}$ is the filtration
obtained by revealing all initial Poisson counts progressively in the
auxiliary parameter $u$.  Let $\mathcal P_0^\varepsilon$ be the predictable
sigma-field on $\Omega\times[0,U_h]$ associated with this filtration. 
An initial control is a vector
$\varphi_0=(\varphi_{0,a})_{a\in\mathcal I_h}$ such that, for every
$a\in\mathcal I_h$, the map
$\varphi_{0,a}:\Omega\times[0,m_{*,a}^h]\to\mathbb R_{\geq0}$ is measurable
with respect to the restriction of $\mathcal P_0^\varepsilon$ to this
domain.  The admissible class is
\begin{equation*}
 \mathcal A_{0,h}^\varepsilon
 :=\left\{\varphi_0\text{ as above}:
 \mathbb E\sum_{a\in\mathcal I_h}
 \int_0^{m_{*,a}^h}\ell(\varphi_{0,a}(u))\,du<\infty\right\}.
\end{equation*}
For
$\varphi_0^\varepsilon
=(\varphi_{0,a}^\varepsilon)_{a\in\mathcal I_h}
\in\mathcal A_{0,h}^\varepsilon$, let
$N_{0,a}^{\varepsilon^{-1}\varphi_{0,a}^\varepsilon}$ be the controlled
counting measure given by the auxiliary-variable construction.  Its
predictable compensator is
$\varepsilon^{-1}\varphi_{0,a}^\varepsilon(u)\,du$.  Define
\begin{align}
 m_a^{\varepsilon,h,\varphi_0^\varepsilon}(0)
 &:=\varepsilon
 N_{0,a}^{\varepsilon^{-1}\varphi_{0,a}^\varepsilon}
 ([0,m_{*,a}^h]),
 \label{controlled-poisson-initial-mass}\\
 \rho_a^\varepsilon
 &:=\int_0^{m_{*,a}^h}\varphi_{0,a}^\varepsilon(u)\,du.
 \label{effective-initial-poisson-intensity}
\end{align}
On this controlled Poisson realization, compensation gives the pathwise
identity
\begin{equation}\label{controlled-initial-poisson-decomposition}
 m_a^{\varepsilon,h,\varphi_0^\varepsilon}(0)
 =\rho_a^\varepsilon+R_{0,a}^\varepsilon,
 \qquad
 R_{0,a}^\varepsilon
 =\varepsilon\int_0^{m_{*,a}^h}
 \left(
 N_{0,a}^{\varepsilon^{-1}\varphi_{0,a}^\varepsilon}(du)
 -\varepsilon^{-1}\varphi_{0,a}^\varepsilon(u)\,du
 \right),
\end{equation}
and the compensated-Poisson isometry yields
\begin{equation*}
 \mathbb E|R_{0,a}^\varepsilon|^2
 =\varepsilon\,\mathbb E\rho_a^\varepsilon.
\end{equation*}
Jensen's inequality on $[0,m_{*,a}^h]$ gives
\begin{equation}\label{initial-poisson-jensen}
 \int_0^{m_{*,a}^h}\ell(\varphi_{0,a}^\varepsilon(u))\,du
 \geq
 m_{*,a}^h\ell\left(\frac{\rho_a^\varepsilon}{m_{*,a}^h}\right).
\end{equation}
For every $s>0$, Fenchel's inequality also gives
\begin{equation*}
 s\,\mathbb E\sum_{a\in\mathcal I_h}\rho_a^\varepsilon
 \leq
 \mathbb E\sum_{a\in\mathcal I_h}\int_0^{m_{*,a}^h}
 \ell(\varphi_{0,a}^\varepsilon(u))\,du
 +(e^s-1)\sum_{a\in\mathcal I_h}m_{*,a}^h.
\end{equation*}
Suppose that the expected initial-control costs are uniformly bounded, and
write
\begin{equation}\label{uniform-initial-poisson-control-cost}
 C_0:=\sup_{0<\varepsilon\leq1}
 \mathbb E\sum_{a\in\mathcal I_h}\int_0^{m_{*,a}^h}
 \ell(\varphi_{0,a}^\varepsilon(u))\,du<\infty.
\end{equation}
Taking $s=1$ in the preceding estimate yields
\begin{equation}\label{uniform-initial-effective-mass-bound}
 \sup_{0<\varepsilon\leq1}
 \mathbb E\sum_{a\in\mathcal I_h}\rho_a^\varepsilon
 \leq C_0+(e-1)\sum_{a\in\mathcal I_h}m_{*,a}^h<\infty.
\end{equation}
Consequently, for every $L>0$, Markov's inequality gives
\begin{align*}
 \sup_{0<\varepsilon\leq1}
 \mathbb P\!\left(\sum_{a\in\mathcal I_h}\rho_a^\varepsilon>L\right)
 &\leq \frac{1}{L}\sup_{0<\varepsilon\leq1}
 \mathbb E\sum_{a\in\mathcal I_h}\rho_a^\varepsilon\\
 &\leq \frac{C_0+(e-1)\sum_{a\in\mathcal I_h}m_{*,a}^h}{L}
 \longrightarrow0
 \qquad\text{as }L\uparrow\infty.
\end{align*}
Since $\mathcal I_h$ is finite and $\rho^\varepsilon$ is nonnegative,
this proves tightness of $(\rho^\varepsilon)_{0<\varepsilon\leq1}$ in
$\mathbb R_{\geq0}^{\mathcal I_h}$.

Set $R_0^\varepsilon=(R_{0,a}^\varepsilon)_{a\in\mathcal I_h}$.  The
compensated-Poisson isometry and
\eqref{uniform-initial-effective-mass-bound} imply
\begin{align*}
 \mathbb E|R_0^\varepsilon|^2
 &=\sum_{a\in\mathcal I_h}\mathbb E|R_{0,a}^\varepsilon|^2
 =\varepsilon\,\mathbb E\sum_{a\in\mathcal I_h}\rho_a^\varepsilon\\
 &\leq\varepsilon\left(
 C_0+(e-1)\sum_{a\in\mathcal I_h}m_{*,a}^h\right)
 \longrightarrow0.
\end{align*}
Thus $R_0^\varepsilon\to0$ in $L^2$ and hence in probability.  In view of
\eqref{controlled-initial-poisson-decomposition},
\begin{equation*}
 m^{\varepsilon,h,\varphi_0^\varepsilon}(0)
 =\rho^\varepsilon+R_0^\varepsilon,
\end{equation*}
so Slutsky's theorem shows that the two families
$(m^{\varepsilon,h,\varphi_0^\varepsilon}(0))_\varepsilon$ and
$(\rho^\varepsilon)_\varepsilon$ have the same subsequential weak limits.

We now prove the entropy lower bound.  Choose a sequence
$\varepsilon_n\downarrow0$ such that
\begin{align*}
 &\lim_{n\to\infty}
 \mathbb E\sum_{a\in\mathcal I_h}\int_0^{m_{*,a}^h}
 \ell(\varphi_{0,a}^{\varepsilon_n}(u))\,du\\
 &\qquad=
 \liminf_{\varepsilon\downarrow0}
 \mathbb E\sum_{a\in\mathcal I_h}\int_0^{m_{*,a}^h}
 \ell(\varphi_{0,a}^{\varepsilon}(u))\,du.
\end{align*}
By tightness, after passing to a further subsequence,
there exists an $\mathbb R_{\geq0}^{\mathcal I_h}$-valued random vector
$g=(g_a)_{a\in\mathcal I_h}$ such that
\begin{equation*}
 \rho^{\varepsilon_n}\Longrightarrow g,
 \qquad
 m^{\varepsilon_n,h,\varphi_0^{\varepsilon_n}}(0)
 \Longrightarrow g.
\end{equation*}
On a Skorokhod representation we may suppose that
$\rho^{\varepsilon_n}\to g$ almost surely.  Since
$r\mapsto m_{*,a}^h\ell(r/m_{*,a}^h)$ is nonnegative and lower
semicontinuous on $[0,\infty)$, Fatou's lemma and
\eqref{initial-poisson-jensen} yield
\begin{align}
 \mathbb E\sum_{a\in\mathcal I_h}m_{*,a}^h
 \ell\left(\frac{g_a}{m_{*,a}^h}\right)
 &\leq\liminf_{n\to\infty}
 \mathbb E\sum_{a\in\mathcal I_h}m_{*,a}^h
 \ell\left(\frac{\rho_a^{\varepsilon_n}}{m_{*,a}^h}\right)\notag\\
 &\leq\liminf_{n\to\infty}
 \mathbb E\sum_{a\in\mathcal I_h}\int_0^{m_{*,a}^h}
 \ell(\varphi_{0,a}^{\varepsilon_n}(u))\,du\notag\\
 &=\liminf_{\varepsilon\downarrow0}
 \mathbb E\sum_{a\in\mathcal I_h}\int_0^{m_{*,a}^h}
 \ell(\varphi_{0,a}^\varepsilon(u))\,du.
 \label{initial-poisson-entropy-liminf}
\end{align}

We next combine the initial and dynamical controls.  Let
$\varepsilon_n\downarrow0$, let
$\varphi_0^n\in\mathcal A_{0,h}^{\varepsilon_n}$ and
$\varphi^n\in\mathcal A_h$, and assume that their joint expected entropy
cost is uniformly bounded:
\begin{align}
 \sup_{n\geq1}\mathbb E\Bigg[&
 \sum_{a\in\mathcal I_h}\int_0^{m_{*,a}^h}
 \ell(\varphi_{0,a}^n(u))\,du
 +\int_0^T\sum_{\kappa\in\mathcal K_h}\int_0^\infty
 \ell(\varphi_\kappa^n(t,u))\,du\,dt\Bigg]<\infty.
 \label{joint-initial-dynamical-control-bound}
\end{align}
Let $m^n$ denote the controlled graphical path whose initial value is
$m^{\varepsilon_n,h,\varphi_0^n}(0)$ from
\eqref{controlled-initial-poisson-decomposition} and whose dynamical mark
control is $\varphi^n$.  Since both terms in
\eqref{joint-initial-dynamical-control-bound} are nonnegative, the first
term satisfies \eqref{uniform-initial-poisson-control-cost}, while the
second satisfies the hypothesis
\eqref{fixed-h-uniform-controlled-entropy-bound} of
Lemma~\ref{lem:fixed-h-controlled-compactness-liminf}.

The compensated integral $R_{0,a}^{\varepsilon_n}$ has mean zero, and
hence
\begin{align*}
 \mathbb E\sum_{a\in\mathcal I_h}m_a^n(0)
 &=\mathbb E\sum_{a\in\mathcal I_h}\rho_a^{\varepsilon_n}
 \leq C_h,
\end{align*}
where $C_h<\infty$ follows from
\eqref{uniform-initial-effective-mass-bound}.  Therefore
\begin{align}
 \sup_{n\geq1}\mathbb P\!\left(
 \sum_{a\in\mathcal I_h}m_a^n(0)>L\right)
 \leq\frac{C_h}{L}\longrightarrow0
 \qquad\text{as }L\uparrow\infty.
 \label{controlled-random-initial-mass-tightness}
\end{align}
For fixed $L>0$, set
\begin{equation*}
 \mathsf K_{h,L}:=\left\{x\in\mathbb R_{\geq0}^{\mathcal I_h}:
 \sum_{a\in\mathcal I_h}x_a\leq L\right\}.
\end{equation*}
On the event $\{m^n(0)\in\mathsf K_{h,L}\}$, pathwise conservation of
total mass gives
\begin{equation*}
 m^n(t)\in\mathsf K_{h,L},
 \qquad 0\leq t\leq T.
\end{equation*}
The set $\mathsf K_{h,L}$ is compact, and the finite family of nominal
rates satisfies
\begin{equation*}
 C_{h,L}:=\sup_{0<\varepsilon\leq1}
 \sup_{x\in\mathsf K_{h,L}}
 \sum_{\kappa\in\mathcal K_h}\beta_\kappa^\varepsilon(x)<\infty.
\end{equation*}
On this event, the proof of
Lemma~\ref{lem:fixed-h-controlled-compactness-liminf} applies with
$\mathsf K_{h,L}$ in place of the unit-mass simplex.  Combining this
localized conclusion with
\eqref{controlled-random-initial-mass-tightness} and then letting
$L\uparrow\infty$ proves tightness of
\begin{equation*}
 \left(
 m^n,
 m^n(0),
 \bigl(J_\kappa^n(t)\,dt\bigr)_{\kappa\in\mathcal K_h}
 \right)_{n\geq1},
 \qquad
 J_\kappa^n(t):=
 \int_0^{\beta_\kappa^{\varepsilon_n}(m^n(t))}
 \varphi_\kappa^n(t,u)\,du.
\end{equation*}

After passing to a subsequence, suppose that
\begin{equation*}
 \left(
 m^n,m^n(0),
 \bigl(J_\kappa^n(t)\,dt\bigr)_{\kappa\in\mathcal K_h}
 \right)
 \Longrightarrow
 \left(
 m,g,
 \bigl(J_\kappa(t)\,dt\bigr)_{\kappa\in\mathcal K_h}
 \right).
\end{equation*}
The initial-value component gives $m(0)=g$ almost surely.  The
identification argument in
Lemma~\ref{lem:fixed-h-controlled-compactness-liminf} gives
\begin{equation*}
 m(t)=g+\sum_{\kappa\in\mathcal K_h}
 \int_0^tJ_\kappa(s)z_\kappa\,ds,
 \qquad 0\leq t\leq T,
\end{equation*}
and, with $q_\kappa$ defined by
\eqref{fixed-h-limit-control-defined}, its dynamical entropy bound is
\begin{align*}
 &\mathbb E\int_0^T\sum_{\kappa\in\mathcal K_h}
 \beta_\kappa(m(t))\ell(q_\kappa(t))\,dt\\
 &\qquad\leq\liminf_{n\to\infty}
 \mathbb E\int_0^T\sum_{\kappa\in\mathcal K_h}\int_0^\infty
 \ell(\varphi_\kappa^n(t,u))\,du\,dt.
\end{align*}
Applying \eqref{initial-poisson-entropy-liminf} to the same subsequence and
using the superadditivity of $\liminf$ for nonnegative sequences, we obtain
the joint lower bound
\begin{align}
 &\mathbb E\Bigg[
 \sum_{a\in\mathcal I_h}m_{*,a}^h
 \ell\left(\frac{m_a(0)}{m_{*,a}^h}\right)
 +\int_0^T\sum_{\kappa\in\mathcal K_h}
 \beta_\kappa(m(t))\ell(q_\kappa(t))\,dt\Bigg]\notag\\
 &\quad\leq\liminf_{n\to\infty}\mathbb E\Bigg[
 \sum_{a\in\mathcal I_h}\int_0^{m_{*,a}^h}
 \ell(\varphi_{0,a}^n(u))\,du
 +\int_0^T\sum_{\kappa\in\mathcal K_h}\int_0^\infty
 \ell(\varphi_\kappa^n(t,u))\,du\,dt\Bigg].
 \label{joint-initial-dynamical-entropy-liminf}
\end{align}
Thus the path $m$ is admissible for the rate
$I_{0,h}(m(0))+I_h(m)$, and
\eqref{joint-initial-dynamical-entropy-liminf} gives the compactness,
identification, and cost lower bound for the joint initial--dynamical
variational problem.

At the deterministic rate-function level, on a sublevel of
$I_{0,h}(m(0))+I_h(m)$ the vectors $m(0)$ range over a compact set and
$\sup_t\sum_{a\in\mathcal I_h}m_a(t)
=\sum_{a\in\mathcal I_h}m_a(0)$ is uniformly bounded. The nominal rates
are therefore uniformly bounded, the effective fluxes are uniformly integrable by
\eqref{fixed-h-flux-fenchel}, and the Arzel\`a--Ascoli and entropy-duality
argument of Lemma~\ref{lem:fixed-h-action-good-rigorous} proves compactness
and lower semicontinuity. Hence the displayed full rate is good.

To verify the recovery condition, if $g=m(0)$, put
$g^\delta=(1-\delta)g+\delta m_*^h$ and use the same buffer for the whole
path and the associated effective dynamical fluxes, using
Lemma~\ref{lem:fixed-h-vacuum-regularization} with $\bar m=m_*^h$.
Joint convexity of relative entropy gives
\begin{equation*}
 I_{0,h}(g^\delta)\leq(1-\delta)I_{0,h}(g),
\end{equation*}
while lower semicontinuity gives the opposite limiting inequality. The
dynamical costs converge from above by
Lemma~\ref{lem:fixed-h-vacuum-regularization}. In the initial Poisson
variational representation, the constant control
$g_a^\delta/m_{*,a}^h$ changes the mean in cell $a$ from
$m_{*,a}^h/\varepsilon$ to $g_a^\delta/\varepsilon$ and has cost
$m_{*,a}^h\ell(g_a^\delta/m_{*,a}^h)$.  Denote the resulting controlled
initial mass by $m^{\varepsilon,h,\delta}(0)$.  Then
\begin{equation*}
 \mathbb E\left|m^{\varepsilon,h,\delta}(0)-g^\delta\right|^2
 =\varepsilon\sum_{a\in\mathcal I_h}g_a^\delta\longrightarrow0.
\end{equation*}
After stopping on a fixed total-mass sublevel, the Gronwall and martingale
estimates in the deterministic recovery apply with this random initial
condition. The controlled initial total masses have moments bounded
uniformly in $\varepsilon$, so the stopping level can then be sent to
infinity.

For $F\in C_b(D([0,T];\mathbb R_{\geq0}^{\mathcal I_h}))$, applying the
Poisson variational representation simultaneously to the independent
initial Poisson random measures and to the dynamical Poisson random measure
gives
\begin{align}
 -\varepsilon\log\mathbb E
 \exp\left\{-\frac{F(m^{\varepsilon,h})}{\varepsilon}\right\}&=\inf_{\substack{
       \varphi_0\in\mathcal A_{0,h}^{\varepsilon}\\
       \varphi\in\mathcal A_h}}
 \mathbb E\Bigg[
 F\bigl(m^{\varepsilon,h,\varphi_0,\varphi}\bigr)
 +\sum_{a\in\mathcal I_h}\int_0^{m_{*,a}^h}
 \ell(\varphi_{0,a}(u))\,du\notag\\
 &\hspace{48mm}
 +\int_0^T\sum_{\kappa\in\mathcal K_h}\int_0^\infty
 \ell(\varphi_\kappa(t,u))\,du\,dt\Bigg].
 \label{joint-initial-dynamical-variational-representation}
\end{align}
Here $m^{\varepsilon,h,\varphi_0,\varphi}$ denotes the graphical solution
whose initial mass is defined by
\eqref{controlled-poisson-initial-mass} with control $\varphi_0$ and whose
dynamical Poisson random measure is controlled by $\varphi$ as in
\eqref{fixed-h-graphical-solution-map}.

The compactness and liminf estimates above give the lower bound in
\eqref{joint-initial-dynamical-variational-representation}, while the
preceding recovery construction gives the reverse bound.  Repeating the
final variational argument in the proof
of Theorem~\ref{thm:fixed-h-rigorous-ldp} proves the joint Laplace
principle with rate \eqref{fixed-h-poisson-full-rate-rigorous}.  Since this
rate is good and the path space is Polish, the corresponding LDP follows.
\end{proof}

\section{The continuum collision limit of the rate function}
\label{sec:collision-limit-rigorous}

The fixed-mesh dynamical rate function $I_h$ is defined in
\eqref{fixed-h-free-initial-action}, and the full rate function for
independent Poisson initial occupations is defined in
\eqref{fixed-h-poisson-full-rate-rigorous}.  We recall both formulas.  With
$\ell(r)=r\log r-r+1$ and
$\mathcal K_h=\mathcal E_h\sqcup\Gamma_h$,
\begin{align*}
 I_h(m)=\inf_q\Bigg\{&
 \int_0^T\sum_{\kappa\in\mathcal K_h}
 \beta_\kappa(m(t))\ell(q_\kappa(t))\,dt:\
 q:[0,T]\to\mathbb R_{\geq0}^{\mathcal K_h}
 \text{ is measurable},\\
 &m\in AC([0,T];\mathbb R_{\geq0}^{\mathcal I_h}),\qquad
 (\beta_\kappa(m)q_\kappa)_{\kappa\in\mathcal K_h}
 \in L^1(0,T;\mathbb R^{\mathcal K_h}),\\
 &\dot m(t)=\sum_{\kappa\in\mathcal K_h}
 \beta_\kappa(m(t))q_\kappa(t)z_\kappa
 \quad\text{for a.e. }t\in[0,T]\Bigg\},
\end{align*}
and
\begin{equation*}
 \mathcal I_h^{\rm Poi}(m)
 =\sum_{a\in\mathcal I_h}m_{*,a}^h
 \ell\left(\frac{m_a(0)}{m_{*,a}^h}\right)+I_h(m).
\end{equation*}
For a prescribed initial value $m_0^h$, the dynamical rate is the
restriction of $I_h$ to paths with $m(0)=m_0^h$.  Our goal is to identify
the limit of these functionals through their channel-control
representation as the mesh is refined.

We work with a fixed, regularized velocity cutoff of the collision kernel
$B$.  Fix $R>0$ and write
\begin{equation*}
 \mathbb D_R=\mathbb T^d\times B_R,
 \qquad
 \mathcal C_R=\{c\in\mathcal C:
 |v|\vee|v_*|\vee|v'|\vee|v_*'|<R\}.
\end{equation*}
The kernel $\mathcal B_R$ is obtained by restricting $B$ to collision
configurations strictly inside $\mathcal C_R$, applying a nonnegative
bounded Lipschitz regularization, and preserving the particle-exchange and
pre/post-collisional symmetries.  We assume that
\begin{align*}
 \mathcal B_R&\quad\text{satisfies
 \eqref{intro-fixed-cutoff-kernel-assumption} and
 \eqref{collision-kernel-symmetries}, with $B$ replaced by $\mathcal B_R$}.
\end{align*}
Both the discrete collision coefficients and the limiting continuum action
are therefore built from the same fixed kernel $\mathcal B_R$.  The
regularization and the velocity cutoff remain fixed throughout the mesh
limit.

Take $R(h)\equiv R$, put
$B_h=\mathcal B_R$ in \eqref{riemann-collision-quadrature}, and use a
quasi-uniform product mesh with $h_x+h_v\downarrow0$.

We use separate notation for cell masses, cell averages of test functions,
and piecewise-constant reconstruction:
\begin{equation}\label{mass-test-projections}
 \begin{gathered}
  \begin{aligned}
   (\Pi_hf)_a&=\int_{E_a^h}f(z)\,dz
   \quad(a\in\mathcal I_h),
   &\qquad
   (\mathsf P_hp)_a&=\frac1{|E_a^h|}\int_{E_a^h}p(z)\,dz
   \quad(a\in\mathcal I_h),
  \end{aligned}\\[2pt]
  \mathsf R_hm=\sum_{a\in\mathcal I_h}m_a\chi_a^h.
 \end{gathered}
\end{equation}
For a cell-mass vector we use the density norm
\begin{equation}\label{discrete-density-supnorm}
 \|m\|_{\infty,h}=\max_{a\in\mathcal I_h}
 \frac{|m_a|}{|E_a^h|}.
\end{equation}
\begin{definition}[Regular controlled pair]
\label{def:fixed-cutoff-regular-controlled-pair}
A pair $(f,q)$, with
$f:[0,T]\times\mathbb D_R\to\mathbb R_{\geq0}$ and
$q:[0,T]\times\mathcal C_R\to\mathbb R_{\geq0}$, is called a
regular controlled pair if $f$ and $q$ are continuously differentiable in
time, have bounded Lipschitz derivatives in all their variables on their
respective truncated domains, satisfy
$0<c\leq f\leq C$ and $0\leq q\leq C$ for some constants $c,C<\infty$,
and $f$ solves the controlled Boltzmann equation
\eqref{controlled-boltzmann-equation} with control $q$ and collision kernel
$\mathcal B_R$.
\end{definition}

For a density path $g$ with endpoint traces, set
\begin{align}
 \mathcal I_{\mathcal B_R}^{\rm Ham}(g)
 :=\sup_{r\in C^1([0,T];W^{2,\infty}(\mathbb D_R))}\Bigg\{
 &\langle g_T,r_T\rangle-\langle g_0,r_0\rangle
 -\int_0^T\langle g_t,\partial_tr_t+v\cdot\nabla_xr_t\rangle\,dt
 \notag\\
 &-\frac12\int_0^T\int_{\mathcal C_R}
 \mathcal B_Rg_tg_{t,*}(e^{\Delta r_t}-1)\,dc\,dt\Bigg\},
 \label{fixed-cutoff-hamiltonian-action}\\
 I_{0,R}(g_0)
 &=\int_{\mathbb D_R}
 \left[g_0\log\frac{g_0}{f_0}-g_0+f_0\right]dz,
 \label{fixed-cutoff-initial-action}\\
 I_{\mathcal B_R}(g)
 &=I_{0,R}(g_0)+\mathcal I_{\mathcal B_R}^{\rm Ham}(g).
 \label{fixed-cutoff-full-action}
\end{align}
The effective domains are specified by setting
$\mathcal I_{\mathcal B_R}^{\rm Ham}(g)=+\infty$ unless, for almost every
$t\in(0,T)$, the measure $g_t$ has a nonnegative density
$g_t(z)$ on $\mathbb D_R$, the map $(t,z)\mapsto g_t(z)$ belongs to
$L^1((0,T)\times\mathbb D_R)$, the finite endpoint traces
$g_0,g_T\in\mathcal M_+(\mathbb D_R)$ exist, and
\begin{equation*}
 \int_0^T\int_{\mathcal C_R}
 \mathcal B_R(v,v_*,\sigma)g_t(x,v)g_t(x,v_*)\,dc\,dt<\infty.
\end{equation*}
Similarly, set $I_{0,R}(g_0)=+\infty$ unless $g_0$ has a nonnegative density,
still denoted by $g_0$, and
\begin{equation*}
 \int_{\mathbb D_R}
 \left[g_0\log\frac{g_0}{f_0}-g_0+f_0\right]dz<\infty,
\end{equation*}
where $0\log0=0$.  Finally,
$I_{\mathcal B_R}(g)=+\infty$ whenever either of its two terms is infinite.

\begin{definition}[Biased regular path]
\label{def:fixed-cutoff-biased-regular-path}
A density path $g$ is called a biased regular path if there exist a
nonnegative control $q:[0,T]\times\mathcal C_R\to\mathbb R_{\geq0}$ and a
function $p\in C^1([0,T];W^{2,\infty}(\mathbb D_R))$ such that $(g,q)$ is a
regular controlled pair in the sense of
Definition~\ref{def:fixed-cutoff-regular-controlled-pair} and
\begin{equation}\label{fixed-cutoff-biased-control}
 q(t,c)=e^{\Delta p_t(c)},
 \qquad (t,c)\in[0,T]\times\mathcal C_R.
\end{equation}
Such a function $q$ is called an associated regular control of $g$.
\end{definition}

\begin{proposition}
\label{prop:fixed-cutoff-biased-duality}
If $(g,q)$ is a regular controlled pair, then
\begin{equation}\label{fixed-cutoff-fenchel-bound}
 \mathcal I_{\mathcal B_R}^{\rm Ham}(g)
 \leq\frac12\int_0^T\int_{\mathcal C_R}
 \mathcal B_Rgg_*\ell(q)\,dc\,dt.
\end{equation}
If, in addition, the control $q$ has the form
\eqref{fixed-cutoff-biased-control}, equality holds in
\eqref{fixed-cutoff-fenchel-bound}.
\end{proposition}

\begin{proof}
Fix $r\in C^1([0,T];W^{2,\infty}(\mathbb D_R))$.  We extend the weak
controlled equation to this time-dependent test function.  If
$r(t,z)=\sum_{k=1}^N\eta_k(t)\varphi_k(z)$, apply the controlled Boltzmann
equation \eqref{controlled-boltzmann-equation} to each fixed spatial test
$\varphi_k$, multiply the resulting identity by $\eta_k(t)$, sum over $k$,
and use the product rule for
$\eta_k(t)\langle g_t,\varphi_k\rangle$.  This proves the identity below for
finite sums.  For a general $r$, first extend it from
$\mathbb D_R$, multiply the extension by a compactly supported velocity
cutoff, and mollify it in time and space.  A subsequent tensor-product
approximation gives finite sums for which $r$, $\partial_t r$, and
$\nabla_xr$ converge uniformly on $[0,T]\times\mathbb D_R$.
Since $g$ and $q$ are bounded and $\mathcal B_R$ has compact support in
$\mathcal C_R$, dominated convergence yields
\begin{align*}
 &\langle g_T,r_T\rangle-\langle g_0,r_0\rangle
 -\int_0^T\langle g_t,\partial_tr_t+v\cdot\nabla_xr_t\rangle\,dt\\
 &\qquad=\frac12\int_0^T\int_{\mathcal C_R}
 \mathcal B_Rgg_*q\,\Delta r\,dc\,dt.
\end{align*}
Combining this identity with the expression in
\eqref{fixed-cutoff-hamiltonian-action} gives
\begin{align*}
 &\langle g_T,r_T\rangle-\langle g_0,r_0\rangle
 -\int_0^T\langle g_t,\partial_tr_t+v\cdot\nabla_xr_t\rangle\,dt\\
 &\quad-\frac12\int_0^T\int_{\mathcal C_R}
 \mathcal B_Rgg_*(e^{\Delta r}-1)\,dc\,dt
 =\frac12\int_0^T\int_{\mathcal C_R}
 \mathcal B_Rgg_*\bigl[q\Delta r-(e^{\Delta r}-1)\bigr]\,dc\,dt.
\end{align*}
Differentiation in $s$ gives
$qs-(e^s-1)\leq q\log q-q+1=\ell(q)$ for every $q\geq0$ and
$s\in\mathbb R$.
Applying this inequality pointwise with $s=\Delta r(t,c)$ and taking the
supremum over $r$ proves \eqref{fixed-cutoff-fenchel-bound}.  If
$q=e^{\Delta p}$, equality holds at $s=\Delta p$.  Choosing $r=p$ then
proves equality in \eqref{fixed-cutoff-fenchel-bound}.
\end{proof}

For a cell-mass vector $m$ and a cell test vector $p$, define
\begin{align}
 \mathcal H_h(m,p)
 &=\sum_{(a,b)\in\mathcal E_h}m_aT_{ab}^h
 (e^{p_b-p_a}-1)\notag\\
 &\quad+\frac12\sum_{\gamma=(a,b;c,d)\in\Gamma_h}
 K_\gamma^hm_am_b(e^{p_c+p_d-p_a-p_b}-1).
 \label{discrete-regular-hamiltonian}
\end{align}

\begin{lemma}
\label{lem:regular-hamiltonian-consistency}
Let $T_{ab}^h$ be the conservative upwind transport coefficients introduced
in Section~\ref{sec:coarse-setup-main}.  Assume that the collision
coefficients $K_\gamma^h$, $\gamma\in\Gamma_h$, are defined by the exact
cell-integration formula \eqref{riemann-collision-quadrature}, with
$B_h=\mathcal B_R$.  Then
\begin{equation}\label{regular-continuum-hamiltonian}
 \mathcal H_h(\Pi_hf,\mathsf P_hp)\longrightarrow
 \langle f,v\cdot\nabla_xp\rangle
 +\frac12\int_{\mathcal C_R}\mathcal B_Rff_*(e^{\Delta p}-1)\,dc
 \qquad\text{as }h_x+h_v\downarrow0,
\end{equation}
for every nonnegative $f\in W^{1,\infty}(\mathbb D_R)$ and every
$p\in W^{2,\infty}(\mathbb D_R)$.
\end{lemma}

\begin{proof}
For the transport part, write
$p_{ij}^h=(\mathsf P_hp)_{(i,j)}$.  For each spatial cell $X_i^h$ and
velocity cell $V_j^h$, choose representatives $x_i^h\in X_i^h$ and
$v_j^h\in V_j^h$, and set
\begin{equation*}
 z_{(i,j)}^h=(x_i^h,v_j^h).
\end{equation*}
For $r\in\{1,\ldots,d\}$, also write
\begin{equation*}
 \overline v_{j,r}^{h,+}
 =\frac1{|V_j^h|}\int_{V_j^h}(v_r)^+\,dv,
 \qquad
 \overline v_{j,r}^{h,-}
 =\frac1{|V_j^h|}\int_{V_j^h}(v_r)^-\,dv.
\end{equation*}
The definition of the upwind coefficients then gives, for
$a=(i,j)\in\mathcal I_h$,
\begin{align}
 \sum_{\substack{b\in\mathcal I_h\\(a,b)\in\mathcal E_h}}
 T_{ab}^h(p_b^h-p_a^h)=\sum_{r=1}^d\left\{
 \overline v_{j,r}^{h,+}
 \frac{p_{i+\mathbf e_r^{\,x},j}^h-p_{ij}^h}{h_x}
 +\overline v_{j,r}^{h,-}
 \frac{p_{i-\mathbf e_r^{\,x},j}^h-p_{ij}^h}{h_x}
 \right\}.
 \label{upwind-transport-hamiltonian-identity}
\end{align}
Translation of the spatial cells and the fundamental theorem of calculus
give
\begin{align*}
 \frac{p_{i+\mathbf e_r^{\,x},j}^h-p_{ij}^h}{h_x}
 &=\frac1{|E_{ij}^h|}\int_{E_{ij}^h}\int_0^1
 \partial_{x_r}p(x+\theta h_x\mathbf e_r^{\,x},v)
 \,d\theta\,dx\,dv,\\
 \frac{p_{i-\mathbf e_r^{\,x},j}^h-p_{ij}^h}{h_x}
 &=-\frac1{|E_{ij}^h|}\int_{E_{ij}^h}\int_0^1
 \partial_{x_r}p(x-\theta h_x\mathbf e_r^{\,x},v)
 \,d\theta\,dx\,dv.
\end{align*}
Every point occurring in the first integral is at distance at most
$C_dh_x+h_v$ from $z_{(i,j)}^h$; the same statement holds for the second
integral.  Since $p\in W^{2,\infty}(\mathbb D_R)$, it follows that
\begin{align*}
 \left|
 \frac{p_{i+\mathbf e_r^{\,x},j}^h-p_{ij}^h}{h_x}
 -\partial_{x_r}p(z_{(i,j)}^h)
 \right|
 &\leq C_d\|p\|_{W^{2,\infty}}(h_x+h_v),\\
 \left|
 \frac{p_{i-\mathbf e_r^{\,x},j}^h-p_{ij}^h}{h_x}
 +\partial_{x_r}p(z_{(i,j)}^h)
 \right|
 &\leq C_d\|p\|_{W^{2,\infty}}(h_x+h_v).
\end{align*}
The maps $s\mapsto s^+$ and $s\mapsto s^-$ are $1$-Lipschitz.  Hence,
because $v_j^h\in V_j^h$ and
$\operatorname{diam}(V_j^h)\leq h_v$,
\begin{equation*}
 \left|\overline v_{j,r}^{h,+}-(v_{j,r}^h)^+\right|
 +\left|\overline v_{j,r}^{h,-}-(v_{j,r}^h)^-\right|
 \leq2h_v.
\end{equation*}
Substituting these estimates into
\eqref{upwind-transport-hamiltonian-identity}, using
$(v_{j,r}^h)^+-(v_{j,r}^h)^-=v_{j,r}^h$ and $|v_j^h|\leq R$, and then
summing over $r=1,\ldots,d$ gives
\begin{align}
 \sup_{a=(i,j)\in\mathcal I_h}
 \left|
 \sum_{\substack{b\in\mathcal I_h\\(a,b)\in\mathcal E_h}}
 T_{ab}^h(p_b^h-p_a^h)
 -v_j^h\cdot\nabla_xp(z_a^h)
 \right|
 \leq C_{R,d}\|p\|_{W^{2,\infty}(\mathbb D_R)}(h_x+h_v).
 \label{upwind-first-order-consistency}
\end{align}
Moreover, the two difference quotients above show that
$|p_b^h-p_a^h|\leq h_x\|\nabla_xp\|_\infty$ whenever
$(a,b)\in\mathcal E_h$.  Hence
\begin{align}
 \sup_{a\in\mathcal I_h}
 \sum_{\substack{b\in\mathcal I_h\\(a,b)\in\mathcal E_h}}
 T_{ab}^h|p_b^h-p_a^h|^2
 \leq C_dR h_x\|\nabla_xp\|_\infty^2.
 \label{upwind-second-order-vanishing}
\end{align}

To pass from the linear upwind operator to the transport Hamiltonian,
Taylor's formula gives
\begin{equation*}
 e^{p_b^h-p_a^h}-1=p_b^h-p_a^h+\rho_{ab}^h,
 \qquad\text{where }\rho_{ab}^h\text{ satisfies }
 |\rho_{ab}^h|
 \leq\frac12e^{|p_b^h-p_a^h|}|p_b^h-p_a^h|^2.
\end{equation*}
Therefore, by \eqref{upwind-second-order-vanishing},
\begin{align*}
 \left|\sum_{a\in\mathcal I_h}(\Pi_hf)_a
 \sum_{\substack{b\in\mathcal I_h\\(a,b)\in\mathcal E_h}}
 T_{ab}^h\rho_{ab}^h\right|\leq
 \frac12e^{h_x\|\nabla_xp\|_\infty}
 C_dRh_x\|\nabla_xp\|_\infty^2
 \sum_{a\in\mathcal I_h}(\Pi_hf)_a
 \longrightarrow0.
\end{align*}
On the other hand, \eqref{upwind-first-order-consistency} and the
Riemann-sum approximation on the product mesh yield
\begin{equation*}
 \sum_{a\in\mathcal I_h}(\Pi_hf)_a
 \sum_{\substack{b\in\mathcal I_h\\(a,b)\in\mathcal E_h}}
 T_{ab}^h(p_b^h-p_a^h)
 \longrightarrow
 \int_{\mathbb D_R}f(x,v)v\cdot\nabla_xp(x,v)\,dx\,dv.
\end{equation*}
Combining the last two displays proves convergence of the transport part
of the discrete Hamiltonian.

For the collision part, set
\begin{equation*}
 f_{ij}^h=\frac{(\Pi_hf)_{(i,j)}}{|E_{ij}^h|},
 \qquad
 p_{ij}^h=(\mathsf P_hp)_{(i,j)}.
\end{equation*}
Thus $f_{ij}^h$ and $p_{ij}^h$ are the cell averages of $f$ and $p$ on
$E_{ij}^h$.  For a local collision channel
\begin{equation*}
 \gamma=((i,j),(i,k);(i,l),(i,m)),
\end{equation*}
the definition \eqref{riemann-collision-quadrature} and
$|E_{ij}^h|=|X_i^h||V_j^h|$ give
\begin{align*}
 K_\gamma^h(\Pi_hf)_{(i,j)}(\Pi_hf)_{(i,k)}=|X_i^h|f_{ij}^hf_{ik}^h
 \int_{V_j^h\times V_k^h\times\mathbb S^{d-1}}
 \mathcal B_R(v,v_*,\sigma)
 \mathbf1_{V_l^h}(v')\mathbf1_{V_m^h}(v_*')
 \,dv\,dv_*\,d\sigma.
\end{align*}
Consequently, the collision part of the discrete Hamiltonian is
\begin{align}
 &\frac12\sum_{\gamma=(a,b;c,d)\in\Gamma_h}
 K_\gamma^h(\Pi_hf)_a(\Pi_hf)_b
 \bigl(e^{(\mathsf P_hp)_c+(\mathsf P_hp)_d
              -(\mathsf P_hp)_a-(\mathsf P_hp)_b}-1\bigr)\notag\\
 &\quad=\frac12
 \sum_{i\in\{0,\ldots,N_x-1\}^d}|X_i^h|
 \sum_{j,k,l,m=1}^{N_v(h)}f_{ij}^hf_{ik}^h
 \int_{V_j^h\times V_k^h\times\mathbb S^{d-1}}
 \mathcal B_R(v,v_*,\sigma)
 \mathbf1_{V_l^h}(v')\mathbf1_{V_m^h}(v_*')\notag\\
 &\hspace{47mm}\times
 \bigl(e^{p_{il}^h+p_{im}^h-p_{ij}^h-p_{ik}^h}-1\bigr)
 \,dv\,dv_*\,d\sigma.
 \label{expanded-discrete-collision-hamiltonian}
\end{align}
Terms with $K_\gamma^h=0$ may be included in the sum without changing its
value.  Since $(V_l^h)_l$ is a partition of $B_R$, the sums over $l$ and
$m$ select the unique cells containing $v'$ and $v_*'$, respectively, up
to cell boundaries of measure zero.

To write this observation as an integral, define the piecewise-constant
reconstructions
\begin{equation*}
 f_h=\mathsf R_h\Pi_hf
 =\sum_{(i,j)\in\mathcal I_h}f_{ij}^h\mathbf1_{E_{ij}^h},
 \qquad
 p_h=\sum_{(i,j)\in\mathcal I_h}p_{ij}^h\mathbf1_{E_{ij}^h}.
\end{equation*}
If $x\in X_i^h$, $v\in V_j^h$, $v_*\in V_k^h$,
$v'\in V_l^h$, and $v_*'\in V_m^h$, then
\begin{equation*}
 f_h(x,v)f_h(x,v_*)=f_{ij}^hf_{ik}^h
\end{equation*}
and
\begin{equation*}
 p_h(x,v')+p_h(x,v_*')-p_h(x,v)-p_h(x,v_*)
 =p_{il}^h+p_{im}^h-p_{ij}^h-p_{ik}^h.
\end{equation*}
It follows that, for each fixed $i,j,k,l,m$,
\begin{align*}
 &\int_{X_i^h\times V_j^h\times V_k^h\times\mathbb S^{d-1}}
 \mathcal B_R f_h(x,v)f_h(x,v_*)
 \mathbf1_{V_l^h}(v')\mathbf1_{V_m^h}(v_*')\\
 &\qquad\times\left[
 e^{p_h(x,v')+p_h(x,v_*')-p_h(x,v)-p_h(x,v_*)}-1
 \right]\,dx\,dv\,dv_*\,d\sigma\\
 &=|X_i^h|f_{ij}^hf_{ik}^h
 \int_{V_j^h\times V_k^h\times\mathbb S^{d-1}}
 \mathcal B_R
 \mathbf1_{V_l^h}(v')\mathbf1_{V_m^h}(v_*')
 \bigl(e^{p_{il}^h+p_{im}^h-p_{ij}^h-p_{ik}^h}-1\bigr)
 \,dv\,dv_*\,d\sigma.
\end{align*}
Because $\operatorname{supp}\mathcal B_R$ lies strictly inside the set on
which $v,v_*,v',v_*'\in B_R$, summing this identity over
$i,j,k,l,m$ and using the velocity partition shows that the right-hand side
of \eqref{expanded-discrete-collision-hamiltonian} is exactly
\begin{align*}
 \frac12\int_{\mathcal C_R}
 \mathcal B_R(v,v_*,\sigma)f_h(x,v)f_h(x,v_*)\left[
 e^{p_h(x,v')+p_h(x,v_*')-p_h(x,v)-p_h(x,v_*)}-1
 \right]\,dc.
\end{align*}
The cell-average property and the diameter bounds on the product mesh imply
\begin{equation*}
 \|f_h-f\|_{L^\infty(\mathbb D_R)}
 +\|p_h-p\|_{L^\infty(\mathbb D_R)}
 \leq C_d(h_x+h_v)
 \bigl(\|f\|_{W^{1,\infty}}+\|p\|_{W^{1,\infty}}\bigr).
\end{equation*}
In particular, the integrand above converges pointwise on $\mathcal C_R$ to
\begin{equation*}
 \mathcal B_R(v,v_*,\sigma)f(x,v)f(x,v_*)
 \bigl(e^{\Delta p(x,v,v_*,\sigma)}-1\bigr).
\end{equation*}
Moreover, since cell averaging does not increase the $L^\infty$ norm, its
absolute value is bounded by
\begin{equation*}
 \|\mathcal B_R\|_\infty\|f\|_\infty^2
 \bigl(e^{4\|p\|_\infty}+1\bigr)
 \mathbf1_{\operatorname{supp}\mathcal B_R}(v,v_*,\sigma),
\end{equation*}
which is integrable in $dc$.  Dominated convergence therefore yields
\begin{align*}
 &\frac12\sum_{\gamma=(a,b;c,d)\in\Gamma_h}
 K_\gamma^h(\Pi_hf)_a(\Pi_hf)_b
 \bigl(e^{(\mathsf P_hp)_c+(\mathsf P_hp)_d
              -(\mathsf P_hp)_a-(\mathsf P_hp)_b}-1\bigr)\\
 &\qquad\longrightarrow
 \frac12\int_{\mathcal C_R}
 \mathcal B_Rff_*\bigl(e^{\Delta p}-1\bigr)\,dc.
\end{align*}
Together with the transport convergence proved above, this establishes
\eqref{regular-continuum-hamiltonian}.
\end{proof}

We next construct discrete controls and paths that approximate a regular
controlled pair and recover its collision cost.

\begin{lemma}
\label{lem:regular-controlled-path-cost-consistency}
Let $T_{ab}^h$, $(a,b)\in\mathcal E_h$, be the conservative upwind
transport coefficients, and let $K_\gamma^h$, $\gamma\in\Gamma_h$, be
defined by \eqref{riemann-collision-quadrature} with
$B_h=\mathcal B_R$.  Recall from \eqref{mass-test-projections} that
$\Pi_h$ maps a density to its vector of cell masses and that $\mathsf R_h$
reconstructs a piecewise-constant density from a cell-mass vector.  Let
$(f,q)$ be a regular controlled pair
in the sense of Definition~\ref{def:fixed-cutoff-regular-controlled-pair}.
Then there exist measurable collision controls
\begin{equation*}
 q^h=(q_\gamma^h)_{\gamma\in\Gamma_h}:[0,T]\longrightarrow
 \mathbb R_{\geq0}^{\Gamma_h}
\end{equation*}
and a path
$m^h=(m_a^h)_{a\in\mathcal I_h}\in
AC([0,T];\mathbb R_{\geq0}^{\mathcal I_h})$ satisfying
\begin{align}
 \dot m^h(t)
 &={}
 \sum_{(a,b)\in\mathcal E_h}
 m_a^h(t)T_{ab}^h(e_b-e_a)\notag\\
 &\quad+\frac12
 \sum_{\gamma=(a,b;c,d)\in\Gamma_h}
 K_\gamma^hm_a^h(t)m_b^h(t)q_\gamma^h(t)
 (e_c+e_d-e_a-e_b),
 \quad\text{for a.e. }t\in[0,T],\notag\\
 m^h(0)&=\Pi_hf(0).
 \label{regular-controlled-discrete-skeleton}
\end{align}
Thus the multiplier on each transport channel is one, whereas
$q_\gamma^h$ is the multiplier on the collision channel
$\gamma\in\Gamma_h$.  The controls and paths can be chosen so that, as
$h_x+h_v\downarrow0$,
\begin{align}
 \sup_{t\in[0,T]}
 \left\|\mathsf R_hm^h(t)-f_t\right\|_{L^1(\mathbb D_R)}
 &\longrightarrow0,
 \label{regular-controlled-path-convergence}\\
 \int_0^T\frac12\sum_{\gamma=(a,b;c,d)\in\Gamma_h}
 K_\gamma^hm_a^h(t)m_b^h(t)\ell(q_\gamma^h(t))\,dt
 &\longrightarrow
 \frac12\int_0^T\int_{\mathcal C_R}\mathcal B_Rff_*\ell(q)\,dc\,dt.
 \label{regular-controlled-action-convergence}
\end{align}
In particular, the discrete collision-control cost converges to the
continuum collision-control cost of $(f,q)$.
\end{lemma}

\begin{proof}
For the regular controlled pair $(f,q)$ and each channel
\begin{equation*}
 \gamma=((i,j),(i,k);(i,l),(i,m))\in\Gamma_h,
\end{equation*}
set
\begin{equation*}
 q_\gamma^h(t)=
 \frac{\displaystyle
 \int_{X_i^h\times V_j^h\times V_k^h\times\mathbb S^{d-1}}
 \mathcal B_R(v,v_*,\sigma)q(t,x,v,v_*,\sigma)
 \mathbf1_{V_l^h}(v')\mathbf1_{V_m^h}(v_*')
 \,dx\,dv\,dv_*\,d\sigma}
 {\displaystyle
 |X_i^h|\int_{V_j^h\times V_k^h\times\mathbb S^{d-1}}
 \mathcal B_R(v,v_*,\sigma)
 \mathbf1_{V_l^h}(v')\mathbf1_{V_m^h}(v_*')
 \,dv\,dv_*\,d\sigma}.
\end{equation*}
The denominator is strictly positive because
$\gamma\in\Gamma_h$ means $K_\gamma^h>0$.  Let $m^h$ be the solution of
the discrete controlled equation
\eqref{regular-controlled-discrete-skeleton}.

At a fixed time $t$, suppress the time dependence of $q^h(t)$ and define
the transport and controlled collision parts of the discrete drift by
\begin{align*}
 \mathcal T_h(m)
 &=\sum_{(a,b)\in\mathcal E_h}m_aT_{ab}^h(e_b-e_a),\\
 Q_h^{q^h}(m)
 &=\frac12\sum_{\gamma=(a,b;c,d)\in\Gamma_h}
 K_\gamma^hm_am_bq_\gamma^h(e_c+e_d-e_a-e_b).
\end{align*}
In the proof below, $Q_h^{q^h}$ at time $t$ denotes
$Q_h^{q^h(t)}$.  Define the error
\begin{equation*}
 w^h(t):=m^h(t)-\Pi_hf_t.
\end{equation*}
Subtracting the equation for $\Pi_hf$ from
\eqref{regular-controlled-discrete-skeleton} gives
\begin{align}
 \dot w^h(t)
 =\frac d{dt}\bigl(m^h(t)-\Pi_hf_t\bigr)
 &=\mathcal T_h\bigl(m^h(t)-\Pi_hf_t\bigr)
 +Q_h^{q^h(t)}\bigl(m^h(t)\bigr)
 -Q_h^{q^h(t)}(\Pi_hf_t)\notag\\
 &\quad-\left[
 \frac d{dt}\Pi_hf_t-\mathcal T_h(\Pi_hf_t)
 -Q_h^{q^h(t)}(\Pi_hf_t)
 \right].
 \label{regular-controlled-error-equation}
\end{align}
We estimate the three terms on the right-hand side separately: the
transport term, the difference of the collision terms, and the consistency
residual in square brackets.

\medskip
\noindent\textbf{Step 1: the transport term.}
We begin with the maximum principle for $\mathcal T_h$.  Given
$w\in\mathbb R^{\mathcal I_h}$, put $r_a=w_a/|E_a^h|$.  If
$a=(i,j)$ and $(a,b)\in\mathcal E_h$, then, for some
$r\in\{1,\ldots,d\}$,
\begin{equation*}
 b=(i+\mathbf e_r^{\,x},j)
 \qquad\text{or}\qquad
 b=(i-\mathbf e_r^{\,x},j).
\end{equation*}
Thus transport changes the spatial cell but leaves the velocity cell
$V_j^h$ unchanged.  Since all spatial cells have volume $h_x^d$,
\begin{equation*}
 |E_a^h|=|X_i^h||V_j^h|
 =|X_{i\pm\mathbf e_r^{\,x}}^h||V_j^h|=|E_b^h|.
\end{equation*}
The total outgoing and incoming transport rates can be compared directly.
For
$a=(i,j)$, the upwind definition gives
\begin{align*}
 \sum_{\substack{b\in\mathcal I_h\\(a,b)\in\mathcal E_h}}T_{ab}^h
 &=\frac1{h_x}\sum_{r=1}^d
 \bigl(\overline v_{j,r}^{h,+}+\overline v_{j,r}^{h,-}\bigr).
\end{align*}
The two incoming
edges in that direction come from these same two neighboring spatial cells:
\begin{align*}
 T_{(i-\mathbf e_r^{\,x},j),(i,j)}^h
 &=\frac{\overline v_{j,r}^{h,+}}{h_x},
 &
 T_{(i+\mathbf e_r^{\,x},j),(i,j)}^h
 &=\frac{\overline v_{j,r}^{h,-}}{h_x}.
\end{align*}
The periodic convention ensures that these neighbors also exist when
$X_i^h$ touches the boundary of the fundamental domain.  Summing over
$r=1,\ldots,d$ therefore yields
\begin{equation*}
 \sum_{\substack{b\in\mathcal I_h\\(a,b)\in\mathcal E_h}}T_{ab}^h
 =\sum_{\substack{b\in\mathcal I_h\\(b,a)\in\mathcal E_h}}T_{ba}^h.
\end{equation*}
It follows that, for every $a\in\mathcal I_h$,
\begin{align}
 \frac{[\mathcal T_h(w)]_a}{|E_a^h|}
 &=\frac1{|E_a^h|}
 \sum_{(c,d)\in\mathcal E_h}w_cT_{cd}^h
 \bigl(\mathbf1_{\{a=d\}}-\mathbf1_{\{a=c\}}\bigr)=\frac1{|E_a^h|}
 \left(
 \sum_{\substack{b\in\mathcal I_h\\(b,a)\in\mathcal E_h}}
 w_bT_{ba}^h
 -w_a\sum_{\substack{b\in\mathcal I_h\\(a,b)\in\mathcal E_h}}
 T_{ab}^h
 \right)\notag\\
 &=\sum_{\substack{b\in\mathcal I_h\\(b,a)\in\mathcal E_h}}
 T_{ba}^hr_b
 -r_a\sum_{\substack{b\in\mathcal I_h\\(a,b)\in\mathcal E_h}}
 T_{ab}^h=\sum_{\substack{b\in\mathcal I_h\\(b,a)\in\mathcal E_h}}
 T_{ba}^h(r_b-r_a).
 \label{upwind-discrete-maximum-principle}
\end{align}
Thus the right-hand side is nonpositive at an index where $r_a$ is
maximal, and nonnegative at an index where $r_a$ is minimal.   For a real-valued function $G$,
its upper right Dini derivative is
\begin{equation*}
 D^+G(t):=\limsup_{s\downarrow0}\frac{G(t+s)-G(t)}s.
\end{equation*}
Suppose that $\dot w=\mathcal T_h(w)$, and continue to write
$r_a(t)=w_a(t)/|E_a^h|$.  Recall from
\eqref{discrete-density-supnorm} that
\begin{equation*}
 \|w(t)\|_{\infty,h}
 =\max_{a\in\mathcal I_h}\frac{|w_a(t)|}{|E_a^h|}
 =\max_{a\in\mathcal I_h}|r_a(t)|.
\end{equation*}
Assume first that
$\|w(t)\|_{\infty,h}>0$.  Choose a sequence $s_n\downarrow0$ along which
the difference quotients converge to the upper limit in the definition of
$D^+\|w(t)\|_{\infty,h}$.  For each $n$, choose
$a_n\in\mathcal I_h$ such that
\begin{equation*}
 |r_{a_n}(t+s_n)|
 =\max_{b\in\mathcal I_h}|r_b(t+s_n)|.
\end{equation*}
The sequence $(a_n)_{n\geq1}$ takes values in the finite set
$\mathcal I_h$.  Hence there exist an index $a_*\in\mathcal I_h$ and a
subsequence $(n_k)_{k\geq1}$ such that $a_{n_k}=a_*$ for every $k$.
Replacing $(s_n,a_n)$ by this subsequence and relabeling $k$ as $n$, we
may assume that $a_n=a_*$ for every $n$.  Continuity of the components of
$r$ then gives
\begin{equation*}
 |r_{a_*}(t)|=\max_{b\in\mathcal I_h}|r_b(t)|.
\end{equation*}
Since $r_{a_*}(t)\neq0$, the function $|r_{a_*}|$ is differentiable at
$t$.  Therefore
\begin{align*}
 D^+\|w(t)\|_{\infty,h}
 &=\lim_{n\to\infty}
 \frac{|r_{a_*}(t+s_n)|-\max_{b\in\mathcal I_h}|r_b(t)|}{s_n}\\
 &\leq\lim_{n\to\infty}
 \frac{|r_{a_*}(t+s_n)|-|r_{a_*}(t)|}{s_n}\\
 &=\operatorname{sgn}(r_{a_*}(t))\dot r_{a_*}(t)\\
 &\leq
 \max_{\substack{a\in\mathcal I_h\\
 |r_a(t)|=\max_{b\in\mathcal I_h}|r_b(t)|}}
 \operatorname{sgn}(r_a(t))\dot r_a(t)\\
 &=\max_{\substack{a\in\mathcal I_h\\
 |r_a(t)|=\max_{b\in\mathcal I_h}|r_b(t)|}}
 \operatorname{sgn}(r_a(t))
 \frac{[\mathcal T_h(w(t))]_a}{|E_a^h|}.
\end{align*}
For any index $a$ occurring in these maxima, if $r_a(t)>0$, then $r_a(t)$
is a maximum of the components of $r(t)$, and
\eqref{upwind-discrete-maximum-principle} shows that
$\dot r_a(t)\leq0$.  If $r_a(t)<0$, then it is a minimum and
$\dot r_a(t)\geq0$, so again
$\operatorname{sgn}(r_a(t))\dot r_a(t)\leq0$.  If
$\|w(t)\|_{\infty,h}=0$, then $w(t)=0$ and
$\mathcal T_h(w(t))=0$.  Hence in every case
\begin{equation*}
 D^+\|w(t)\|_{\infty,h}\leq0.
\end{equation*}

We also need the corresponding estimate with a forcing term.  If
$w\in C^1([0,T];\mathbb R^{\mathcal I_h})$ satisfies
\begin{equation*}
 \dot w(t)=\mathcal T_h(w(t))+g(t)
\end{equation*}
for some $g\in C([0,T];\mathbb R^{\mathcal I_h})$, then, whenever
$\|w(t)\|_{\infty,h}>0$,
\begin{align*}
 D^+\|w(t)\|_{\infty,h}
 &\leq
 \max_{\substack{a\in\mathcal I_h\\
 |r_a(t)|=\max_{b\in\mathcal I_h}|r_b(t)|}}
 \operatorname{sgn}(r_a(t))
 \frac{[\mathcal T_h(w(t))]_a+g_a(t)}{|E_a^h|}\\
 &\leq \max_{a\in\mathcal I_h}\frac{|g_a(t)|}{|E_a^h|}
 =\|g(t)\|_{\infty,h}.
\end{align*}
The transport contribution is nonpositive, and the contribution of $g(t)$
is bounded by $\|g(t)\|_{\infty,h}$.  If
$\|w(t)\|_{\infty,h}=0$, then $w(t)=0$ and
$\mathcal T_h(w(t))=0$, and the same inequality follows directly from the
definition of the upper right Dini derivative.  We have therefore proved
\begin{equation*}
 D^+\|w(t)\|_{\infty,h}\leq\|g(t)\|_{\infty,h}.
\end{equation*}

\medskip
\noindent\textbf{Step 2: the collision term.}
We then estimate the collision part.  Its $\alpha$-th coordinate is
\begin{align*}
 [Q_h^{q^h}(m)]_\alpha
 =\frac12\sum_{\gamma=(a,b;c,d)\in\Gamma_h}
 K_\gamma^hm_am_bq_\gamma^h
 \bigl(\mathbf1_{\{\alpha=c\}}+\mathbf1_{\{\alpha=d\}}-\mathbf1_{\{\alpha=a\}}-\mathbf1_{\{\alpha=b\}}\bigr).
\end{align*}
We use the uniform coefficient bound
\begin{align}
 \sup_{\alpha\in\mathcal I_h}\frac1{|E_\alpha^h|}
 \sum_{\gamma=(a,b;c,d)\in\Gamma_h}
 K_\gamma^h|E_a^h||E_b^h|
 \bigl(&\mathbf1_{\{\alpha=a\}}+\mathbf1_{\{\alpha=b\}}
       +\mathbf1_{\{\alpha=c\}}+\mathbf1_{\{\alpha=d\}}\bigr)
 \leq C_R,
 \label{uniform-collision-coefficient-sum}
\end{align}
where $C_R$ is independent of $h$.  Fix
$\alpha=(i,\rho)\in\mathcal I_h$.  Since collisions are local in the
spatial variable, only channels whose four cells have spatial index $i$
contribute.  The contribution from channels for which $\alpha=a$ is
\begin{align*}
 &\frac1{|E_{i\rho}^h|}
 \sum_{k,l,m}K_{i;\rho k\to lm}^h
 |E_{i\rho}^h||E_{ik}^h|\\
 &\quad=\frac1{|V_\rho^h|}
 \sum_k\int_{V_\rho^h\times V_k^h\times\mathbb S^{d-1}}
 \mathcal B_R(v,v_*,\sigma)
 \left(\sum_l\mathbf1_{V_l^h}(v')\right)
 \left(\sum_m\mathbf1_{V_m^h}(v_*')\right)
 \,dv\,dv_*\,d\sigma\\
 &\quad=\frac1{|V_\rho^h|}
 \int_{V_\rho^h\times B_R\times\mathbb S^{d-1}}
 \mathcal B_R(v,v_*,\sigma)
 \left(\sum_l\mathbf1_{V_l^h}(v')\right)
 \left(\sum_m\mathbf1_{V_m^h}(v_*')\right)
 \,dv\,dv_*\,d\sigma\\
 &\quad\leq
 \|\mathcal B_R\|_\infty |B_R|\,|\mathbb S^{d-1}|.
\end{align*}
Here we used $|E_{ij}^h|=|X_i^h||V_j^h|$, summed first over the partition
$(V_k^h)_k$ of $B_R$, and used
$\sum_l\mathbf1_{V_l^h}\leq1$ and
$\sum_m\mathbf1_{V_m^h}\leq1$.  Similarly, the contribution from
$\alpha=b$ satisfies
\begin{align*}
 &\frac1{|E_{i\rho}^h|}
 \sum_{j,l,m}K_{i;j\rho\to lm}^h
 |E_{ij}^h||E_{i\rho}^h|\\
 &\quad=\frac1{|V_\rho^h|}
 \sum_j\int_{V_j^h\times V_\rho^h\times\mathbb S^{d-1}}
 \mathcal B_R(v,v_*,\sigma)
 \left(\sum_l\mathbf1_{V_l^h}(v')\right)
 \left(\sum_m\mathbf1_{V_m^h}(v_*')\right)
 \,dv\,dv_*\,d\sigma\\
 &\quad=\frac1{|V_\rho^h|}
 \int_{B_R\times V_\rho^h\times\mathbb S^{d-1}}
 \mathcal B_R(v,v_*,\sigma)
 \left(\sum_l\mathbf1_{V_l^h}(v')\right)
 \left(\sum_m\mathbf1_{V_m^h}(v_*')\right)
 \,dv\,dv_*\,d\sigma\\
 &\quad\leq
 \|\mathcal B_R\|_\infty |B_R|\,|\mathbb S^{d-1}|.
\end{align*}
For the two outgoing cells, we use
\eqref{volume-weighted-collision-microreversibility}.  If $\alpha=c$, then
\begin{align*}
 \frac1{|E_{i\rho}^h|}
 \sum_{j,k,m}K_{i;jk\to\rho m}^h
 |E_{ij}^h||E_{ik}^h|=\sum_m|E_{im}^h|\sum_{j,k}K_{i;\rho m\to jk}^h
 \leq \|\mathcal B_R\|_\infty |B_R|\,|\mathbb S^{d-1}|,
\end{align*}
where the expression after the equality is of the first incoming type
estimated above.  If $\alpha=d$, the same identity gives
\begin{align*}
 \frac1{|E_{i\rho}^h|}
 \sum_{j,k,l}K_{i;jk\to l\rho}^h
 |E_{ij}^h||E_{ik}^h|=\sum_l|E_{il}^h|\sum_{j,k}K_{i;l\rho\to jk}^h
 \leq \|\mathcal B_R\|_\infty |B_R|\,|\mathbb S^{d-1}|,
\end{align*}
which is of the second incoming type.  Adding these four estimates and
taking the supremum over $\alpha\in\mathcal I_h$ proves
\eqref{uniform-collision-coefficient-sum}, for instance with
$C_R=4\|\mathcal B_R\|_\infty|B_R|\,|\mathbb S^{d-1}|$.

Since $0\leq q_\gamma^h\leq\|q\|_\infty$ and
\begin{equation*}
 |m_am_b-\widetilde m_a\widetilde m_b|
 \leq |m_a|\,|m_b-\widetilde m_b|
      +|\widetilde m_b|\,|m_a-\widetilde m_a|,
\end{equation*}
inequality \eqref{uniform-collision-coefficient-sum} yields, for all
$m,\widetilde m\in\mathbb R_{\geq0}^{\mathcal I_h}$,
\begin{equation}
 \|Q_h^{q^h}(m)-Q_h^{q^h}(\widetilde m)\|_{\infty,h}
 \leq C_R\|q\|_\infty
 (\|m\|_{\infty,h}+\|\widetilde m\|_{\infty,h})
 \|m-\widetilde m\|_{\infty,h}
 \label{discrete-collision-linfty-lipschitz}
\end{equation}
with a constant independent of $h$.

\medskip
\noindent\textbf{Step 3: the consistency residual.}
It remains to estimate the consistency error obtained by inserting the projected
continuum solution $\Pi_hf$ into the discrete controlled equation.  Write
\begin{equation*}
 f_{i\rho}^h(t)=\frac{(\Pi_hf_t)_{(i,\rho)}}{|E_{i\rho}^h|},
 \qquad
 f_h(t)=\mathsf R_h\Pi_hf_t.
\end{equation*}
For the transport part, the definition of the upwind coefficients gives
\begin{align*}
 \frac{[\mathcal T_h(\Pi_hf_t)]_{(i,\rho)}}{|E_{i\rho}^h|}
 =\sum_{r=1}^d\bigg[&
 \overline v_{\rho,r}^{h,+}
 \frac{f_{i-\mathbf e_r^{\,x},\rho}^h(t)-f_{i\rho}^h(t)}{h_x}
 +\overline v_{\rho,r}^{h,-}
 \frac{f_{i+\mathbf e_r^{\,x},\rho}^h(t)-f_{i\rho}^h(t)}{h_x}
 \bigg].
\end{align*}
Taylor's formula in $x$, together with
$\operatorname{diam}(V_\rho^h)\leq h_v$, therefore yields
\begin{align*}
 \sup_{t\in[0,T]}
 \left\|
 \Pi_h(-v\cdot\nabla_xf_t)-\mathcal T_h(\Pi_hf_t)
 \right\|_{\infty,h}
 \leq C_{R,d}(h_x+h_v)
 \sup_{t\in[0,T]}\|f_t\|_{W^{2,\infty}(\mathbb D_R)}.
\end{align*}

For the collision part, since $f$ solves the controlled Boltzmann equation
\eqref{controlled-boltzmann-equation} with kernel $\mathcal B_R$ and
multiplier $q$, it satisfies
\begin{equation*}
 \partial_tf_t=-v\cdot\nabla_xf_t+Q^q(f_t,f_t).
\end{equation*}
Consequently,
\begin{equation*}
 \frac d{dt}\Pi_hf_t
 =\Pi_h(-v\cdot\nabla_xf_t)+\Pi_hQ^q(f_t,f_t).
\end{equation*}
Here the continuum controlled collision operator $Q^q$ is defined by
\begin{align*}
 \int_{\mathbb D_R}Q^q(f_t,f_t)(x,v)\varphi(x,v)\,dx\,dv
 =\frac12\int_{\mathcal C_R}
 \mathcal B_R(v,v_*,\sigma)f_t(x,v)f_t(x,v_*)
 q(t,x,v,v_*,\sigma)\Delta\varphi(c)\,dc
\end{align*}
for every $\varphi\in C^\infty(\mathbb D_R)$, where
\begin{equation*}
 \Delta\varphi(c)
 =\varphi(x,v')+\varphi(x,v_*')
 -\varphi(x,v)-\varphi(x,v_*).
\end{equation*}

The definition of $q_\gamma^h$ and the exact cell-integration formula
imply, for every $(i,\rho)\in\mathcal I_h$,
\begin{align*}
 [Q_h^{q^h}(\Pi_hf_t)]_{(i,\rho)}&=\frac12
 \int_{X_i^h\times B_R\times B_R\times\mathbb S^{d-1}}
 \mathcal B_R(v,v_*,\sigma)q(t,x,v,v_*,\sigma)
 f_h(t,x,v)f_h(t,x,v_*)\\
 &\qquad\qquad\times
 \bigl(
 \mathbf1_{V_\rho^h}(v')+\mathbf1_{V_\rho^h}(v_*')
 -\mathbf1_{V_\rho^h}(v)-\mathbf1_{V_\rho^h}(v_*)
 \bigr)\,dx\,dv\,dv_*\,d\sigma.
\end{align*}
On the other hand, the controlled Boltzmann equation
\eqref{controlled-boltzmann-equation}, applied by approximation to the
indicator of $E_{i\rho}^h$, gives the same integral with
$f_h(t,x,v)f_h(t,x,v_*)$ replaced by $f(t,x,v)f(t,x,v_*)$.  Since
\begin{equation*}
 \sup_{t\in[0,T]}\|f_h(t)-f(t)\|_{L^\infty(\mathbb D_R)}
 \leq C_d(h_x+h_v)
 \sup_{t\in[0,T]}\|f_t\|_{W^{1,\infty}(\mathbb D_R)},
\end{equation*}
partitioning the collision integral into channels and using
\eqref{uniform-collision-coefficient-sum} gives
\begin{align*}
 &\sup_{t\in[0,T]}
 \left\|
 \Pi_hQ^q(f_t,f_t)-Q_h^{q^h}(\Pi_hf_t)
 \right\|_{\infty,h}\notag\\
 &\qquad\leq C_R\|q\|_\infty
 \sup_{t\in[0,T]}
 \bigl(\|f_h(t)\|_\infty+\|f_t\|_\infty\bigr)
 \|f_h(t)-f_t\|_\infty
 \longrightarrow0.
\end{align*}
Define
\begin{equation*}
 \omega_h:=\sup_{t\in[0,T]}
 \left\|\frac d{dt}\Pi_hf_t-
 \mathcal T_h(\Pi_hf_t)-Q_h^{q^h}(\Pi_hf_t)\right\|_{\infty,h}.
\end{equation*}
The two preceding estimates show that $\omega_h\to0$ as
$h_x+h_v\downarrow0$.  In particular, for every
$t\in[0,T]$,
\begin{equation}\label{regular-controlled-residual}
 \left\|\frac d{dt}\Pi_hf-
 \mathcal T_h(\Pi_hf)-Q_h^{q^h}(\Pi_hf)\right\|_{\infty,h}
 \leq\omega_h.
\end{equation}

We combine the three estimates in
\eqref{regular-controlled-error-equation}.
Let
\begin{equation*}
 \tau_h:=\inf\left\{t\in[0,T]:
 \|m^h(t)-\Pi_hf_t\|_{\infty,h}\geq1\right\}\wedge T.
\end{equation*}
For $t\leq\tau_h$, apply the forced maximum-principle inequality above to
\eqref{regular-controlled-error-equation}, with
$w=w^h=m^h-\Pi_hf$.  The local Lipschitz estimate
\eqref{discrete-collision-linfty-lipschitz} and
\eqref{regular-controlled-residual} then yield
\begin{equation*}
 D^+\|m^h(t)-\Pi_hf_t\|_{\infty,h}
 \leq C_{R,f,q}\|m^h(t)-\Pi_hf_t\|_{\infty,h}+\omega_h.
\end{equation*}
Since $m^h(0)=\Pi_hf_0$, Gronwall's lemma gives
\begin{equation*}
 \sup_{t\leq\tau_h}\|m^h(t)-\Pi_hf_t\|_{\infty,h}
 \leq T e^{C_{R,f,q}T}\omega_h.
\end{equation*}
The right-hand side is smaller than one for all sufficiently fine meshes;
hence $\tau_h=T$. Since
$|\mathbb D_R|<\infty$, this proves
\eqref{regular-controlled-path-convergence}, because
\begin{equation*}
 \|\mathsf R_hm^h-f\|_{L^1}
 \leq |\mathbb D_R|\|m^h-\Pi_hf\|_{\infty,h}
 +\|\mathsf R_h\Pi_hf-f\|_{L^1}.
\end{equation*}

We conclude by proving convergence of the action.  The path estimate and
the uniform convergence of the cell averages give
\begin{align*}
 \sup_{t\in[0,T]}\|\mathsf R_hm^h(t)-f_t\|_{L^\infty(\mathbb D_R)}
 &\leq \sup_{t\in[0,T]}\|m^h(t)-\Pi_hf_t\|_{\infty,h}
 +\sup_{t\in[0,T]}
 \|\mathsf R_h\Pi_hf_t-f_t\|_{L^\infty(\mathbb D_R)}
 \longrightarrow0.
\end{align*}

For a channel
$\gamma=((i,j),(i,k);(i,l),(i,m))\in\Gamma_h$, let
\begin{align*}
 \mathcal C_\gamma^h:=\bigl\{(x,v,v_*,\sigma):{}
 x\in X_i^h,\ v\in V_j^h,\ v_*\in V_k^h,v'\in V_l^h,\ v_*'\in V_m^h\bigr\}.
\end{align*}
The definition of $q_\gamma^h$ can then be written as
\begin{equation*}
 q_\gamma^h(t)
 =\frac{\displaystyle\int_{\mathcal C_\gamma^h}
 \mathcal B_R(c)q(t,c)\,dc}
 {\displaystyle\int_{\mathcal C_\gamma^h}\mathcal B_R(c)\,dc}.
\end{equation*}
The denominator is positive for every $\gamma\in\Gamma_h$.  Moreover, apart
from collision configurations for which one of the incoming or outgoing
velocities lies on a cell boundary, every $c\in\mathcal C_R$ with
$\mathcal B_R(c)>0$ belongs to exactly one set $\mathcal C_\gamma^h$.  More
precisely,
\begin{equation*}
 \int_{\mathcal C_R\setminus
       \bigcup_{\gamma\in\Gamma_h}\mathcal C_\gamma^h}
 \mathcal B_R(c)\,dc=0,
 \qquad
 \int_{\mathcal C_\gamma^h\cap\mathcal C_{\widetilde\gamma}^h}
 \mathcal B_R(c)\,dc=0
 \quad\text{if }\gamma\neq\widetilde\gamma.
\end{equation*}
Thus the sets $(\mathcal C_\gamma^h)_{\gamma\in\Gamma_h}$ form a partition
of $\mathcal C_R$ modulo sets of zero $\mathcal B_R(c)\,dc$-measure.
Consequently, \eqref{riemann-collision-quadrature} gives
\begin{align*}
 &\frac12\int_0^T\sum_{\gamma=(a,b;c,d)\in\Gamma_h}
 K_\gamma^hm_a^h(t)m_b^h(t)\ell(q_\gamma^h(t))\,dt\\
 &\qquad=\frac12\int_0^T\sum_{\gamma\in\Gamma_h}
 \int_{\mathcal C_\gamma^h}\mathcal B_R(c)
 \bigl(\mathsf R_hm^h(t)\bigr)(x,v)
 \bigl(\mathsf R_hm^h(t)\bigr)(x,v_*)
 \ell(q_\gamma^h(t))\,dc\,dt.
\end{align*}

We compare each $q_\gamma^h$ with $q$ on the corresponding channel set.
Since $q_\gamma^h$ is a weighted average of $q$,
\begin{equation*}
 0\leq q_\gamma^h(t)\leq\|q\|_\infty,
 \qquad t\in[0,T],\quad \gamma\in\Gamma_h.
\end{equation*}
For $\eta>0$, write
\begin{equation*}
 \mathcal C_{R,\eta}:=
 \{c\in\mathcal C_R:|v-v_*|\leq\eta\}.
\end{equation*}
Because $\mathcal B_R$ is bounded and supported in the bounded velocity
domain,
\begin{align*}
 \int_0^T\int_{\mathcal C_{R,\eta}}\mathcal B_R(c)\,dc\,dt
 &\leq T\|\mathcal B_R\|_\infty
 |\mathbb T^d|\,|\mathbb S^{d-1}|\,|B_R|\,\omega_d\eta^d
 \longrightarrow0
 \qquad\text{as }\eta\downarrow0.
\end{align*}
Here $\omega_d$ is the volume of the unit ball in $\mathbb R^d$; for each
fixed $v$, the condition $|v-v_*|\leq\eta$ restricts $v_*$ to a set of
volume at most $\omega_d\eta^d$.
Since $\mathsf R_hm^h$, $f$, and
$\ell(q_\gamma^h)$ are uniformly bounded, the contributions of
$\mathcal C_{R,\eta}$ to both the discrete and limiting actions are
uniformly small as $\eta\downarrow0$.

Fix $\eta>0$ and take $h_v$ sufficiently small that
$C_dh_v<\eta/2$, where $C_d$ is the dimensional constant in the cell
diameter estimates below.  If two collision configurations
$c=(x,v,v_*,\sigma)$ and
$\widetilde c=(\widetilde x,\widetilde v,
\widetilde v_*,\widetilde\sigma)$ belong to the same channel set
$\mathcal C_\gamma^h$, where
$\gamma=((i,j),(i,k);(i,l),(i,m))$, then the definition of
$\mathcal C_\gamma^h$ gives
\begin{equation*}
 v,\widetilde v\in V_j^h,\qquad
 v_*,\widetilde v_*\in V_k^h,
 \qquad v',\widetilde v'\in V_l^h,
 \qquad v_*',\widetilde v_*'\in V_m^h.
\end{equation*}
Consequently, if $|v-v_*|>\eta$, the diameters of the spatial and velocity
cells give
\begin{equation*}
 d_{\mathbb T^d}(x,\widetilde x)\leq C_dh_x,\qquad
 |v-\widetilde v|+|v_*-\widetilde v_*|\leq C_dh_v,
 \qquad
 |v'-\widetilde v'|+|v_*'-\widetilde v_*'|\leq C_dh_v.
\end{equation*}
In particular, the reverse triangle inequality gives
\begin{align*}
 |\widetilde v-\widetilde v_*|
 &\geq |v-v_*|-|v-\widetilde v|-|v_*-\widetilde v_*|\\
 &>\eta-C_dh_v>\frac{\eta}{2}.
\end{align*}
Thus the second relative velocity is bounded below by $\eta/2$.  Using
\begin{equation*}
 \sigma=\frac{v'-v_*'}{|v-v_*|},\qquad
 \widetilde\sigma
 =\frac{\widetilde v'-\widetilde v_*'}
 {|\widetilde v-\widetilde v_*|},
\end{equation*}
and the preservation of relative speed by an elastic collision, we obtain
\begin{align*}
 |\sigma-\widetilde\sigma|
 \leq\frac{|(v'-v_*')-(\widetilde v'-\widetilde v_*')|}
 {|v-v_*|}
 +|\widetilde v'-\widetilde v_*'|
 \left|\frac1{|v-v_*|}
 -\frac1{|\widetilde v-\widetilde v_*|}\right|.
\end{align*}
Since
$|\widetilde v'-\widetilde v_*'|=|\widetilde v-\widetilde v_*|$, the
second term on the right is
\begin{align*}
 \frac{\bigl||v-v_*|-|\widetilde v-\widetilde v_*|\bigr|}
 {|v-v_*|}
 \leq\frac{|v-\widetilde v|+|v_*-\widetilde v_*|}{|v-v_*|}
 \leq C_d\frac{h_v}{\eta}.
\end{align*}
The first term satisfies the same bound by the estimate for the outgoing
velocities.  Hence $|\sigma-\widetilde\sigma|\leq C_dh_v/\eta$.
Let $L_q$ be a Lipschitz constant of $q$ with respect to
$(x,v,v_*,\sigma)$ on the truncated collision domain.  Fix
$t\in[0,T]$, $\gamma\in\Gamma_h$, and
$c=(x,v,v_*,\sigma)\in\mathcal C_\gamma^h$ with
$|v-v_*|>\eta$.  Writing
$\widetilde c=(\widetilde x,\widetilde v,
\widetilde v_*,\widetilde\sigma)$ for the integration variable, the
weighted-average definition of $q_\gamma^h$ gives
\begin{align*}
 |q_\gamma^h(t)-q(t,c)|
 &=\left|\frac{\displaystyle
    \int_{\mathcal C_\gamma^h}\mathcal B_R(\widetilde c)
       \bigl(q(t,\widetilde c)-q(t,c)\bigr)\,d\widetilde c}
   {\displaystyle
    \int_{\mathcal C_\gamma^h}\mathcal B_R(\widetilde c)\,d\widetilde c}
   \right|\\
 &\leq\frac{\displaystyle
    \int_{\mathcal C_\gamma^h}\mathcal B_R(\widetilde c)L_q
    \left(d_{\mathbb T^d}(x,\widetilde x)
    +|v-\widetilde v|+|v_*-\widetilde v_*|
    +|\sigma-\widetilde\sigma|\right)\,d\widetilde c}
   {\displaystyle
    \int_{\mathcal C_\gamma^h}\mathcal B_R(\widetilde c)\,d\widetilde c}\\
 &\leq C_dL_q\left(h_x+h_v+\frac{h_v}{\eta}\right).
\end{align*}
The denominator is positive because $\gamma\in\Gamma_h$.  Taking the
supremum over $t$, $\gamma$, and $c$ yields
\begin{align*}
 \sup_{\substack{t\in[0,T],\ \gamma\in\Gamma_h,\ c\in\mathcal C_\gamma^h\\
                  |v-v_*|>\eta}}
 |q_\gamma^h(t)-q(t,c)|
 \leq C_dL_q\left(h_x+h_v+\frac{h_v}{\eta}\right)
 \longrightarrow0
 \qquad\text{as }h_x+h_v\downarrow0.
\end{align*}
Since $\ell$ is continuous on the compact interval
$[0,\|q\|_\infty]$, it is uniformly continuous there.  The last estimate
therefore implies
\begin{align*}
 \sup_{\substack{t\in[0,T],\ \gamma\in\Gamma_h,\ c\in\mathcal C_\gamma^h\\
                  |v-v_*|>\eta}}
 |\ell(q_\gamma^h(t))-\ell(q(t,c))|
 \longrightarrow0
 \qquad\text{as }h_x+h_v\downarrow0.
\end{align*}

Combining the preceding estimates, we obtain
\begin{align*}
 &\left|\frac12\int_0^T
 \sum_{\gamma=(a,b;c,d)\in\Gamma_h}
 K_\gamma^hm_a^h(t)m_b^h(t)\ell(q_\gamma^h(t))\,dt
 -\frac12\int_0^T\int_{\mathcal C_R}
 \mathcal B_R(c)f(t,x,v)f(t,x,v_*)\ell(q(t,c))\,dc\,dt\right|\\
 &\quad\leq\frac12\int_0^T\sum_{\gamma\in\Gamma_h}
 \int_{\mathcal C_\gamma^h}\mathcal B_R(c)
 \left|\bigl(\mathsf R_hm^h(t)\bigr)(x,v)
 \bigl(\mathsf R_hm^h(t)\bigr)(x,v_*)
 -f(t,x,v)f(t,x,v_*)\right|
 \ell(q_\gamma^h(t))\,dc\,dt\\
 &\qquad+\frac12\int_0^T\sum_{\gamma\in\Gamma_h}
 \int_{\mathcal C_\gamma^h}\mathcal B_R(c)f(t,x,v)f(t,x,v_*)
 \left|\ell(q_\gamma^h(t))-\ell(q(t,c))\right|\,dc\,dt.
\end{align*}
The first term tends to zero by the uniform convergence of
$\mathsf R_hm^h$ to $f$.  For the second term, boundedness of $f$ gives
\begin{align*}
 &\frac12\int_0^T\sum_{\gamma\in\Gamma_h}
 \int_{\mathcal C_\gamma^h}\mathcal B_R(c)f(t,x,v)f(t,x,v_*)
 \left|\ell(q_\gamma^h(t))-\ell(q(t,c))\right|\,dc\,dt\\
 &\quad\leq C_f
 \sup_{\substack{t\in[0,T],\ \gamma\in\Gamma_h,\ c\in\mathcal C_\gamma^h\\
                  |v-v_*|>\eta}}
 |\ell(q_\gamma^h(t))-\ell(q(t,c))|
 \int_0^T\int_{\mathcal C_R}\mathcal B_R(c)\,dc\,dt\\
 &\qquad+2C_f\sup_{0\leq s\leq\|q\|_\infty}\ell(s)
 \int_0^T\int_{\mathcal C_{R,\eta}}\mathcal B_R(c)\,dc\,dt.
\end{align*}
For fixed $\eta$, the first term on the right tends to zero as
$h_x+h_v\downarrow0$.  The second term is independent of $h$ and tends to
zero as $\eta\downarrow0$.  Thus the second term also tends to zero.  This proves
\eqref{regular-controlled-action-convergence}.
\end{proof}

The preceding recovery statement can be combined with the dual
representation of the finite-mesh action to give convergence of the rate
functions themselves on the regular class.  The corresponding equivalence
between the Hamiltonian and collision-control representations of the
continuum rate is proved in
Lemma~\ref{lem:Hamiltonian-control-rate-comparison} of
Appendix~\ref{app:fluctuating-boltzmann-rate}.  At fixed $h$, the channel set
$\mathcal K_h$ is finite, and the same channelwise convex-duality argument
applies directly.  Thus, for
$m\in AC([0,T];\mathbb R_{\geq0}^{\mathcal I_h})$,
\begin{align}
 I_h(m)=\sup_{p\in C^1([0,T];\mathbb R^{\mathcal I_h})}
 \Bigg\{&\langle m(T),p(T)\rangle-\langle m(0),p(0)\rangle
 -\int_0^T\langle m(t),\dot p(t)\rangle\,dt\notag\\
 &-\int_0^T\mathcal H_h(m(t),p(t))\,dt\Bigg\}.
 \label{finite-mesh-hamiltonian-action-duality}
\end{align}
The pointwise conjugacy used here is, for every $\beta\geq0$,
\begin{equation*}
 \sup_{q\geq0}\{r\beta q-\beta\ell(q)\}
 =\beta(e^r-1),
\end{equation*}
where the maximizing multiplier is $q=e^r$.  Taking this conjugate subject
to
$\dot m=\sum_{\kappa\in\mathcal K_h}
\beta_\kappa(m)q_\kappa z_\kappa$ yields the Hamiltonian
\eqref{discrete-regular-hamiltonian}; integration in time and
Fenchel--Moreau duality give
\eqref{finite-mesh-hamiltonian-action-duality}.  The formula remains valid
with value $+\infty$ outside the absolutely continuous paths.

\begin{theorem}[Fixed-cutoff regular-path convergence of the rate functions]
\label{thm:regular-path-rate-convergence-rigorous}
Assume that $f_0$ is Lipschitz on $\mathbb D_R$ and satisfies
$0<c_0\leq f_0\leq C_0<\infty$.  Let $g$ be a
biased regular path, with regular control $q=e^{\Delta p}$, and take
$m_{*,a}^h=(\Pi_hf_0)_a$.  Let
\begin{equation*}
 \widetilde m^h\in
 D([0,T];\mathbb R_{\geq0}^{\mathcal I_h})
\end{equation*}
for each $h$, and suppose that
\begin{equation*}
 d_{\infty,R}^h(\widetilde m^h,g)
 =\sup_{t\in[0,T]}
 \|\mathsf R_h\widetilde m^h(t)-g_t\|_{L^\infty(\mathbb D_R)}
 \longrightarrow0.
\end{equation*}
Then
\begin{equation}\label{fixed-cutoff-rate-liminf}
 \liminf_{h_x+h_v\downarrow0}
 \mathcal I_h^{\rm Poi}(\widetilde m^h)
 \geq I_{\mathcal B_R}(g).
\end{equation}
There is also a sequence
$m^h\in AC([0,T];\mathbb R_{\geq0}^{\mathcal I_h})$ for which
$d_{\infty,R}^h(m^h,g)\to0$ and
\begin{equation}\label{fixed-cutoff-rate-recovery}
 \lim_{h_x+h_v\downarrow0}\mathcal I_h^{\rm Poi}(m^h)
 =I_{\mathcal B_R}(g).
\end{equation}
Thus the $\Gamma$-liminf and recovery properties hold at $g$ in the strong
reconstruction topology displayed above.
\end{theorem}

\begin{proof}
We begin with the lower bound.  Since $g$ is Lipschitz, convergence of the
reconstructed densities implies
\begin{equation}\label{regular-rate-strong-convergence}
 \sup_{t\in[0,T]}
 \|\widetilde m^h(t)-\Pi_hg_t\|_{\infty,h}\longrightarrow0.
\end{equation}
If the left-hand side of \eqref{fixed-cutoff-rate-liminf} is infinite, the
claim is immediate.  Otherwise, choose a sequence of meshes that realizes
the liminf.  After discarding finitely many terms, we may
assume that
\begin{equation*}
 \mathcal I_h^{\rm Poi}(\widetilde m^h)<\infty.
\end{equation*}
The initial entropy in \eqref{fixed-h-poisson-full-rate-rigorous} is
nonnegative, and hence $I_h(\widetilde m^h)<\infty$.  The definition
\eqref{fixed-h-free-initial-action} then implies
\begin{equation*}
 \widetilde m^h\in
 AC([0,T];\mathbb R_{\geq0}^{\mathcal I_h}).
\end{equation*}
For
$r\in C^1([0,T];W^{2,\infty}(\mathbb D_R))$, insert
$p^h(t)=\mathsf P_hr(t)$ in
\eqref{finite-mesh-hamiltonian-action-duality}.  The identity
\begin{equation*}
 \sum_{a\in\mathcal I_h}m_a(\mathsf P_hr)_a
 =\int_{\mathbb D_R}(\mathsf R_hm)(z)r(z)\,dz
\end{equation*}
and \eqref{regular-rate-strong-convergence} give
\begin{align*}
 \langle\widetilde m^h(T),\mathsf P_hr_T\rangle
 &\longrightarrow\langle g_T,r_T\rangle,
 &
 \langle\widetilde m^h(0),\mathsf P_hr_0\rangle
 &\longrightarrow\langle g_0,r_0\rangle,\\
 \int_0^T\langle\widetilde m^h(t),
                  \mathsf P_h\partial_tr_t\rangle\,dt
 &\longrightarrow\int_0^T\langle g_t,\partial_tr_t\rangle\,dt.
\end{align*}

We next replace $\Pi_hg_t$ by $\widetilde m^h(t)$ in the discrete
Hamiltonian.  For the transport part, the upwind definition and
$|e^s-1|\leq e^{|s|}|s|$ give
\begin{align*}
 \left|\sum_{(a,b)\in\mathcal E_h}
 \bigl(\widetilde m_a^h(t)-(\Pi_hg_t)_a\bigr)T_{ab}^h
 \left(e^{(\mathsf P_hr_t)_b-(\mathsf P_hr_t)_a}-1\right)\right|\leq C_{R,r}
 \|\mathsf R_h\widetilde m^h(t)-\mathsf R_h\Pi_hg_t
 \|_{L^\infty(\mathbb D_R)}.
\end{align*}
For the collision part, the cell-integrated definition of $K_\gamma^h$ gives 
\begin{align*}
 &\left|\frac12\sum_{\gamma=(a,b;c,d)\in\Gamma_h}
 K_\gamma^h
 \left(\widetilde m_a^h(t)\widetilde m_b^h(t)
       -(\Pi_hg_t)_a(\Pi_hg_t)_b\right)
 \left(e^{(\mathsf P_hr_t)_c+(\mathsf P_hr_t)_d
              -(\mathsf P_hr_t)_a-(\mathsf P_hr_t)_b}-1\right)\right|\\
 &\qquad\leq C_{R,r}
 \left(\|\mathsf R_h\widetilde m^h(t)\|_{L^\infty(\mathbb D_R)}
       +\|\mathsf R_h\Pi_hg_t\|_{L^\infty(\mathbb D_R)}\right)
 \|\mathsf R_h\widetilde m^h(t)-\mathsf R_h\Pi_hg_t
 \|_{L^\infty(\mathbb D_R)}.
\end{align*}
The two norms in parentheses are uniformly bounded.  Moreover,
\eqref{regular-rate-strong-convergence} and the uniform convergence of the
cell averages of $g$ imply
\begin{equation*}
 \sup_{t\in[0,T]}
 \|\mathsf R_h\widetilde m^h(t)-\mathsf R_h\Pi_hg_t
 \|_{L^\infty(\mathbb D_R)}\longrightarrow0.
\end{equation*}
It follows that
\begin{equation*}
 \sup_{t\in[0,T]}
 \left|\mathcal H_h(\widetilde m^h(t),\mathsf P_hr_t)
 -\mathcal H_h(\Pi_hg_t,\mathsf P_hr_t)\right|\longrightarrow0.
\end{equation*}
The estimates in the proof of
Lemma~\ref{lem:regular-hamiltonian-consistency}, applied uniformly to
$(g_t,r_t)_{t\in[0,T]}$, give
\begin{align*}
 \sup_{t\in[0,T]}\left|\mathcal H_h(\Pi_hg_t,\mathsf P_hr_t)
 -\langle g_t,v\cdot\nabla_xr_t\rangle
 -\frac12\int_{\mathcal C_R}\mathcal B_R(c)g_tg_{t,*}
       (e^{\Delta r_t}-1)\,dc\right|\longrightarrow0.
\end{align*}
Applying \eqref{finite-mesh-hamiltonian-action-duality} with
$p^h=\mathsf P_hr$ and using these convergences therefore yields
\begin{align*}
 \liminf_{h_x+h_v\downarrow0}I_h(\widetilde m^h)
 \geq{}&\langle g_T,r_T\rangle-\langle g_0,r_0\rangle
 -\int_0^T\langle g_t,\partial_tr_t+v\cdot\nabla_xr_t\rangle\,dt\\
 &-\frac12\int_0^T\int_{\mathcal C_R}
 \mathcal B_Rg_tg_{t,*}(e^{\Delta r_t}-1)\,dc\,dt.
\end{align*}
Taking the supremum over $r$ gives
\begin{equation}\label{regular-rate-dynamical-liminf}
 \liminf_{h_x+h_v\downarrow0}I_h(\widetilde m^h)
 \geq\mathcal I_{\mathcal B_R}^{\rm Ham}(g).
\end{equation}

Since $f_0$ and $g_0$ are positive and regular on the bounded domain,
\eqref{regular-rate-strong-convergence} and the continuity of
$(u,a)\mapsto a\ell(u/a)$ for $a>0$ imply
\begin{equation}\label{regular-initial-entropy-convergence}
 \sum_{a\in\mathcal I_h}
 (\Pi_hf_0)_a
 \ell\left(\frac{\widetilde m_a^h(0)}{(\Pi_hf_0)_a}\right)
 \longrightarrow I_{0,R}(g_0).
\end{equation}
Both cell masses contain the same cell-volume factor, so this sum is a
Riemann approximation of the integral defining $I_{0,R}$.  Since $g$ is a
biased regular path,
Proposition~\ref{prop:fixed-cutoff-biased-duality} gives
\begin{equation*}
 \mathcal I_{\mathcal B_R}^{\rm Ham}(g)
 =\frac12\int_0^T\int_{\mathcal C_R}
 \mathcal B_Rgg_*\ell(q)\,dc\,dt.
\end{equation*}
Combining this identity with
\eqref{regular-rate-dynamical-liminf} and
\eqref{regular-initial-entropy-convergence} proves the liminf inequality.

For the recovery bound, take the controlled solutions $m^h$ and controls
$q^h$ constructed in
Lemma~\ref{lem:regular-controlled-path-cost-consistency}, now with initial value
$\Pi_hg_0$.  The proof of that lemma gives the stronger convergence
\begin{equation*}
 \sup_{t\in[0,T]}\|m^h(t)-\Pi_hg_t\|_{\infty,h}\longrightarrow0.
\end{equation*}
In the control formulation of $I_h(m^h)$, choose $q_{(a,b)}^h=1$ on every
transport edge and use the multipliers $q_\gamma^h$ constructed above on
the collision channels.  Since $\ell(1)=0$, the transport contribution to
the cost vanishes, and hence
\begin{equation*}
 I_h(m^h)
 \leq\frac12\int_0^T\sum_{\gamma\in\Gamma_h}
 K_\gamma^hm_a^h(t)m_b^h(t)\ell(q_\gamma^h(t))\,dt.
\end{equation*}
Equations \eqref{regular-controlled-action-convergence} and
\eqref{regular-initial-entropy-convergence} therefore imply
\begin{equation*}
 \limsup_{h_x+h_v\downarrow0}\mathcal I_h^{\rm Poi}(m^h)
 \leq I_{0,R}(g_0)+\frac12\int_0^T\int_{\mathcal C_R}
 \mathcal B_Rgg_*\ell(q)\,dc\,dt
 =I_{\mathcal B_R}(g).
\end{equation*}
Applying the lower bound to this recovery sequence gives
\eqref{fixed-cutoff-rate-recovery}.
\end{proof}

\begin{remark}
The purpose of this section is to test whether the coarse-grained jump
model retains the large-deviation structure of the underlying Boltzmann
dynamics when the mesh is refined.  For every fixed cutoff $R$, the
$\Gamma$-liminf and recovery results above identify the limit of the
discrete rate functions on the biased regular class as
$I_{\mathcal B_R}$.  Thus the convergence concerns not only the
deterministic Boltzmann drift: it also recovers the initial entropy cost and
the nonlinear collision cost that govern rare deviations.

When the original collision kernel is the hard-sphere kernel and
$\mathcal B_R$ is its regularized microreversible truncation,
$I_{\mathcal B_R}$ is the corresponding truncated version of the
hard-sphere particle-system rate function recalled in
\eqref{hard-sphere-path-rate}.  Therefore, already at a fixed velocity
cutoff, the mesh limit shows that the Poissonian coarse-grained model
captures the correct continuum large-deviation action.  The cutoff
separates this identification from the additional problem of controlling
large velocities; removing it is a further step rather than part of the
mesh-consistency result proved here.
\end{remark}

\subsection*{Acknowledgements}
The author is grateful to Zihui He for helpful discussions and comments.

\appendix

\renewcommand{\theequation}{\Alph{section}.\arabic{equation}}

\section{Large-deviation framework for Poissonian SPDEs}
\label{app:poisson-spde-ldp-framework}

This appendix recalls the part of the Poisson-noise framework of Budhiraja,
Chen, and Dupuis needed here; see
\cite[(1.1), (2.1)--(2.4), and (3.10)]{BudhirajaChenDupuis}.  Let
$\mathbb X$ be a locally compact Polish mark space, let $\nu$ be a Borel
measure on $\mathbb X$ that is finite on compact sets, and put
\begin{equation*}
 \mathbb X_T=[0,T]\times\mathbb X,
 \qquad
 \nu_T(dt,d\eta)=dt\,\nu(d\eta).
\end{equation*}
Following the Poisson-random-measure convention of
Section~\ref{sec:coarse-setup-main}, for $\varepsilon>0$ let
$N^{\varepsilon^{-1}}$ be a Poisson random measure on $\mathbb X_T$ with
intensity $\varepsilon^{-1}\nu_T$, and define
\begin{equation*}
 \widetilde N^{\varepsilon^{-1}}(dt,d\eta)
 :=N^{\varepsilon^{-1}}(dt,d\eta)
   -\varepsilon^{-1}dt\,\nu(d\eta),
\end{equation*}
with the usual interpretation under predictable stochastic integrals.

Let $\mathsf S$ be a separable Banach space and set
\begin{equation*}
 \mathbb U_T=D([0,T];\mathsf S),
\end{equation*}
where $D([0,T];\mathsf S)$ denotes the space of $\mathsf S$-valued
right-continuous paths with left limits, equipped with the Skorokhod
topology.  The state space $\mathsf S$ depends on the SPDE under
consideration.  Since we use the abstract framework only to derive the
formal rate function, we assume that the original and controlled equations
define $\mathbb U_T$-valued processes without specifying $\mathsf S$
further.  A rigorous application also requires the
well-posedness, compactness, and controlled-convergence assumptions in
\cite{BudhirajaChenDupuis}.  For coefficients
\begin{equation*}
 A:[0,T]\times\mathsf S\to\mathsf S,
 \qquad
 G:[0,T]\times\mathsf S\times\mathbb X\to\mathsf S,
\end{equation*}
the abstract small-noise equation for $X^\varepsilon\in\mathbb U_T$ is
\begin{align}
 X_t^\varepsilon
 =X_0^\varepsilon
 +\int_0^t A(s,X_s^\varepsilon)\,ds
 +\varepsilon\int_0^t\int_{\mathbb X}
 G(s,X_{s-}^\varepsilon,\eta)\,
 \widetilde N^{\varepsilon^{-1}}(ds,d\eta).
 \label{BCD-abstract-Poisson-SPDE}
\end{align}
The jumps have size of order $\varepsilon$ and intensity of order
$\varepsilon^{-1}$.  In the variational representation, a
nonnegative predictable control changes the compensator to
$\varepsilon^{-1}h(t,\eta)dt\,\nu(d\eta)$.  To construct the controlled
measure, introduce an
auxiliary variable $\theta\in[0,\infty)$ and a Poisson random measure
$\overline N$ on
$[0,T]\times\mathbb X\times[0,\infty)$ with intensity
$dt\,\nu(d\eta)\,d\theta$.  For a nonnegative predictable $h$, define
\begin{align}
 N^{\varepsilon^{-1}h}((0,t]\times B)
 :=\int_{(0,t]\times B\times[0,\infty)}
 \boldsymbol1_{\{0\leq\theta\leq
          \varepsilon^{-1}h(s,\eta)\}}
 \,\overline N(ds,d\eta,d\theta),
 \qquad B\in\mathcal B(\mathbb X).
 \label{BCD-controlled-PRM}
\end{align}
The predictability of $h$ implies that the acceptance indicator
$(s,\eta,\theta)\mapsto
\boldsymbol1_{\{0\leq\theta\leq\varepsilon^{-1}h(s,\eta)\}}$ is a predictable
integrand for $\overline N$.  Since
\begin{equation*}
 \int_0^\infty
 \boldsymbol1_{\{0\leq\theta\leq\varepsilon^{-1}h(s,\eta)\}}\,d\theta
 =\varepsilon^{-1}h(s,\eta),
\end{equation*}
the compensation identity
\eqref{predictable-acceptance-compensation} shows that the
predictable compensator of $N^{\varepsilon^{-1}h}$ is
\begin{equation*}
 \varepsilon^{-1}h(s,\eta)\,ds\,\nu(d\eta).
\end{equation*}
For $h\equiv1$, this random measure has the same law as
$N^{\varepsilon^{-1}}$.  The rate function uses nonnegative measurable
controls $h$ with cost
\begin{equation}\label{BCD-control-cost}
 L_T(h)=\int_0^T\int_{\mathbb X}
 \ell(h(t,\eta))\,\nu(d\eta)\,dt,
 \qquad \ell(a)=a\log a-a+1,
\end{equation}
where $0\log0:=0$.  The variational representation uses random predictable
controls, while the skeleton and rate function below use deterministic
measurable controls of finite cost.  The continuum Boltzmann calculation
uses this drift--jump structure.  The functional-analytic assumptions of
\cite{BudhirajaChenDupuis} are verified for the finite-mesh model in
Section~\ref{sec:fixed-mesh-ldp-rigorous}.

We use the fixed-initial-condition case
$X_0^\varepsilon=X_0$ used in the large-deviation theorem of
\cite{BudhirajaChenDupuis}.  Replacing the original Poisson random measure
by \eqref{BCD-controlled-PRM} gives the controlled stochastic equation
\begin{align}
 X_t^{\varepsilon,h}
 ={}&X_0+\int_0^t A(s,X_s^{\varepsilon,h})\,ds\notag\\
 &+\int_0^t\int_{\mathbb X}
 G(s,X_{s-}^{\varepsilon,h},\eta)
 \left(\varepsilon
 N^{\varepsilon^{-1}h}(ds,d\eta)-ds\,\nu(d\eta)\right).
 \label{BCD-controlled-stochastic-equation}
\end{align}
For a deterministic control $h$, the corresponding zero-noise skeleton is
\begin{align}
 X_t^h
 =X_0
 +\int_0^t A(s,X_s^h)\,ds
 +\int_0^t\int_{\mathbb X}
 G(s,X_s^h,\eta)\bigl(h(s,\eta)-1\bigr)
 \,\nu(d\eta)\,ds.
 \label{BCD-abstract-controlled-skeleton}
\end{align}
The factor $h-1$ comes from the compensated form of
\eqref{BCD-abstract-Poisson-SPDE}.  More precisely,
\begin{align*}
 &\varepsilon N^{\varepsilon^{-1}h}(ds,d\eta)
   -ds\,\nu(d\eta)\\
 &\quad=\varepsilon\left(
 N^{\varepsilon^{-1}h}(ds,d\eta)
 -\varepsilon^{-1}h(s,\eta)ds\,\nu(d\eta)\right)
 +(h(s,\eta)-1)ds\,\nu(d\eta).
\end{align*}
Under the controlled-convergence assumptions of
\cite{BudhirajaChenDupuis}, the stochastic integral against the compensated
term in parentheses, multiplied by $\varepsilon$, converges to zero as
$\varepsilon\downarrow0$.

We parametrize the skeleton by the control and set
\begin{equation*}
 \mathcal G_0(h):=X^h,
\end{equation*}
where $X^h$ solves \eqref{BCD-abstract-controlled-skeleton}.  Equivalently,
$\mathcal G_0$ is the zero-noise solution map of
\cite{BudhirajaChenDupuis}, after identifying a control $h$ with the
controlled intensity measure $h(t,\eta)dt\,\nu(d\eta)$.  This gives the rate
function
\begin{equation}\label{BCD-abstract-rate-function}
 I_{X_0}(\phi)
 =\inf\left\{L_T(h):
 \phi=\mathcal G_0(h)\right\},
 \qquad \phi\in\mathbb U_T,
\end{equation}
with value $+\infty$ if no such control exists.  By
\cite{BudhirajaChenDupuis}, this is the large-deviation rate function once
the required compactness and convergence properties of the controlled
equations have been verified.

For \eqref{poisson-fluctuating-boltzmann-strong}, the terms in this framework
are identified as follows.  Recall the collision activity $a_g$ from
\eqref{Boltzmann-accepted-collision-activity}.
We continue to denote the acceptance mark by $u$, as in
\eqref{poisson-fluctuating-boltzmann-strong}, and define
$\Theta_\varepsilon(t,c,u)=(t,c,\varepsilon u)$.  More explicitly,
the push-forward random measure
\begin{equation*}
 N^{\varepsilon^{-1}}
 :=(\Theta_\varepsilon)_{\#}\mathfrak N
\end{equation*}
is defined, for every measurable
$A\subset[0,\infty)\times\mathcal C\times[0,\infty)$, by
\begin{align*}
 N^{\varepsilon^{-1}}(A)
 &:=\mathfrak N\bigl(\Theta_\varepsilon^{-1}(A)\bigr)\\
 &=\int_{[0,\infty)\times\mathcal C\times[0,\infty)}
 \boldsymbol1_A(t,c,\varepsilon u)\,
 \mathfrak N(dt,dc,du).
\end{align*}
Thus the map multiplies the third coordinate of every atom of $\mathfrak N$ by
$\varepsilon$, while its time and collision mark are unchanged.
Equivalently, for every nonnegative measurable function $F$,
\begin{equation*}
 \int F(t,c,u)\,N^{\varepsilon^{-1}}(dt,dc,du)
 =\int F(t,c,\varepsilon u)\,\mathfrak N(dt,dc,du).
\end{equation*}
Since $\mathfrak N$ has intensity $dt\,dc\,du$, the intensity of the
push-forward is obtained by scaling the third coordinate:
\begin{align*}
 \mathbb E N^{\varepsilon^{-1}}(A)
 &=\int\boldsymbol1_A(t,c,\varepsilon u)\,dt\,dc\,du\\
 &=\varepsilon^{-1}\int\boldsymbol1_A(t,c,u)\,dt\,dc\,du.
\end{align*}
Hence $N^{\varepsilon^{-1}}$ is a Poisson random measure on the variables
$(t,c,u)$ with intensity $\varepsilon^{-1}dt\,dc\,du$.  In the notation of
\cite{BudhirajaChenDupuis}, the correspondence is
\begin{align}
 &X^\varepsilon=f^\varepsilon,\qquad
 \eta=(c,u),\qquad
 \mathbb X=\mathcal C\times[0,\infty),\qquad
 \nu(d\eta)=dc\,du
 =dx\,dv\,dv_*\,d\sigma\,du,\notag\\
 &A(g)=-v\cdot\nabla_xg+Q(g,g),\qquad
 G(g;c,u)=J_c\boldsymbol1_{\{0\leq u\leq a_g(c)\}}.
\label{Boltzmann-BCD-mark-and-coefficients}
\end{align}
Here $u$ is part of the original mark $\eta$ and implements the
state-dependent collision acceptance rule.  It is distinct from the
auxiliary acceptance variable $\theta$ in \eqref{BCD-controlled-PRM}.  With
this notation, the noise in
\eqref{poisson-fluctuating-boltzmann-strong} can be written as
 \begin{align*}
 \varepsilon\int_{\mathcal C\times[0,\infty)}
 J_c\boldsymbol1_{\{u\leq\varepsilon^{-1}a_{f_{t-}^\varepsilon}(c)\}}
 \,\widetilde{\mathfrak N}(dt,dc,du)=
 \varepsilon\int_{\mathcal C\times[0,\infty)}
 G(f_{t-}^\varepsilon;c,u)\,
 \widetilde N^{\varepsilon^{-1}}(dt,dc,du).
\end{align*}
Moreover, for every
$\varphi\in C_c^\infty(\mathbb T^d\times\mathbb R^d)$,
\begin{equation*}
 \langle G(g;c,u),\varphi\rangle
 =\Delta\varphi(c)\boldsymbol1_{\{0\leq u\leq a_g(c)\}}.
\end{equation*}

For the deterministic skeleton, let $h=h(t,c,u)\geq0$ be a measurable
control of finite cost.  In
\eqref{BCD-abstract-controlled-skeleton}, take
\begin{equation*}
 X^h=g,\qquad \eta=(c,u),\qquad
 \nu(d\eta)=dc\,du,
\end{equation*}
and use the definitions of $A$ and $G$ in
\eqref{Boltzmann-BCD-mark-and-coefficients}.  Substitution into
\eqref{BCD-abstract-controlled-skeleton} gives the following identity in
the sense of distributions on $\mathbb T^d\times\mathbb R^d$:
\begin{align*}
 g_t
 =g_0+\int_0^t
 \bigl[-v\cdot\nabla_xg_s+Q(g_s,g_s)\bigr]\,ds+\int_0^t\int_{\mathcal C}\int_0^\infty
 J_c\boldsymbol1_{\{0\leq u\leq a_{g_s}(c)\}}
 \bigl(h(s,c,u)-1\bigr)\,du\,dc\,ds.
\end{align*}
The indicator restricts the $u$-integration to the accepted interval, so
the controlled term can equivalently be written as
\begin{equation*}
 \int_0^t\int_{\mathcal C}
 J_c\left[\int_0^{a_{g_s}(c)}
 \bigl(h(s,c,u)-1\bigr)\,du\right]dc\,ds.
\end{equation*}
Let
$\varphi\in C_c^\infty(\mathbb T^d\times\mathbb R^d)$.  The three
terms in the preceding distributional identity satisfy
\begin{align*}
 \langle-v\cdot\nabla_xg_s,\varphi\rangle
 &=\langle g_s,v\cdot\nabla_x\varphi\rangle,\\
 \langle Q(g_s,g_s),\varphi\rangle
 &=\int_{\mathcal C}a_{g_s}(c)\Delta\varphi(c)\,dc,\\
 \langle J_c,\varphi\rangle&=\Delta\varphi(c).
\end{align*}
Here the second identity follows from the weak form of the Boltzmann
collision operator and the definition
$a_g(c)=\frac12B(v-v_*,\sigma)g(x,v)g(x,v_*)$.  Pairing the distributional
identity with $\varphi$ and using these formulas gives
\begin{align*}
 \langle g_t,\varphi\rangle-\langle g_0,\varphi\rangle
 ={}&\int_0^t\langle g_s,v\cdot\nabla_x\varphi\rangle\,ds\\
 &+\int_0^t\int_{\mathcal C}
 \left[a_{g_s}(c)+\int_0^{a_{g_s}(c)}
       \bigl(h(s,c,u)-1\bigr)\,du\right]
 \Delta\varphi(c)\,dc\,ds.
\end{align*}
This is the controlled skeleton \eqref{BCD-Boltzmann-skeleton}.  Its control
cost in \eqref{BCD-abstract-rate-function} gives the fluctuating Boltzmann
rate function
\eqref{BCD-Boltzmann-rate-before-reduction}.

\section{The rate function of the fluctuating Boltzmann equation}
\label{app:fluctuating-boltzmann-rate}

Recall that the Poisson-control rate $I_{X_0}^{\rm FB}$, the
collision-control action $\mathcal I_B$, and the Hamiltonian action
$\mathcal I_{\rm Ham}$ are defined in
\eqref{BCD-Boltzmann-rate-before-reduction},
\eqref{boltzmann-path-rate-control}, and
\eqref{hard-sphere-hamiltonian-action}, respectively.  The initial rate
$I_0$ and the full hard-sphere rate $I_{\rm HS}$ are given in
\eqref{hard-sphere-initial-rate} and \eqref{hard-sphere-path-rate}.
We compare the three dynamical rate functions on the following common
effective domain.
Recall from \eqref{boltzmann-path-space} that
$\mathscr X_T=D([0,T];\mathcal M_+(\mathbb T^d\times\mathbb R^d))$, endowed
with the Skorokhod $J_1$ topology.

\begin{definition}[Common effective domain]
\label{def:common-effective-domain}
Let $\mathscr D_T\subset\mathscr X_T$ consist of paths $g$ such that, for
almost every $t\in[0,T]$, the measure $g_t$ has a nonnegative density,
denoted again by $g_t$, and
\begin{equation}
\label{common-effective-domain-collision-activity}
 \int_0^T\int_{\mathcal C}a_{g_t}(c)\,dc\,dt<\infty
\end{equation}
Moreover, for every
$\varphi\in C_c^\infty(\mathbb T^d\times\mathbb R^d)$, the map
$t\mapsto\langle g_t,\varphi\rangle$ is absolutely continuous.
\end{definition}

All three rate functions are assigned the value $+\infty$ outside their
stated effective domains.

\begin{lemma}
\label{lem:FB-control-rate-equivalence}
For every $g\in\mathscr D_T$ and every prescribed initial state $X_0$,
\begin{equation}\label{BCD-Boltzmann-fixed-initial-rate}
 I_{X_0}^{\rm FB}(g)=
 \begin{cases}
 \mathcal I_B(g),&g_0=X_0,\\
 +\infty,&g_0\neq X_0.
 \end{cases}
\end{equation}
\end{lemma}

\begin{proof}
For each $(t,c)$, the skeleton in \eqref{BCD-Boltzmann-skeleton} depends on
the function $u\mapsto h(t,c,u)$ only through
\begin{equation*}
 a_{g_t}(c)+\int_0^{a_{g_t}(c)}(h(t,c,u)-1)\,du
 =\int_0^{a_{g_t}(c)}h(t,c,u)\,du.
\end{equation*}
Whenever $a_{g_t}(c)>0$, define the average of the control over the accepted
marks by
\begin{equation*}
 q_t(c):=\frac1{a_{g_t}(c)}
 \int_0^{a_{g_t}(c)}h(t,c,u)\,du,
\end{equation*}
and set $q_t(c)=1$ when $a_{g_t}(c)=0$.  With this definition, the expression
in square brackets in \eqref{BCD-Boltzmann-skeleton} equals
$a_{g_s}(c)q_s(c)$.  Hence every control $h$ producing the path $g$ through
\eqref{BCD-Boltzmann-skeleton} produces a multiplier
$q\in\mathcal Q(g)$ in the sense of
\eqref{controlled-boltzmann-equation}.

We first prove one inequality between the costs.  The function
$\ell(z)=z\log z-z+1$ is nonnegative and convex on $[0,\infty)$ because
$\ell''(z)=z^{-1}>0$ for $z>0$, with the convex extension
$\ell(0)=1$.  For $a_{g_t}(c)>0$, apply Jensen's inequality on
$[0,a_{g_t}(c)]$ with respect to the probability measure
$a_{g_t}(c)^{-1}\,du$.  Since the average of $h(t,c,\cdot)$ under this
measure is $q_t(c)$, every control $h$ producing $g$ satisfies
\begin{align*}
 \int_0^T\int_{\mathcal C}\int_0^\infty
 \ell(h(t,c,u))\,du\,dc\,dt&\geq
 \int_0^T\int_{\mathcal C}\int_0^{a_{g_t}(c)}
 \ell(h(t,c,u))\,du\,dc\,dt\\
 &\geq
 \int_0^T\int_{\mathcal C}
 a_{g_t}(c)\ell(q_t(c))\,dc\,dt.
\end{align*}
When $a_{g_t}(c)=0$, the last pointwise inequality is trivial because both
sides vanish.
The last integral is the cost of $q$ in
\eqref{boltzmann-path-rate-control}, because
$a_g=\frac12Bgg_*$.  Taking the infimum over all controls $h$ in
\eqref{BCD-Boltzmann-rate-before-reduction} yields
\begin{equation*}
 \mathcal I_B(g)\leq I_{X_0}^{\rm FB}(g)
 \qquad\text{when }g_0=X_0.
\end{equation*}

Conversely, let $q\in\mathcal Q(g)$ and define explicitly
\begin{equation*}
 h^q(t,c,u):=
 \begin{cases}
 q_t(c),&0\leq u\leq a_{g_t}(c),\\
 1,&u>a_{g_t}(c).
 \end{cases}
\end{equation*}
For this choice, the expression in square brackets in
\eqref{BCD-Boltzmann-skeleton} can be computed explicitly:
\begin{align*}
 a_{g_t}(c)+\int_0^{a_{g_t}(c)}
 \bigl(h^q(t,c,u)-1\bigr)\,du=a_{g_t}(c)
 +\int_0^{a_{g_t}(c)}\bigl(q_t(c)-1\bigr)\,du
 =a_{g_t}(c)q_t(c).
\end{align*}
Therefore, the skeleton equation for $h^q$ is
\eqref{controlled-boltzmann-equation}, which is already satisfied by $g$
because $q\in\mathcal Q(g)$.  Thus $h^q$ is an admissible control producing
the path $g$ in \eqref{BCD-Boltzmann-rate-before-reduction}.

For its cost, split the $u$-integral at
$a_{g_t}(c)$.  Since $h^q=q_t(c)$ on the accepted interval,
$h^q=1$ outside it, and $\ell(1)=0$, we have
\begin{align*}
 \int_0^\infty\ell(h^q(t,c,u))\,du
 &=\int_0^{a_{g_t}(c)}\ell(q_t(c))\,du
   +\int_{a_{g_t}(c)}^\infty\ell(1)\,du=a_{g_t}(c)\ell(q_t(c)).
\end{align*}
Integrating this pointwise identity over $(t,c)$ gives
\begin{equation*}
 \int_0^T\int_{\mathcal C}\int_0^\infty
 \ell(h^q(t,c,u))\,du\,dc\,dt
 =\int_0^T\int_{\mathcal C}
 a_{g_t}(c)\ell(q_t(c))\,dc\,dt.
\end{equation*}
Taking the infimum over $q\in\mathcal Q(g)$ gives the reverse inequality.
If $g_0\neq X_0$, the admissible set in
\eqref{BCD-Boltzmann-rate-before-reduction} is empty.  This proves
\eqref{BCD-Boltzmann-fixed-initial-rate}.
\end{proof}

If the initial condition is random and satisfies an independent
large-deviation principle with rate $I_0$, Lemma~
\ref{lem:FB-control-rate-equivalence} gives the full rate
$I_0(g_0)+\mathcal I_B(g)$.

\begin{lemma}
\label{lem:Hamiltonian-control-rate-comparison}
For every $g\in\mathscr D_T$,
\begin{equation*}
 \mathcal I_{\rm Ham}(g)\leq\mathcal I_B(g).
\end{equation*}
If, in addition, $g\in\mathcal R$, then
\begin{equation*}
 \mathcal I_{\rm Ham}(g)=\mathcal I_B(g),
 \qquad
 I_{\rm HS}(g)=I_0(g_0)+\mathcal I_B(g).
\end{equation*}
\end{lemma}

\begin{proof}
Recall the definition of $\mathcal I_{\rm Ham}$ in
\eqref{hard-sphere-hamiltonian-action}.  We begin with the pointwise convex
duality used in the proof.

For fixed $s\in\mathbb R$, differentiating
$q s-\ell(q)$ over $q>0$ gives
\begin{equation*}
 \frac{d}{dq}\bigl(qs-\ell(q)\bigr)=s-\log q,
 \qquad
 \frac{d^2}{dq^2}\bigl(qs-\ell(q)\bigr)=-\frac1q<0.
\end{equation*}
Its unique maximizer is therefore $q=e^s$, and substitution gives
\begin{equation*}
 \ell^*(s):=\sup_{q\geq0}\bigl(qs-\ell(q)\bigr)=e^s-1.
\end{equation*}
Equivalently, for all $q\geq0$ and $s\in\mathbb R$,
\begin{equation}\label{Boltzmann-Fenchel-inequality}
 qs-\bigl(e^s-1\bigr)\leq\ell(q),
\end{equation}
with equality for $q=e^s$.

We first prove $\mathcal I_{\rm Ham}(g)\leq\mathcal I_B(g)$.  Fix any
$q\in\mathcal Q(g)$ and any admissible $p$.  Applying weak integration by
parts in time to \eqref{controlled-boltzmann-equation} gives
\begin{align*}
 \langle g_T,p_T\rangle-\langle g_0,p_0\rangle
 -\int_0^T\langle g_t,\partial_tp_t+
 v\cdot\nabla_xp_t\rangle\,dt=\int_0^T\int_{\mathcal C}
 a_{g_t}(c)q_t(c)\Delta p_t(c)\,dc\,dt.
\end{align*}
Using the definition of $\mathcal H_{\rm coll}$, the expression inside
the supremum defining $\mathcal I_{\rm Ham}(g)$ is
\begin{align*}
 \int_0^T\int_{\mathcal C}a_{g_t}(c)
 \left[q_t(c)\Delta p_t(c)
 -\bigl(e^{\Delta p_t(c)}-1\bigr)\right]dc\,dt.
\end{align*}
Apply \eqref{Boltzmann-Fenchel-inequality} pointwise with
$s=\Delta p_t(c)$.  Since $a_{g_t}(c)\geq0$, this is bounded above by
\begin{equation*}
 \int_0^T\int_{\mathcal C}
 a_{g_t}(c)\ell(q_t(c))\,dc\,dt.
\end{equation*}
Taking first the supremum over $p$ and then the infimum over
$q\in\mathcal Q(g)$ proves
\begin{equation*}
 \mathcal I_{\rm Ham}(g)\leq\mathcal I_B(g).
\end{equation*}

For the reverse inequality, let $g\in\mathcal R$, where
$\mathcal R=\mathcal R_{r,T}$ is the exponentially biased class defined in
\cite[Theorem~3 and equation~(7.0.7)]{BodineauGallagherSaintRaymondSimonella}.
By the definition of this class, there exists an admissible Lipschitz field
$p$ such that $g$ has the regularity and initial datum specified in that
reference and satisfies the following weak form of the biased Boltzmann
equation: for every
$\varphi\in C_c^\infty(\mathbb T^d\times\mathbb R^d)$ and every
$t\in[0,T]$,
\begin{align*}
 \langle g_t,\varphi\rangle-\langle g_0,\varphi\rangle
 =\int_0^t\langle g_s,v\cdot\nabla_x\varphi\rangle\,ds+\int_0^t\int_{\mathcal C}
 a_{g_s}(c)e^{\Delta p_s(c)}
 \Delta\varphi(c)\,dc\,ds.
\end{align*}
Thus, in the control notation of the present paper,
\begin{equation*}
 q_p(t,c):=e^{\Delta p(t,c)}.
\end{equation*}
The displayed weak equation is \eqref{controlled-boltzmann-equation} with
$q=q_p$.  Moreover, the bounds on $p$ and the Gaussian velocity decay
required in the definition of $\mathcal R$ make the corresponding
collision cost finite.  Hence $q_p\in\mathcal Q(g)$.  For a
general path $g$ and a general field $p$, however, the multiplier
$e^{\Delta p}$ need not belong to $\mathcal Q(g)$ because $g$ need not solve
the corresponding controlled Boltzmann equation.  For the compatible pair
$(g,p)$ above,
$\Delta p(t,c)=\log q_p(t,c)$ and the Fenchel inequality is an equality:
\begin{align*}
 q_p(t,c)\Delta p(t,c)
 -\bigl(e^{\Delta p(t,c)}-1\bigr)
 =q_p(t,c)\log q_p(t,c)-q_p(t,c)+1=\ell(q_p(t,c)).
\end{align*}
The bias $p$ is only Lipschitz and need not belong to the test class in
\eqref{hard-sphere-hamiltonian-action}.  To justify its use in the dual
formula, let $\chi_R$ be a smooth velocity cutoff equal to one on
$\{|v|\leq R\}$ and supported in $\{|v|\leq2R\}$, and let $p^{R,n}$ be a
smooth mollification of $\chi_Rp$.  Then
$p^{R,n}\in C_c^1([0,T]\times\mathbb T^d\times\mathbb R^d)$.  Applying
the weak equation above with the time-dependent test function $p^{R,n}$,
and then sending first $n\to\infty$ and afterwards $R\to\infty$, gives
\begin{align*}
 &\lim_{R\to\infty}\lim_{n\to\infty}\Bigg\{
 \langle g_T,p_T^{R,n}\rangle-\langle g_0,p_0^{R,n}\rangle
 -\int_0^T\langle g_t,\partial_tp^{R,n}_t
       +v\cdot\nabla_xp^{R,n}_t\rangle\,dt\\
 &\hspace{42mm}
 -\int_0^T\mathcal H_{\rm coll}(g_t,p_t^{R,n})\,dt\Bigg\}
 =\int_0^T\int_{\mathcal C}a_{g_t}(c)\ell(q_p(t,c))\,dc\,dt.
\end{align*}
Here the passage to the limit follows from the Gaussian velocity estimate
for the biased Boltzmann solution in
\cite[Appendix~A.1, in particular equation~(A.1.4)]{BodineauGallagherSaintRaymondSimonella},
the uniform bounds inherited from $p\in W^{1,\infty}$, and dominated
convergence.  Since every $p^{R,n}$ is admissible in the supremum defining
$\mathcal I_{\rm Ham}(g)$ and $q_p\in\mathcal Q(g)$, we obtain
\begin{align*}
 \mathcal I_{\rm Ham}(g)\geq\int_0^T\int_{\mathcal C}
 a_{g_t}(c)\ell(q_p(t,c))\,dc\,dt\geq\mathcal I_B(g).
\end{align*}
Combining this with the opposite bound
$\mathcal I_{\rm Ham}(g)\leq\mathcal I_B(g)$ gives
\begin{equation*}
 \mathcal I_{\rm Ham}(g)=\mathcal I_B(g),
 \qquad
 I_{\rm HS}(g)=I_0(g_0)+\mathcal I_B(g).
\end{equation*}
The two inequalities also show that $q_p$ attains the infimum in the
control representation.  The reverse inequality relies on the compatible
exponential tilt $e^{\Delta p}\in\mathcal Q(g)$.  This compatibility removes
the duality gap on the regular class.  For general distribution-valued
paths, the argument still gives
$\mathcal I_{\rm Ham}\leq\mathcal I_B$, but it does not prove equality.
\end{proof}

The comparison above identifies the continuum rate functional on the
exponentially biased class $\mathcal R$.  Extending it to a rigorous
large-deviation principle for the continuum equation would require
well-posedness and compactness for the controlled measure-valued dynamics. 

\bibliographystyle{alpha}
\bibliography{reference}

\newcommand{\etalchar}[1]{$^{#1}$}
\begin{thebibliography}{BDSG{\etalchar{+}}15}

\bibitem[ADVW00]{AlexandreDesvillettesVillaniWennberg}
R.~Alexandre, L.~Desvillettes, C.~Villani, and B.~Wennberg.
\newblock Entropy dissipation and long-range interactions.
\newblock {\em Arch. Ration. Mech. Anal.}, 152:327--355, 2000.

\bibitem[AMU{\etalchar{+}}11]{AlexandreMorimotoUkaiXuYang}
R.~Alexandre, Y.~Morimoto, S.~Ukai, C.-J. Xu, and T.~Yang.
\newblock Global existence and full regularity of the {Boltzmann} equation
  without angular cutoff.
\newblock {\em Comm. Math. Phys.}, 304:513--581, 2011.

\bibitem[AV02]{AlexandreVillaniLongRange}
R.~Alexandre and C.~Villani.
\newblock On the {Boltzmann} equation for long-range interactions.
\newblock {\em Comm. Pure Appl. Math.}, 55:30--70, 2002.

\bibitem[BBBH25]{BasileBenedettoBertiniHeydecker}
G.~Basile, D.~Benedetto, L.~Bertini, and D.~Heydecker.
\newblock Large time analysis of the rate function associated to the
  {Boltzmann} equation: dynamical phase transitions.
\newblock {\em SIAM J. Math. Anal.}, 57:5746--5770, 2025.

\bibitem[BBO26]{BasileBenedettoOrrieri}
G.~Basile, D.~Benedetto, and C.~Orrieri.
\newblock Variational derivation of the homogeneous {Boltzmann} equation, 2026.
\newblock arXiv:2604.06919.

\bibitem[BCD13]{BudhirajaChenDupuis}
A.~Budhiraja, J.~Chen, and P.~Dupuis.
\newblock Large deviations for stochastic partial differential equations driven
  by a {Poisson} random measure.
\newblock {\em Stochastic Process. Appl.}, 123:523--560, 2013.

\bibitem[BD00]{BudhirajaDupuisBrownian}
A.~Budhiraja and P.~Dupuis.
\newblock A variational representation for positive functionals of infinite
  dimensional {Brownian} motion.
\newblock {\em Probab. Math. Statist.}, 20:39--61, 2000.

\bibitem[BDM11]{BudhirajaDupuisMaroulas}
A.~Budhiraja, P.~Dupuis, and V.~Maroulas.
\newblock Variational representations for continuous time processes.
\newblock {\em Ann. Inst. H. Poincar\'e Probab. Statist.}, 47:725--747, 2011.

\bibitem[BDSG{\etalchar{+}}15]{BertiniEtAlMFT}
L.~Bertini, A.~De~Sole, D.~Gabrielli, G.~Jona-Lasinio, and C.~Landim.
\newblock Macroscopic fluctuation theory.
\newblock {\em Rev. Modern Phys.}, 87:593--636, 2015.

\bibitem[BGL93]{BardosGolseLevermoreFluidLimits}
C.~Bardos, F.~Golse, and C.~D. Levermore.
\newblock Fluid dynamic limits of kinetic equations. {II}. convergence proofs
  for the {Boltzmann} equation.
\newblock {\em Comm. Pure Appl. Math.}, 46:667--753, 1993.

\bibitem[BGSRS23]{BodineauGallagherSaintRaymondSimonella}
T.~Bodineau, I.~Gallagher, L.~Saint-Raymond, and S.~Simonella.
\newblock Statistical dynamics of a hard sphere gas: fluctuating {Boltzmann}
  equation and large deviations.
\newblock {\em Ann. of Math.}, 198:1047--1201, 2023.

\bibitem[Bol72]{Boltzmann1872}
L.~Boltzmann.
\newblock Weitere studien \"uber das w\"armegleichgewicht unter gasmolek\"ulen.
\newblock {\em Sitzungsber. Kais. Akad. Wiss. Wien Math.-Naturwiss. Classe},
  66:275--370, 1872.

\bibitem[Bou20]{Bouchet}
F.~Bouchet.
\newblock Is the {Boltzmann} equation reversible? a large deviation perspective
  on the irreversibility paradox.
\newblock {\em J. Stat. Phys.}, 181:515--550, 2020.

\bibitem[BPZ23]{BrzezniakPengZhaiNavierStokes}
Z.~Brze\'zniak, X.~Peng, and J.~Zhai.
\newblock Well-posedness and large deviations for {2D} stochastic
  {Navier--Stokes} equations with jumps.
\newblock {\em J. Eur. Math. Soc.}, 25:3093--3176, 2023.

\bibitem[CF23]{CornalbaFischerStructure}
F.~Cornalba and J.~Fischer.
\newblock The {Dean--Kawasaki} equation and the structure of density
  fluctuations in systems of diffusing particles.
\newblock {\em Arch. Ration. Mech. Anal.}, 247:Paper No. 76, 2023.

\bibitem[CIP94]{CercignaniIllnerPulvirenti}
C.~Cercignani, R.~Illner, and M.~Pulvirenti.
\newblock {\em The Mathematical Theory of Dilute Gases}, volume 106 of {\em
  Applied Mathematical Sciences}.
\newblock Springer, New York, 1994.

\bibitem[CSZ19]{CornalbaShardlowZimmer}
F.~Cornalba, T.~Shardlow, and J.~Zimmer.
\newblock A regularized {Dean--Kawasaki} model: derivation and analysis.
\newblock {\em SIAM J. Math. Anal.}, 51:1137--1187, 2019.

\bibitem[DE97]{DupuisEllis}
P.~Dupuis and R.~S. Ellis.
\newblock {\em A Weak Convergence Approach to the Theory of Large Deviations}.
\newblock Wiley Series in Probability and Statistics. Wiley, New York, 1997.

\bibitem[Dea96]{Dean1996}
D.~S. Dean.
\newblock Langevin equation for the density of a system of interacting
  {Langevin} processes.
\newblock {\em J. Phys. A}, 29:L613--L617, 1996.

\bibitem[DFG26]{DirrFehrmanGessSSEP}
N.~Dirr, B.~Fehrman, and B.~Gess.
\newblock Conservative stochastic {PDE} and fluctuations of the symmetric
  simple exclusion process.
\newblock {\em Comm. Math. Phys.}, 407:Paper No. 74, 2026.

\bibitem[DGGG21]{DareiotisGessGnannGruen}
K.~Dareiotis, B.~Gess, M.~V. Gnann, and G.~Gr\"un.
\newblock Non-negative martingale solutions to the stochastic thin-film
  equation with nonlinear gradient noise.
\newblock {\em Arch. Ration. Mech. Anal.}, 242:179--234, 2021.

\bibitem[DHM24]{DengHaniMaLongTimeBoltzmann}
Y.~Deng, Z.~Hani, and X.~Ma.
\newblock Long time derivation of the {Boltzmann} equation from hard sphere
  dynamics.
\newblock Ann. of Math., to appear; arXiv:2408.07818, 2024.

\bibitem[DHW26]{DuongHeWuFluctuatingLandau}
M.~H. Duong, Z.~He, and Z.~Wu.
\newblock The homogeneous {Landau} equation with regularised thermal noise,
  2026.
\newblock arXiv:2607.22329.

\bibitem[DJP26]{DjurdjevacJiPerkowski}
A.~Djurdjevac, X.~Ji, and N.~Perkowski.
\newblock Weak error of {Dean--Kawasaki} equation with smooth mean-field
  interactions.
\newblock {\em J. Math. Pures Appl.}, 213:Paper No. 103937, 2026.

\bibitem[DKP24]{DjurdjevacKrempPerkowski}
A.~Djurdjevac, H.~Kremp, and N.~Perkowski.
\newblock Weak error analysis for a nonlinear {SPDE} approximation of the
  {Dean--Kawasaki} equation.
\newblock {\em Stoch. Partial Differ. Equ. Anal. Comput.}, 12:2330--2355, 2024.

\bibitem[DL89]{DiPernaLionsBoltzmann}
R.~J. DiPerna and P.-L. Lions.
\newblock On the {Cauchy} problem for {Boltzmann} equations: global existence
  and weak stability.
\newblock {\em Ann. of Math.}, 130:321--366, 1989.

\bibitem[DV05]{DesvillettesVillaniEquilibrium}
L.~Desvillettes and C.~Villani.
\newblock On the trend to global equilibrium for spatially inhomogeneous
  kinetic systems: the {Boltzmann} equation.
\newblock {\em Invent. Math.}, 159:245--316, 2005.

\bibitem[EK86]{EthierKurtz}
S.~N. Ethier and T.~G. Kurtz.
\newblock {\em Markov Processes: Characterization and Convergence}.
\newblock Wiley Series in Probability and Mathematical Statistics. John Wiley
  \& Sons, New York, 1986.

\bibitem[FG18]{FischerGruenThinFilm}
J.~Fischer and G.~Gr\"un.
\newblock Existence of positive solutions to stochastic thin-film equations.
\newblock {\em SIAM J. Math. Anal.}, 50:411--455, 2018.

\bibitem[FG23]{FehrmanGessNonEquilibrium}
B.~Fehrman and B.~Gess.
\newblock Non-equilibrium large deviations and parabolic-hyperbolic {PDE} with
  irregular drift.
\newblock {\em Invent. Math.}, 234:573--636, 2023.

\bibitem[FG24]{FehrmanGessDeanKawasaki}
B.~Fehrman and B.~Gess.
\newblock Well-posedness of the {Dean--Kawasaki} and the nonlinear
  {Dawson--Watanabe} equation with correlated noise.
\newblock {\em Arch. Ration. Mech. Anal.}, 248:Paper No. 20, 2024.

\bibitem[GG20]{GessGnannThinFilm}
B.~Gess and M.~V. Gnann.
\newblock The stochastic thin-film equation: existence of nonnegative
  martingale solutions.
\newblock {\em Stochastic Process. Appl.}, 130:7260--7302, 2020.

\bibitem[GHW23]{GessHeydeckerWuLLNS}
B.~Gess, D.~Heydecker, and Z.~Wu.
\newblock {Landau--Lifshitz--Navier--Stokes} equations: large deviations and
  relationship to the energy equality, 2023.
\newblock arXiv:2311.02223.

\bibitem[GMR06]{GrunMeckeRauscher}
G.~Gr\"un, K.~Mecke, and M.~Rauscher.
\newblock Thin-film flow influenced by thermal noise.
\newblock {\em J. Stat. Phys.}, 122:1261--1291, 2006.

\bibitem[GSR04]{GolseSaintRaymondNavierStokesLimit}
F.~Golse and L.~Saint-Raymond.
\newblock The {Navier--Stokes} limit of the {Boltzmann} equation for bounded
  collision kernels.
\newblock {\em Invent. Math.}, 155:81--161, 2004.

\bibitem[GSR09]{GolseSaintRaymondHardCutoff}
F.~Golse and L.~Saint-Raymond.
\newblock The incompressible {Navier--Stokes} limit of the {Boltzmann} equation
  for hard cutoff potentials.
\newblock {\em J. Math. Pures Appl.}, 91:508--552, 2009.

\bibitem[GSRT14]{GallagherSaintRaymondTexier}
I.~Gallagher, L.~Saint-Raymond, and B.~Texier.
\newblock {\em From {Newton} to {Boltzmann}: Hard Spheres and Short-Range
  Potentials}.
\newblock Z\"urich Lectures in Advanced Mathematics. European Mathematical
  Society, Z\"urich, 2014.

\bibitem[GSW26]{GessSauerbreyWuThermalNSF}
B.~Gess, M.~Sauerbrey, and Z.~Wu.
\newblock The incompressible {Navier--Stokes--Fourier} system with thermal
  noise, 2026.
\newblock arXiv:2603.26307.

\bibitem[Guo04]{GuoWholeSpace}
Y.~Guo.
\newblock The {Boltzmann} equation in the whole space.
\newblock {\em Indiana Univ. Math. J.}, 53:1081--1094, 2004.

\bibitem[Guo10]{GuoBoundedDomains}
Y.~Guo.
\newblock Decay and continuity of the {Boltzmann} equation in bounded domains.
\newblock {\em Arch. Ration. Mech. Anal.}, 197:713--809, 2010.

\bibitem[GWZ26]{GessWuZhangHigherOrder}
B.~Gess, Z.~Wu, and R.~Zhang.
\newblock Higher order fluctuation expansions for nonlinear stochastic heat
  equations in singular limits.
\newblock {\em Stochastic Process. Appl.}, 193:Paper No. 104847, 2026.

\bibitem[Hey23]{HeydeckerKacLDP}
D.~Heydecker.
\newblock Large deviations of {Kac}'s conservative particle system and energy
  nonconserving solutions to the {Boltzmann} equation: a counterexample to the
  predicted rate function.
\newblock {\em Ann. Appl. Probab.}, 33:1758--1826, 2023.

\bibitem[HWZ25]{HaoWuZimmerVFPDK}
Z.~Hao, Z.~Wu, and J.~Zimmer.
\newblock Kinetic theory with fluctuations: well-posedness of the
  {Vlasov--Fokker--Planck--Dean--Kawasaki} equations, 2025.
\newblock arXiv:2511.10194.

\bibitem[Kac56]{Kac1956}
M.~Kac.
\newblock Foundations of kinetic theory.
\newblock In {\em Proceedings of the Third Berkeley Symposium on Mathematical
  Statistics and Probability}, volume~3, pages 171--197. University of
  California Press, Berkeley, 1956.

\bibitem[KLvR19]{KonarovskyiLehmannRenesse}
V.~Konarovskyi, T.~Lehmann, and M.-K. von Renesse.
\newblock {Dean--Kawasaki} dynamics: ill-posedness vs. triviality.
\newblock {\em Electron. Commun. Probab.}, 24:Paper No. 8, 9 pp., 2019.

\bibitem[Lan75]{Lanford1975}
III Lanford, O.~E.
\newblock Time evolution of large classical systems.
\newblock In J.~Moser, editor, {\em Dynamical Systems, Theory and
  Applications}, volume~38 of {\em Lecture Notes in Physics}, pages 1--111.
  Springer, Berlin, 1975.

\bibitem[LL87]{LandauLifshitzFluidMechanics}
L.~D. Landau and E.~M. Lifshitz.
\newblock {\em Fluid Mechanics}, volume~6 of {\em Course of Theoretical
  Physics}.
\newblock Pergamon Press, Oxford, 2 edition, 1987.

\bibitem[LM01]{LionsMasmoudiFluidMechanicsII}
P.-L. Lions and N.~Masmoudi.
\newblock From the {Boltzmann} equations to the equations of incompressible
  fluid mechanics. {II}.
\newblock {\em Arch. Ration. Mech. Anal.}, 158:195--211, 2001.

\bibitem[LM10]{LevermoreMasmoudiNSF}
C.~D. Levermore and N.~Masmoudi.
\newblock From the {Boltzmann} equation to an incompressible
  {Navier--Stokes--Fourier} system.
\newblock {\em Arch. Ration. Mech. Anal.}, 196:753--809, 2010.

\bibitem[McK75]{McKeanFluctuations}
H.~P. McKean.
\newblock Fluctuations in the kinetic theory of gases.
\newblock {\em Comm. Pure Appl. Math.}, 28:435--455, 1975.

\bibitem[MG22]{MetzgerGruenThinFilm}
S.~Metzger and G.~Gr\"un.
\newblock Existence of nonnegative solutions to stochastic thin-film equations
  in two space dimensions.
\newblock {\em Interfaces Free Bound.}, 24:307--387, 2022.

\bibitem[MM13]{MischlerMouhotKac}
S.~Mischler and C.~Mouhot.
\newblock {Kac}'s program in kinetic theory.
\newblock {\em Invent. Math.}, 193:1--147, 2013.

\bibitem[MM25]{MartiniMayorcasKSDK}
A.~Martini and A.~Mayorcas.
\newblock An additive-noise approximation to {Keller--Segel--Dean--Kawasaki}
  dynamics: local well-posedness of paracontrolled solutions.
\newblock {\em Stoch. Partial Differ. Equ. Anal. Comput.}, 13:956--1033, 2025.

\bibitem[MN06]{MouhotNeumann}
C.~Mouhot and L.~Neumann.
\newblock Quantitative perturbative study of convergence to equilibrium for
  collisional kinetic models in the torus.
\newblock {\em Nonlinearity}, 19:969--998, 2006.

\bibitem[MvRZ25]{MullerVonRenesseZimmerVFPDK}
F.~M\"uller, M.~von Renesse, and J.~Zimmer.
\newblock Well-posedness for {Dean--Kawasaki} models of
  {Vlasov--Fokker--Planck} type.
\newblock {\em Proc. R. Soc. A}, 481:Paper No. 20250089, 2025.

\bibitem[Pop25]{PopatBoundedDomains}
S.~Popat.
\newblock Well-posedness of the generalised {Dean--Kawasaki} equation with
  correlated noise on bounded domains.
\newblock {\em Stochastic Process. Appl.}, 179:Paper No. 104503, 2025.

\bibitem[Rez98]{Rezakhanlou}
F.~Rezakhanlou.
\newblock Large deviations from a kinetic limit.
\newblock {\em Ann. Probab.}, 26:1259--1340, 1998.

\bibitem[Sau24]{SauerbreyThinFilm}
M.~Sauerbrey.
\newblock Martingale solutions to the stochastic thin-film equation in two
  dimensions.
\newblock {\em Ann. Inst. Henri Poincar\'e Probab. Stat.}, 60:373--412, 2024.

\bibitem[Uka74]{UkaiGlobalBoltzmann}
S.~Ukai.
\newblock On the existence of global solutions of mixed problem for non-linear
  {Boltzmann} equation.
\newblock {\em Proc. Japan Acad.}, 50:179--184, 1974.

\bibitem[Vil02]{VillaniCollisionalReview}
C.~Villani.
\newblock A review of mathematical topics in collisional kinetic theory.
\newblock In {\em Handbook of Mathematical Fluid Dynamics}, volume~1, pages
  71--305. North-Holland, Amsterdam, 2002.

\bibitem[Wu25]{WuNonlinearFluctuations}
Z.~Wu.
\newblock Fluctuating hydrodynamics of the {Ising--Kac--Kawasaki} model and
  nonlinear fluctuations near criticality, 2025.
\newblock arXiv:2511.15263.

\bibitem[WWZ26]{WangWuZhangSingular}
L.~Wang, Z.~Wu, and R.~Zhang.
\newblock Well-posedness of the {Dean--Kawasaki} equation with singular
  interactions.
\newblock {\em SIAM J. Math. Anal.}, 58:2738--2785, 2026.

\bibitem[WZ24]{WuZhaiPorousMedia}
W.~Wu and J.~Zhai.
\newblock Large deviations for stochastic generalized porous media equations
  driven by {L\'evy} noise.
\newblock {\em SIAM J. Math. Anal.}, 56:1--42, 2024.

\bibitem[XZ18]{XiongZhaiLocallyMonotone}
J.~Xiong and J.~Zhai.
\newblock Large deviations for locally monotone stochastic partial differential
  equations driven by {L\'evy} noise.
\newblock {\em Bernoulli}, 24:2842--2874, 2018.

\bibitem[YZZ15]{YangZhaiZhangJumpSPDE}
X.~Yang, J.~Zhai, and T.~Zhang.
\newblock Large deviations for {SPDEs} of jump type.
\newblock {\em Stoch. Dyn.}, 15:Paper No. 1550026, 2015.

\bibitem[ZZ15]{ZhaiZhangNavierStokesLevy}
J.~Zhai and T.~Zhang.
\newblock Large deviations for 2-d stochastic {Navier--Stokes} equations driven
  by multiplicative {L\'evy} noises.
\newblock {\em Bernoulli}, 21:2351--2392, 2015.

\end{thebibliography}

\end{document}